\documentclass[12pt]{article}

\usepackage
{
makeidx,epsfig,amsmath,amssymb,amsthm,amsfonts,times,latexsym,mathptmx,algorithm,algorithmic,
wrapfig,bbm,graphicx,color,enumerate,endnotes
}
\usepackage[letterpaper,twoside,outer=1.3in,vmargin=1.3in,]{geometry}

\newtheorem{theorem}{Theorem}

\newtheorem{LE}[theorem]{Lemma}
\newtheorem{PR}[theorem]{Proposition}

\newtheorem{RE}[theorem]{Remark}

\newcounter{claim_nb}[theorem]
\newtheorem{claim}[claim_nb]{Claim}
\newtheorem*{claim*}{Claim}

\newcommand{\zB}{\mathcal B}

\newcommand{\zI}{\mathcal I}
\newcommand{\bfB}{\mathbf B}
\newcommand{\bS}{\mathbb S}
\newcommand{\bF}{\mathbb F}
\newcommand{\bL}{\mathbb L}
\newcommand{\dst}{\mathbb T}

\DeclareMathOperator{\cycle}{cycle}
\DeclareMathOperator{\ecycle}{ecycle}
\DeclareMathOperator{\cut}{cut}
\DeclareMathOperator{\ecut}{ecut}
\DeclareMathOperator{\spa}{span}
\DeclareMathOperator{\Wflip}{W_{\mbox{\tiny flip}}}
\DeclareMathOperator{\loops}{loop}

\newcommand{\tr}{\triangle}
\newcommand{\odd}{V_{odd}}
 
\newenvironment{cproof}
{\begin{proof}
 [Proof.]
 \vspace{-1.2\parsep}}
{ \end{proof}}

\title{Isomorphism for even cycle matroids - I}

\author{
Bertrand Guenin\\
Dept. of Combinatorics and Optimization\\
University of Waterloo\\
200 University Avenue\\
Waterloo, ON, Canada
\and
Irene Pivotto
\thanks{email:\texttt{ipivotto@sfu.ca}; phone:$(+1) 778\ 782\ 5754$}\\
Dept. of Mathematics\\
Simon Fraser University\\
8888 University Drive\\
Burnaby, BC, Canada
\and
Paul Wollan\\
Dept. of Computer Science\\
University of Rome, {\it La Sapienza}\\
Via Salaria, 113\\
Rome, Italy\\
}
\begin{document}
\maketitle

\begin{abstract}
A seminal result by Whitney describes when two graphs have the same cycles.
We consider the analogous problem for even cycle matroids.
A representation of an even cycle matroid is a pair formed by a graph together with a special set of edges of the graph. 
Such a pair is called a signed graph.
We consider the problem of  determining the relation between two signed graphs representing the same even cycle matroid.
We refer to this problem as the Isomorphism Problem for even cycle matroids.
We present two classes of signed graphs and we solve the Isomorphism Problem for these two classes.
We conjecture that, up to simple operations, any two signed graphs representing the same even cycle matroid are either in one of these classes,
or related by a modification of an operation for graphic matroids, or belonging to a small set of examples.
\end{abstract}

\section{Introduction} \label{sec:intro}
We assume that the reader is familiar with the basics of matroid theory. 
See Oxley~\cite{Oxley92} for the definition of the terms used here.
We will only consider binary matroids in this paper. Thus the reader should substitute
the term ``binary matroid" every time ``matroid" appears in this text.

Throughout this paper, we will consider graphs with multiple edges and loops.  Let $G$ be a graph. For a set $X\subseteq E(G)$, we write $V_G(X)$ to refer to the set of vertices incident with an edge of $X$ 
and $G[X]$ for the subgraph with vertex set $V_G(X)$ and edge set $X$.
A subset $C$ of edges of $G$ is a {\em cycle} if $G[C]$ is a graph where every vertex has even degree. 
An inclusion-wise minimal non-empty cycle is a {\em circuit}. 
We denote by $\cycle(G)$ the set of all cycles of $G$. 
A {\em cycle} in a binary matroid $M$ is the symmetric difference of circuits of $M$. 
Since the cycles of $G$ correspond to the cycles of the {\em cycle matroid} of $G$, 
we identify $\cycle(G)$ with that matroid and say that $G$ is a {\em representation} of that matroid. 
Cycle matroids are also referred to as graphic matroids.
The classes of matroids considered in this work all arise from graphs. 
Hence, when referring to a representation of a matroid we will always mean a graphic representation of the matroid.
When referring to a matrix representing a matroid over some field, 
we will refer to that matrix as the {\em matrix representation} of the matroid.

We may ask when two graphs represent the same cycle matroid. 
We define an operation on graphs which preserves cycles as follows. 
Given sets $A$ and $B$ we denote by $A-B$ the set $\{a\in A:a\notin B\}$.
Given a set of edges $X$ of $G$, we define the {\em boundary} of $X$ in $G$ as $\zB_G(X)=V_G(X)\cap V_G(\bar{X})$, where $\bar{X}=E(G)-X$
(we will always denote by $\bar{X}$ the complement of a set $X$).
Consider a graph $G$ and let $X\subseteq E(G)$.
Suppose that $\zB_G(X)=\{u_1, u_2\}$ for some $u_1,u_2\in V(G)$.
Let $G'$ be the graph obtained by identifying vertices $u_1,u_2$ of $G[X]$ 
with vertices $u_2,u_1$ of $G[\bar{X}]$ respectively. 
Then $G'$ is obtained from $G$ by a {\em Whitney-flip} on $X$.
We will also call Whitney-flip the operation consisting of identifying
two vertices from distinct components, or the operation 
consisting of partitioning the graph into components each of which is a block of $G$.

It is easy to see that two graphs related by a sequence of Whitney-flips have the same cycles; 
in particular, they are representations of the same cycle matroid.
In a seminal paper~\cite{Whitney33}, Whitney proved that the converse also holds.
\begin{theorem}[Whitney \cite{Whitney33}]\label{intro:whitney}
Two graphs represent the same cycle matroid if and only if they are related by a sequence of Whitney-flips.
\end{theorem}
In light of Theorem~\ref{intro:whitney}, we define two graphs to be {\em equivalent} if one can be obtained from the other by a sequence of Whitney-flips.

Given a set of vertices $U$, we denote by $\delta_G(U)$ the {\em cut} induced by $U$,
that is $\delta_G(U):=\{(u,v) \in E(G): u \in U, v \not \in U\}$.
We denote by $\cut(G)$ the set of all cuts of $G$. 
Since the cuts of $G$ correspond to the cycles of the {\em cut matroid} of $G$, 
we identify $\cut(G)$ with that matroid and say that $G$ is a {\em representation} of that matroid.
Cut matroids are also referred to as co-graphic matroids, as they are duals of graphic matroids.
Theorem~\ref{intro:whitney} may be restated as follows.
\begin{theorem}[Whitney \cite{Whitney33}]\label{intro:whitneyCut}
Two graphs represent the same cut matroid if and only if they are related by Whitney-flips.
\end{theorem}
Theorem~\ref{intro:whitney} and Theorem~\ref{intro:whitneyCut} provide solutions to the problem of determining when two graphs represent the 
same graphic or co-graphic matroid. We refer to this problem as the {\em Isomorphism Problem}.
In this paper we study the Isomorphism Problem for the class of even cycle matroids.

A {\em signed graph} is a pair $(G,\Sigma)$ where $G$ is a graph and $\Sigma\subseteq E(G)$.
We call $\Sigma$ a {\em signature} of $G$.
A subset $D\subseteq E(G)$ is {\em $\Sigma$-even} if $|D\cap\Sigma|$ is even (and {\em $\Sigma$-odd} otherwise). 
When there is no ambiguity we omit the prefix 
$\Sigma$ when referring to $\Sigma$-even and $\Sigma$-odd sets. 
Given a signed graph $(G,\Sigma)$, we denote by $\ecycle(G,\Sigma)$ the set of all even cycles of $(G,\Sigma)$.
It can be verified that $\ecycle(G,\Sigma)$ is the set of cycles of a matroid which we call the {\em even cycle matroid}.
We identify $\ecycle(G,\Sigma)$ with that matroid and say that $(G,\Sigma)$ is a 
{\em representation} of that matroid.
Note that, if $\Sigma$ is empty, all the cycles of $(G,\Sigma)$ are even, hence $\ecycle(G,\Sigma)$ is a cycle matroid.
Hence the class of even cycle matroids contains the class of cycle matroids.

\noindent
\textbf{Isomorphism Problem for even cycles:} What is the relation between two representations of the same even cycle matroid?

The Isomorphism Problem has been solved for even cycle matroids which are graphic,
by Shih (in his doctoral disseration, see \cite{Shih82}) and independently by 
Gerards, Lov\'asz, Schrijver, Seymour, Truemper (see \cite{Gerards00}).
We report the second result here, while Shih's result, which describes the structure of the graphs more precisely, 
is presented in Section~\ref{sec:statements}.
\begin{theorem}
Let $(G,\Sigma)$ and $(G',\Sigma')$ be signed graphs.
Suppose that $\ecycle(G,\Sigma)=\ecycle(G',\Sigma')$ and that this matroid is a cycle matroid.
Then $(G,\Sigma)$ and $(G',\Sigma')$ are related by 
a sequence of Whitney-flips, signature exchanges, and Lov\'asz-flips.
\end{theorem}
\noindent
We need to define the terms ``signature exchange" and  ``Lov\'asz-flip".
Given a signed graph  $(G,\Sigma)$, we say that $\Sigma'$
is obtained from $\Sigma$ by a {\em signature exchange}
if $\Sigma \tr \Sigma'$ is a cut of $G$ (where $\tr$ denotes symmetric difference).
Every set $\Sigma'$ which may be obtained from $\Sigma$ by a signature exchange is a {\em signature} of $(G,\Sigma)$.
It is easy to show that $\ecycle(G,\Sigma)=\ecycle(G,\Sigma')$ if and only if $\Sigma'$ is a signature of $(G,\Sigma)$.

Given a graph $G$ we denote by $\loops(G)$ the set of all loops of $G$.
Let $(G,\Sigma)$ be a signed graph.
A vertex $s$ is a {\em blocking vertex} of $(G,\Sigma)$ if every odd circuit of $(G,\Sigma)$ either is a loop or uses $s$.
A pair of vertices $s,t$ is a {\em blocking pair} if every odd circuit of $(G,\Sigma)$ is either a loop or  
uses at least one of $s,t$.
Note that $s$ is a blocking vertex (respectively $s,t$ is a blocking pair) of $(G,\Sigma)$ if and only if 
there exists a signature $\Sigma'$ of $(G,\Sigma)$ such that $\Sigma'\subseteq\delta(s) \cup \loops(G)$ 
(respectively $\Sigma'\subseteq\delta(s)\cup\delta(t)\cup \loops(G)$).

Consider a signed graph $(G,\Sigma)$ and vertices $v_1,v_2 \in V(G)$, where
$\Sigma\subseteq\delta_G(v_1)\cup\delta_G(v_2)\cup \loops(G)$.
So $v_1,v_2$ is a blocking pair of $(G,\Sigma)$.
We can construct a signed graph $(G',\Sigma)$ from $(G,\Sigma)$ by 
replacing endpoints the $x,y$ of every odd edge $e$ with new endpoints $x',y'$ as follows:
\begin{enumerate}[\;\;\;(a)]
	\item if $x=v_1$ and $y=v_2$ then $x'=y'$ (i.e. $e$ becomes a loop);
	\item if $x=y$ (i.e. $e$ is a loop), then $x'=v_1$ and $y'=v_2$;
	\item if $x=v_1$ and $y\neq v_1,v_2$, then $x'=v_2$ and $y'=y$;
	\item if $x=v_2$ and $y\neq v_1,v_2$, then $x'=v_1$ and $y'=y$.
\end{enumerate}
Then we say that $(G',\Sigma)$ is obtained from $(G,\Sigma)$ by a {\em Lov\'asz-flip} on $v_1,v_2$.
It is easy to show that Lov\'asz-flips preserve even cycles.

Suppose that $(G_1,\Sigma_1)$ and $(G_2, \Sigma_2)$ are signed graphs where
$G_1$ and $G_2$ are equivalent and $\Sigma_2$ is obtained from $\Sigma_1$ by a signature exchange.
Then we say that $(G_1,\Sigma_1)$ and $(G_2, \Sigma_2)$ are {\em equivalent} signed graphs.
It is easy to see that, if $G_1$ and $G_2$ are equivalent graphs and $\ecycle(G_1,\Sigma_1)=\ecycle(G_2,\Sigma_2)$
for some signatures $\Sigma_1$ and $\Sigma_2$, then $(G_1,\Sigma_1)$ and $(G_2,\Sigma_2)$ are equivalent.
Thus the Isomorphism Problem is easily solved for signed graphs having equivalent underlying graphs. 
Therefore we focus on the Isomorphism Problem for the case that the two graphs are inequivalent.
We say that two graphs $G_1$ and $G_2$ are {\em siblings} if $G_1$ and $G_2$ are inequivalent and, for some signatures $\Sigma_1$ and $\Sigma_2$,
we have $\ecycle(G_1,\Sigma_1)=\ecycle(G_2,\Sigma_2)$.
We extend this terminology to the signed graphs and say that $(G_1,\Sigma_1)$ and $(G_2,\Sigma_2)$ are siblings.
We call the pair $\Sigma_1,\Sigma_2$ the {\em matching signature pair} for $G_1$, $G_2$.
In~\cite{GPW241} we proved that for any pair of siblings the matching signature pair is unique up
to signature exchange.

In Section~\ref{sec:preliminaries} we define another class of binary matroids,
the class of even cut matroids, which is a generalization of the class of co-graphic matroids. 
A result proved in~\cite{GPW241} shows the relation between the Isomorphism Problem for even cycle and even cut matroids.
We report such result in Section~\ref{sec:preliminaries} (which also contains preliminary results about even cycle matroids).

We thus focus on the Isomorphism Problem for even cycles: in Section~\ref{sec:statements} 
we present two classes of siblings and we characterize all the operations relating two siblings in the same class,
thus solving the Isomorphism Problem for these classes.
We conjecture that, up to Whitney-flips, signature exchanges, Lov\'asz-flips and some reductions, every pair of
siblings is either contained in one of these two classes, or is a modification of an operation
for graphic matroids, or forms a sporadic example.

Section~\ref{sec:whitney} contains results about Whitney-flips in graphs, which are used in Sections~\ref{sec:proofs1} and~\ref{sec:proofs2}.
The last two sections contain the proofs of the results stated in Section~\ref{sec:statements}.
%

\section{Preliminaries} \label{sec:preliminaries}
In this section we present some basic properties of even cycle matroids.
In particular, we specify what the bases and co-cycles are and present some simple results about connectivity. 
We introduce the class of even cut matroids and show the relation between
pairs of representations of even cycle matroids and pairs of representations of even cut matroids.
Finally, we introduce an operation which relates representations of some matroids which are both
even cycle matroids and duals of even cut matroids.
\subsection{Bases and co-cycles}
Consider a signed graph $(G,\Sigma)$.
A set $F \subseteq E(G)$ is dependent in $\ecycle(G,\Sigma)$ if and only if it contains an even cycle. 
As we consider graphs up to equivalence and identifying two vertices in distinct components of a graph is a Whitney-flip,
we may assume without loss of generality that $G$ is connected.
If $(G,\Sigma)$ does not contain any odd cycles, then $\ecycle(G,\Sigma)=\cycle(G)$ and a basis for $\ecycle(G,\Sigma)$ 
is just formed by a spanning tree of $G$.
If $(G,\Sigma)$ contains at least one odd cycle, every basis for $\ecycle(G,\Sigma)$ 
is formed by a spanning tree $B$ together with an edge $f \in \bar{B}$ forming an odd cycle with edges in $B$.

The co-cycles of $\ecycle(G,\Sigma)$ are the subsets of $E(G)$ which intersect every even cycle with even parity.
Hence we have the following.
\begin{RE}\label{co-cycleEcycle}
The co-cycles of $\ecycle(G,\Sigma)$ are the cuts of $G$ and the signatures of $(G,\Sigma)$.
\end{RE}
%
\subsection{Connectivity}\label{sec:connectivity}
Let $M$ be a matroid with rank function $r$.
Given $X\subseteq E(M)$ we define $\lambda_M(X)$, the {\em connectivity function} of $M$,
to be equal to $r(X)+r(\bar{X})-r(E(M))+1$. 
The set $X$ is a {\em $k$-separation} of $M$ if $\min\{|X|,|\bar{X}|\}\geq k$ and $\lambda_M(X)=k$.
$M$ is $k$-connected if it has no $r$-separations for any $r<k$.
Let $G$ be a graph and let $X\subseteq E(G)$.
The set $X$ is a {\em $k$-separation} of $G$ if $\min\{|X|,|\bar{X}|\}\geq k$, $|\zB_G(X)|=k$ and both $G[X]$ and $G[\bar{X}]$ are connected.
Note that with this definition two parallel edges of $G$ form a $2$-separation of $G$.
A graph $G$ is $k$-connected if it has no $r$-separations for any $r<k$.
We relate graph connectivity with connectivity for even cycle matroids.
Recall that we denote by $\loops(G)$ the set of loops of $G$.
A signed graph $(G,\Sigma)$ is {\em bipartite} if $G$ has no $\Sigma$-odd cycle.
Equivalently, $(G,\Sigma)$ is bipartite if $\Sigma$ is a cut of $G$.
We will make repeated use of the following result.
\begin{PR}\label{connectivity:n3c}
Suppose that $\ecycle(G,\Sigma)$ is $3$-connected. Then:
\begin{enumerate}[\;\;\;(1)]
	\item $|\loops(G)|\leq 1$ and if $e\in\loops(G)$ then $e\in\Sigma$;
	\item $G\setminus\loops(G)$ is $2$-connected;
	\item if $G$ has a $2$-separation $X$, then $(G[X],\Sigma\cap X)$ and $(G[\bar{X}],\Sigma\cap\bar{X})$ are both non-bipartite.
\end{enumerate}
\end{PR}

To prove Proposition~\ref{connectivity:n3c}, we require a definition and a preliminary result.
Let $(G,\Sigma)$ be a signed graph and $X\subseteq E(G)$. 
We say that $X$ is a {\em $k$-$(i,j)$-separation} of $(G,\Sigma)$, where $i,j\in\{0,1\}$,~if the following hold:
\begin{enumerate}[\;\;\;(a)]
	\item $X$ is a $k$-separation of $G$;
	\item $i=0$ when $(G[X],\Sigma\cap X)$ is bipartite and $i=1$ otherwise;
	\item $j=0$ when $(G[\bar{X}],\Sigma\cap\bar{X})$ is bipartite and $j=1$ otherwise.
\end{enumerate}
\begin{LE} \label{prelim:conn}
Let $(G,\Sigma)$ be a non-bipartite signed graph and $M_S:= \ecycle(G,\Sigma)$.
For every $k$-$(i,j)$-separation $X$ of $(G,\Sigma)$, we have $\lambda_{M_S}(X)=k+i+j-1$.
\end{LE}
\begin{proof}
Let $r$ be the rank function of $M:=\cycle(G)$ and $r_S$ be the rank function of $M_S$.
As $(G,\Sigma)$ is non-bipartite, a basis for $M_S$ consists of a spanning tree $B$ of $G$ 
plus an edge $e \in \bar{B}$ that forms a $\Sigma$-odd circuit with elements in $B$.
Hence $r_S(M_S)=r(M)+1$.
Similarly, if $(G[X],\Sigma\cap X)$ (respectively $(G[\bar{X}],\Sigma\cap\bar{X})$) is non-bipartite, then 
the rank of $X$ (respectively $\bar{X}$) in $M_S$ is one more that in $M$, 
otherwise the rank of $X$ (respectively $\bar{X}$) is the same in both matroids.
Thus $r_S(X)=r(X)+i$ and $r_S(\bar{X})=r(\bar{X})+j$.
Hence
\begin{align}
\lambda_{M_S}(X) &
= r_S(X)+r_S(\bar{X})-r_S(M_S)+1\nonumber\\
& = r(X)+i + r(\bar{X})+j -r(M)-1+1\nonumber\\
& = \lambda_M(X) +i +j -1\nonumber\\
& = k +i +j -1\nonumber
\end{align}
\end{proof}
\begin{proof}[\textbf{Proof of Proposition~\ref{connectivity:n3c}}]
Let $M:=\ecycle(G,\Sigma)$.
As $M$ is $3$-connected, it has no loops, no co-loops and no parallel elements.
We may assume that $(G,\Sigma)$ is non-bipartite, for otherwise $M=\cycle(G)$ and $G$ is $3$-connected.
{\bf (1)}
Let $e$ be a loop of $G$. Then $e\in\Sigma$ for otherwise $e$ would be a loop of $M$.
There do not exist distinct loops $e,f$ of $G$, for otherwise $\{e,f\}$ would be a circuit of $M$ and $e,f$ would be in parallel in $M$.
{\bf (2)}
Suppose that $X$ is a $1$-$(i,j)$-separation of $(G,\Sigma)$.
By Lemma~\ref{prelim:conn}, $\lambda_{M}(X)=1+i+j-1 \leq 2$. 
As $M$ is $3$-connected, $X$ is not a $2$-separation; hence either $|X|=1$ or $|\bar{X}|=1$.
The single element in $X$ (or $\bar{X}$) is not a bridge of $G$, for otherwise it is a co-loop of $M$.
Hence $X$ or $\bar{X}$ is a loop of $G$.
{\bf (3)}
Suppose that $X$ is a $2$-$(i,j)$-separation of $(G,\Sigma)$.
As $M$ is $3$-connected, $\lambda_{M}(X) \geq 3$.
By Lemma~\ref{prelim:conn}, $2+i+j-1 \geq 3$, hence $i=j=1$.
\end{proof}

\subsection{Even cut matroids}
A {\em graft} is a pair $(G,T)$ where $G$ is a graph, $T \subseteq V(G)$ and $|T|$ is even. 
The vertices in $T$ are the {\em terminals} of the graft.
A cut $\delta(U)$ is {\em $T$-even} (respectively {\em $T$-odd}) 
if $|T \cap U|$ is even (respectively odd). 
When there is no ambiguity we omit the prefix $T$
when referring to $T$-even and $T$-odd cuts.
We denote by $\ecut(G,T)$ the set of all even cuts of $(G,T)$.
It can be verified that $\ecut(G,T)$ is the set of cycles of a binary matroid, which we call the {\em even cut matroid} represented by $(G,T)$.
We identify $\ecut(G,T)$ with that matroid and say that $(G,T)$ is a {\em representation} of that matroid.
Note that, if $T$ is empty, all the cuts of $(G,T)$ are even, hence $\ecut(G,T)$ is a cut matroid.

\noindent
\textbf{Isomorphism Problem for even cuts:} What is the relation between two representations of the same even cut matroid?

There is a close relation between the Isomorphism Problem for even cycle matroids and 
the Isomorphism Problem for even cut matroids, 
as the results of \cite{GPW241} discussed in the next section indicate.

Given a graph $H$, we denote by $\odd(H)$ the set of vertices of $H$ of odd degree.
Given a graft $(G,T)$ we say that $J \subseteq E(G)$ is a {\em $T$-join} of $G$ if $T=\odd(G[J])$.
Note that, if $J$ is a $T$-join of $G$, a cut $C$ of $G$ is $T$-even if and only if $|C \cap J|$ is even.
We say that two grafts $(G_1,T_1)$ and $(G_2,T_2)$ are {\em equivalent}  
if $G_1$ and $G_2$ are equivalent and a $T_1$-join of $G_1$ is a $T_2$-join of $G_2$.
Given equivalent graphs $G_1$ and $G_2$, it is easy to see that $\ecut(G_1,T_1)=\ecut(G_2,T_2)$, 
for two sets of terminals $T_1$ for $G_1$ and $T_2$ for $G_2$,
if and only if $(G_1,T_1)$ and $(G_2,T_2)$ are equivalent.
\begin{RE}\label{co-cycleEcut}
The co-cycles of $\ecut(G,T)$ are the cycles of $G$ and the $T$-joins of $(G,T)$.
\end{RE}
%
\subsection{Pairing Isomorphism Problems}\label{sec:241}
Recall that two graphs are equivalent if they have the same cycles.
We will make repeated use of the following result of~\cite{GPW241}.
\begin{theorem}\label{isopair-co}
Let $G_1$ and $G_2$ be inequivalent graphs. 
\begin{enumerate}[\;\;\;(1)]
	\item
		Suppose there exists a pair $\Sigma_1,\Sigma_2\subseteq E(G_1)$ such that $\ecycle(G_1,\Sigma_1)=\ecycle(G_2,\Sigma_2)$.
		For $i=1,2$, if $(G_i,\Sigma_i)$ is bipartite define $C_i:=\emptyset$; otherwise let $C_i$ be a $\Sigma_i$-odd cycle of $G_i$.
		Let $T_i:=\odd(G_i[C_{3-i}])$.
		Then $\ecut(G_1,T_1)=\ecut(G_2,T_2)$.
	\item
		Suppose there exists a pair $T_1\subseteq V(G_1)$ and $T_2\subseteq V(G_2)$ (where $|T_1|,|T_2|$ are even) such that $\ecut(G_1,T_1)$ $=$
		$\ecut(G_2,T_2)$.
		For $i=1,2$, if $T_i=\emptyset$ let $\Sigma_{3-i} = \emptyset$; otherwise let $t_i \in T_i$ and $\Sigma_{3-i}:=\delta_{G_i}(t_i)$.
		Then $\ecycle$ $(G_1,\Sigma_1)$ $=$ $\ecycle(G_2,\Sigma_2)$.
\end{enumerate}
Moreover, if they exist, the pairs $\Sigma_1$, $\Sigma_2$ and $T_1$, $T_2$ are unique.
\end{theorem}
Theorem~\ref{isopair-co} implies that, if $G_1$ and $G_2$ are siblings, then there exist $T_1$ and $T_2$
such that $\ecut(G_1,T_1)=\ecut(G_2,T_2)$.
In this case we also say that $(G_1,T_1)$ and $(G_2,T_2)$ are siblings
and that $T_1,T_2$ is the {\em matching terminal pair} for $G_1,G_2$.
\subsection{Folding and unfolding}
In this section we define an operation that relates signed graphs with blocking pairs to grafts with four terminals.
For our purpose the position of the loops is immaterial. 
Thus we will assume that all loops form distinct components of the graph.

Consider a graph $H$ with a vertex $v$ and $\alpha\subseteq \delta_H(v) \cup \loops(H)$.
We say that $G$ is obtained from $H$ by 
{\em splitting $v$} into $v_1,v_2$ according to $\alpha$ if 
$V(G)=V(H)-\{v\}\cup\{v_1,v_2\}$ and for every $e=(u,w) \in E(H)$:
\begin{enumerate}[\;\;\;(a)]
	\item if $e \not \in \alpha \cup \delta_H(v)$, then $e=(u,w)$ in $G$;
	\item if $e \in \loops(H) \cap \alpha$, then $e=(v_1,v_2)$ in $G$;
	\item if $e \in \delta_H(v) \cap \alpha$ and $w=v$ then $e=(u,v_1)$ in $G$;
	\item if $e \in \delta_H(v) - \alpha$ and $w=v$ then $e=(u,v_2)$ in $G$.
\end{enumerate}

Consider a signed graph $(H,\Gamma)$ where $\Gamma\subseteq\delta_H(s)\cup\delta_H(t)\cup \loops(H)$ 
for two distinct vertices $s$ and $t$ of $H$. 
Choose $\alpha,\beta \subseteq E(H)$, where $\alpha \Delta \beta = \Gamma$, 
$\alpha \subseteq \delta(s) \cup \loops(H)$,
$\beta \subseteq \delta(t) \cup \loops(H)$, and $\alpha \cap \beta \cap \loops(H) = \emptyset$.
Construct a graft $(G,T)$ as follows:
\begin{enumerate}[\;\;\;(a)]
\item split $s$ into $s_1,s_2$ according to $\alpha$;
\item split $t$ into $t_1,t_2$ according to $\beta$;
\item set $T=\{s_1,s_2,t_1,t_2\}$.
\end{enumerate}
\noindent
Then the graft $(G,T)$ is obtained by {\em unfolding} $(H,\Gamma)$ according to vertices $s,t$ and signature $\Gamma$
(or according to vertices $s,t$ and $\alpha$,$\beta$).
Note that the resulting graft $(G,T)$ depends on the choice of $\alpha, \beta$, not only on $\Gamma$. 
Finally, we say that $(H,\Gamma)$ is obtained by {\em folding} the graft $(G,T)$ with the pairing $s_1,s_2$ and $t_1,t_2$.
We denote by $M^*$ the dual of a matroid $M$.
\begin{RE}\label{iso:unfold}
Let $(H,\Gamma)$ be a signed graph with $\Gamma\subseteq\delta(s)\cup\delta(t)\cup \loops(H)$ 
and let $(G,T)$ be a graft obtained by unfolding $(H,\Gamma)$ according to $s,t$ and $\Gamma$. Then:
\begin{enumerate}[\;\;\;(1)]
	\item a set of edges is an even cycle of $(H,\Gamma)$ if and only if it is a cycle or a $T$-join of $G$;
	\item $\ecycle(H,\Gamma)=\ecut(G,T)^*$.
\end{enumerate}
\end{RE}
\begin{proof}
Suppose we choose $\alpha$ and $\beta$ as in the definition of unfolding.
Suppose that $C$ is an even cycle of $(H,\Gamma)$.
For every $v\in V(H)-\{s,t\}$, $|\delta_H(v) \cap C|=|\delta_G(v) \cap C|$, which is even.
For $i=1,2$, define 
$d(s,i):=|C\cap\delta_G(s_i)|$ and $d(t,i):=|C\cap\delta_G(t_i)|$.
Since $C$ is a cycle, $d(s,1)$ and $d(s,2)$ have the same parity and so do $d(t,1)$ and $d(t,2)$.
Note that $\alpha=\delta_G(s_1)$, $\beta=\delta_G(t_1)$ and $\Gamma=\alpha \Delta \beta$.
Thus, as $|C\cap\Gamma|$ is even, $d(s,1)$ and $d(t,1)$ have the same parity.
Thus $d(s,1)$, $d(s,2)$, $d(t,1)$ and $d(t,2)$ are either all even or all odd.
In the former case $C$ is a cycle of $G$, in the later case it is a $T$-join of $G$.
The converse is similar. 
Finally, $(2)$ follows from $(1)$ and Remark~\ref{co-cycleEcut}.
\end{proof}
%
\section{Even cycle isomorphism}\label{sec:statements}
In this section we provide a partial answer to the Isomorphism Problem for even cycle matroids.
First we present a result by Shih that solves the Isomorphism Problem for even cycle matroids which are graphic.
We show, as a direct consequence of the results in Section~\ref{sec:241}, that this also solves the Isomorphism Problem
for even cut matroids which are co-graphic.
In Sections~\ref{sec:Shih} and~\ref{sec:quad} we introduce two classes of even cycle siblings: Shih siblings and quad siblings. 
For each one of these classes we provide a list of operations and we show that any two siblings in the class are related
by Whitney-flips and exactly one of these operations, thus solving the Isomorphism Problem for these two classes.

We believe that solving the Isomorphism Problem for Shih siblings and quad siblings
is a relevant step toward the solution of the Isomorphism Problem for $3$-connected even cycle matroids.
Suppose $(G_1,\Sigma_1)$ and $(G_2,\Sigma_2)$ are siblings such that $\ecycle(G_1,\Sigma_1)$ is $3$-connected.
We conjecture that there exist signed graphs $(G_1',\Sigma_1')$ and $(G_2',\Sigma_2')$ such that, for $i=1,2$, 
$(G_i',\Sigma_i')$ is obtained from $(G_i,\Sigma_i)$ by a sequence of Whitney-flips, Lov\'asz-flips and signature exchanges
and one of the following occurs:
\begin{enumerate}[\;\;\;(1)]
	\item $(G_1',\Sigma_1')=(G_2',\Sigma_2')$;
	\item $(G_1',\Sigma_1')$ and $(G_2',\Sigma_2')$ are either Shih siblings or quad siblings;
	\item $(G_1',\Sigma_1')$ and $(G_2',\Sigma_2')$ may be reduced;
	\item $(G_1',\Sigma_1')$ and $(G_2',\Sigma_2')$ belong to a sporadic set of examples;
	\item $(G_1',\Sigma_1')$ and $(G_2',\Sigma_2')$ are obtain by a local modification of representations of a graphic matroid.
\end{enumerate}
The reductions in part (3) are similar to, and include, the reductions described in Sections~\ref{sec:Shih} and~\ref{sec:quad}.
The small set of examples in part (4) arises from a construction as the one in Figure~\ref{fig:badEx}.
\begin{figure}[ht]
\begin{center}
\includegraphics[width=15cm]{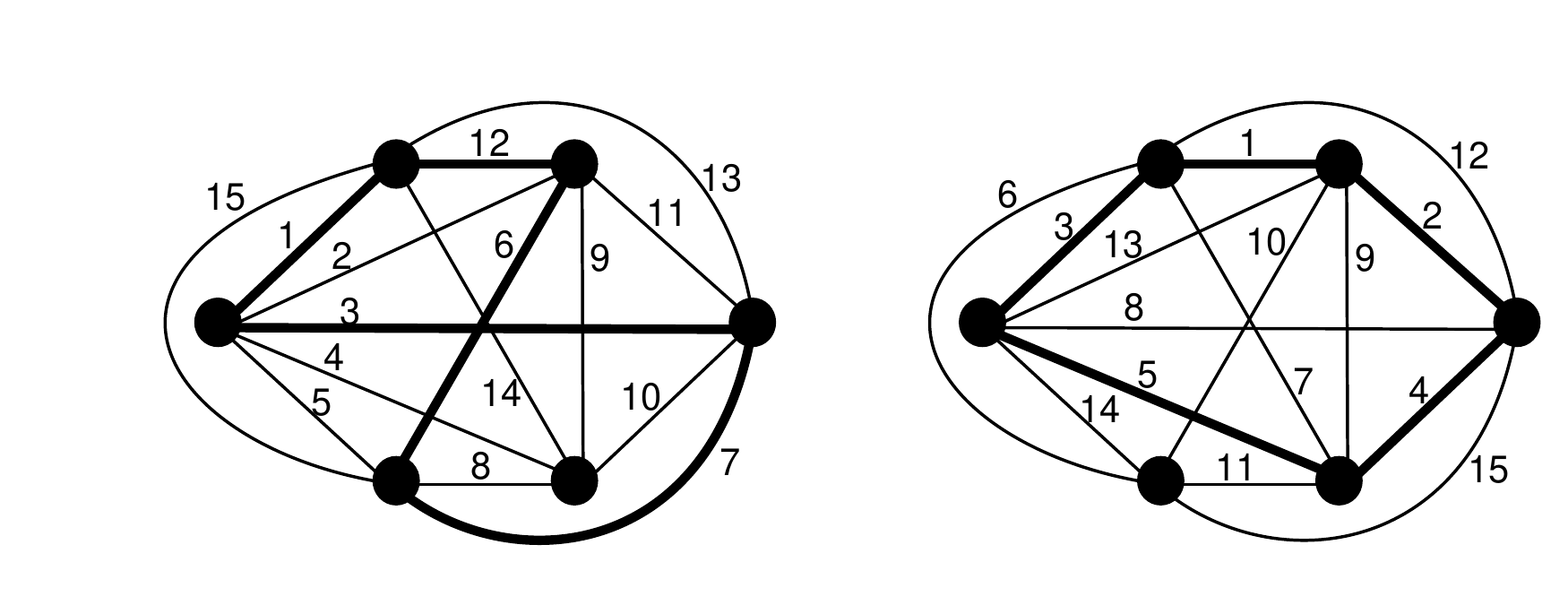}
\end{center}
\vspace{-0.3 in}\caption{Siblings. Bold edges are odd.}
\label{fig:badEx}
\end{figure}
Outcome (5) is constructed as follows. Let $G$ be a graph and $(H,\Gamma)$ be a signed graph such that $\cycle(G)=\ecycle(H,\Gamma)$. 
Suppose that $e,f,g$ are edges forming an odd triangle in $(H,\Gamma)$. 
Let $v_{ef}$ be the vertex in $H$ incident to $e$ and $f$; define $v_{fg}$ and $v_{eg}$ similarly.
Construct a graph $H'$ by adding a new vertex $v$ and three new edges $\bar{e}, \bar{f}, \bar{g}$ to $H$ as follows: 
$\{\bar{e}, \bar{f}, \bar{g}\}$ form a triad in $H'$ incident to the new vertex $v$.
The other end of $\bar{e}$ (respectively $\bar{f}, \bar{g}$) in $H'$ is $v_{fg}$ (respectively $v_{eg}, v_{ef}$).
Now construct a graph $G'$ from $G$ by adding edges $\bar{e}, \bar{f}, \bar{g}$,
where $\bar{e}$ is parallel to $e$, $\bar{f}$ is parallel to $f$ and $\bar{g}$ is parallel to $g$.
Then $\ecut(H',\{v,v_{ef},v_{eg},v_{fg}\})=\ecut(G',T')$, where $T':=\odd(G[\{e,f,g\}])$. Hence the graphs $G'$ and $H'$ are siblings.
An example of this construction is given in Figure~\ref{fig:extraEx}. This example arises from one of the constructions in Shih's theorem, which is discussed in the next section.
\begin{figure}[ht]
\begin{center}
\includegraphics[width=13cm]{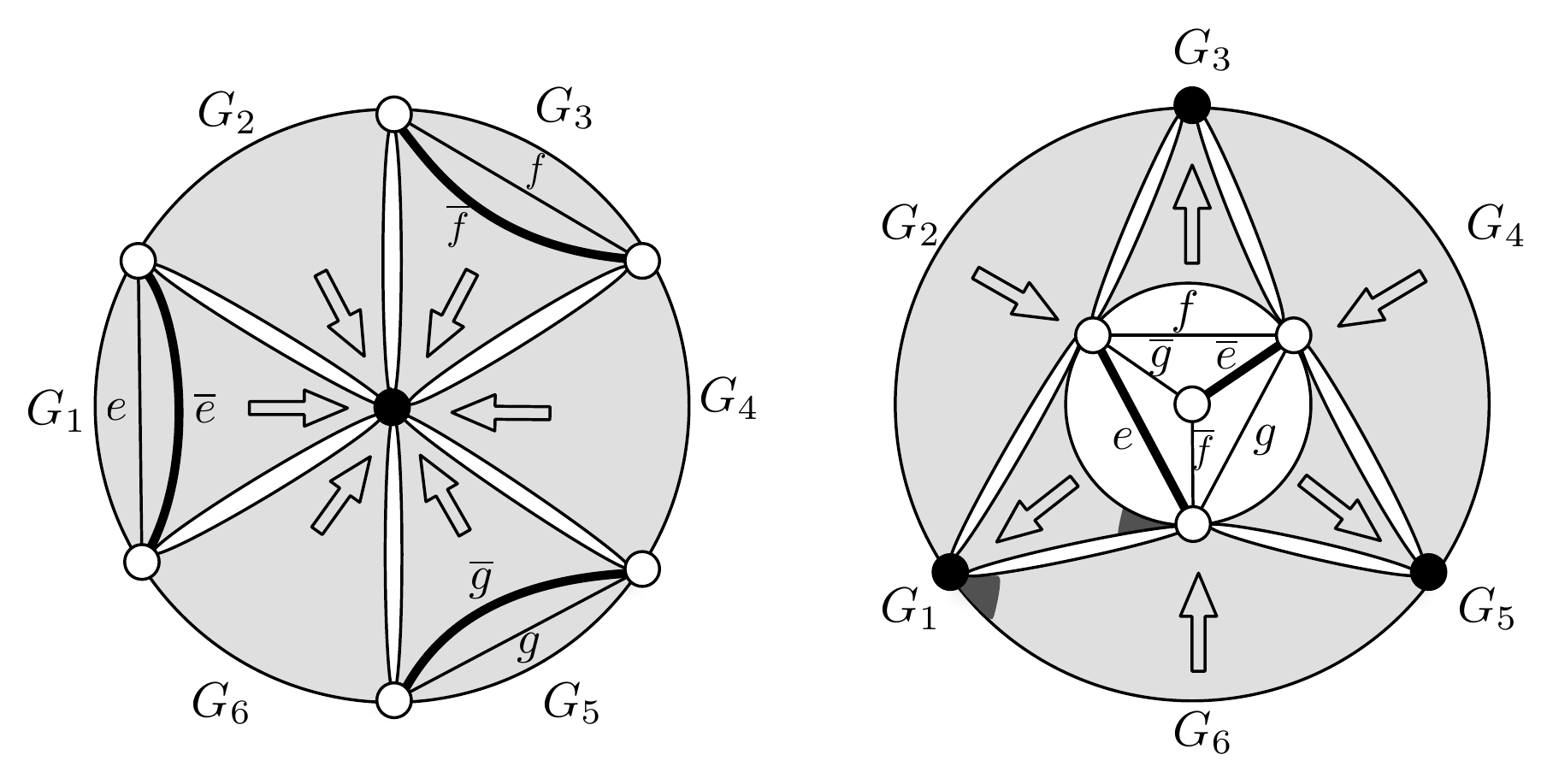}
\end{center}
\vspace{-0.3 in}\caption[Modification of Shih's operation.]{Modification of Shih's operation. Bold and shaded edges are odd, white vertices are terminals.}
\label{fig:extraEx}
\end{figure}
%
%
\subsection{The graphic and co-graphic case}
In this section we consider the Isomorphism Problem for even cycle matroids which are graphic.
Suppose that for a signed graph $(H,\Gamma)$, $\ecycle(H,\Gamma)$ is a graphic matroid.
Hence there exists a graph $G$ such that $\ecycle(H,\Gamma)=\cycle(G)$.
If $(H,\Gamma)$ does not contain any odd cycles, then $\cycle(H)=\cycle(G)$, the two graphs are equivalent and the Isomorphism Problem is solved.
Thus we assume that $(H,\Gamma)$ contains an odd cycle $C$.
Every odd cycle of $H$ can be generated by $C$ and a basis for the even cycles of $H$.
Thus $\cycle(G)$ is a subspace of $\cycle(H)$ and
$\textrm{dim}(\cycle(G))=\textrm{dim}(\cycle(H))-1$.
Moreover, if we know the structure of $G$ and $H$, then we can determine the signature $\Gamma$ by Theorem~\ref{isopair-co},
as the signature pair is unique in this case.
Therefore the following result (proved by Shih in his doctoral dissertation, see \cite{Shih82}) 
provides an answer to the Isomorphism Problem for graphic even cycle matroids.
\begin{theorem}\label{iso:shih}
Suppose $G$ and $H$ are graphs such that $\cycle(G)$ is a subspace of $\cycle(H)$ and
$\textrm{dim}(\cycle(G))=\textrm{dim}(\cycle(H))-1$.
Then there exist graphs $G'$ and $H'$, equivalent to $G$ and $H$ respectively, such that one of the following holds.
\begin{enumerate}[\;\;\;(1)]
	\item $H'$ is obtained from $G'$ by identifying two distinct vertices.
	\item There exist graphs $G_1,\ldots, G_4$ (not necessarily all non-empty) and distinct vertices $x_i,y_i,z_i \in V(G_i)$
		such that $G'$ is obtained by identifying $x_i,y_{3-i},z_{2+i}$ to a vertex $w_i$, 
		for $i=1,\ldots,4$ (where the indices are modulo $4$).
		Moreover, $H'$ is obtained by identifying $x_1,x_2,x_3,x_4$ to a vertex $x$,
		identifying $y_1,y_2,y_3,y_4$ to a vertex $y$ and identifying $z_1,z_2,z_3,z_4$ to a vertex $z$.
	\item There exist graphs $G_1,\ldots,G_k$, with $k \geq 3$, and distinct vertices $x_i,y_i,z_i \in V(G_i)$, for $i=1,\ldots,k$, 
		such that $G'$ is obtained by identifying $z_1,\ldots,z_k$ to a vertex $z$ and, for $i=1,\ldots,k$, identifying 
		$y_{i-1}$ and $x_i$ to a vertex $w_i$ (where the indices are modulo $k$).
		Moreover, $H'$ is obtained by identifying $y_{i-1}$, $z_i$, $x_{i+1}$ to a vertex $w_i'$, for $i=1,\ldots,k$ 
		(where the indeces are modulo $k$).
\end{enumerate}
\end{theorem}
An example of outcome (2) is given in Figure~\ref{fig:Shih2}, where dotted lines represent vertices that are identified.
$G'$ is the graph on the left and $H'$ the graph on the right.
Let $P_1$ be a $(y,z)$-path in $G_1$ and $P_2$ be a $(y,z)$ path in $G_2$.
Then $P_1 \cup P_2$ is a cycle of $H'$ and not a cycle of $G'$.
Let $T:=\odd(G'[P_1 \cup P_2])=\{w_1,w_2,w_3,w_4\}$. By Theorem~\ref{isopair-co}, $\ecut(G',T)=\cut(H')$ and we may choose $\Gamma:=\delta_{G'}(w_1)$
(shaded in the figure). 
\begin{figure}[ht]
\begin{center}
\includegraphics[width=12cm]{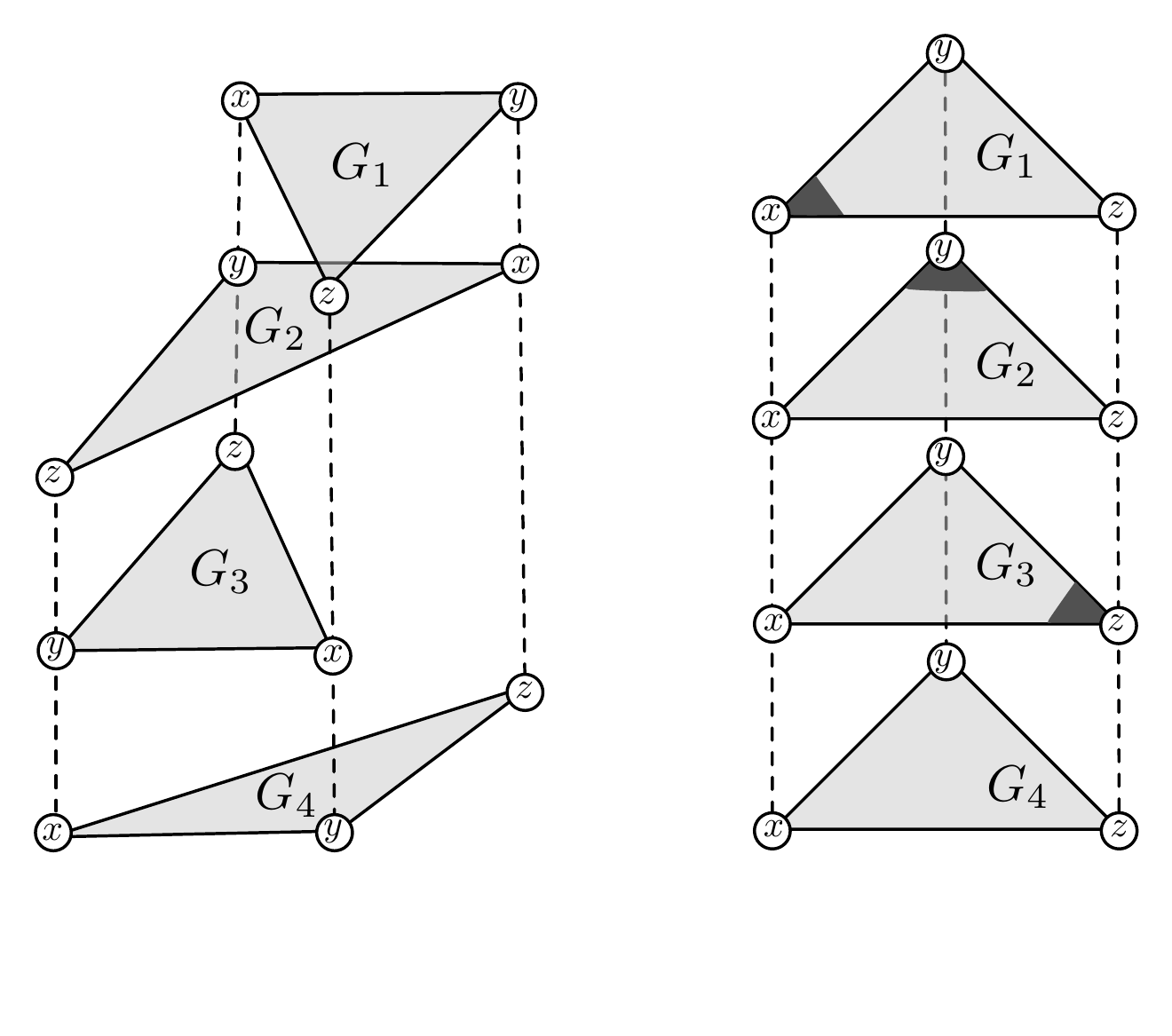}
\end{center}
\vspace{-0.6in}\caption{Shih operation 2.}
\label{fig:Shih2}
\end{figure}

An example of outcome (3) is given in Figure~\ref{fig:Shih3}, where the graph on the left is $G'$ and the one on the right is $H'$.
In this example we chose $G_1$ to be the graph with edges $1,2,3$ as in the figure. The arrows indicate how each piece is flipped.
We may choose $\Gamma:=\delta_{G'}(w_1)$ (shaded in the figure).
\begin{figure}[ht]
\begin{center}
\includegraphics[width=12cm]{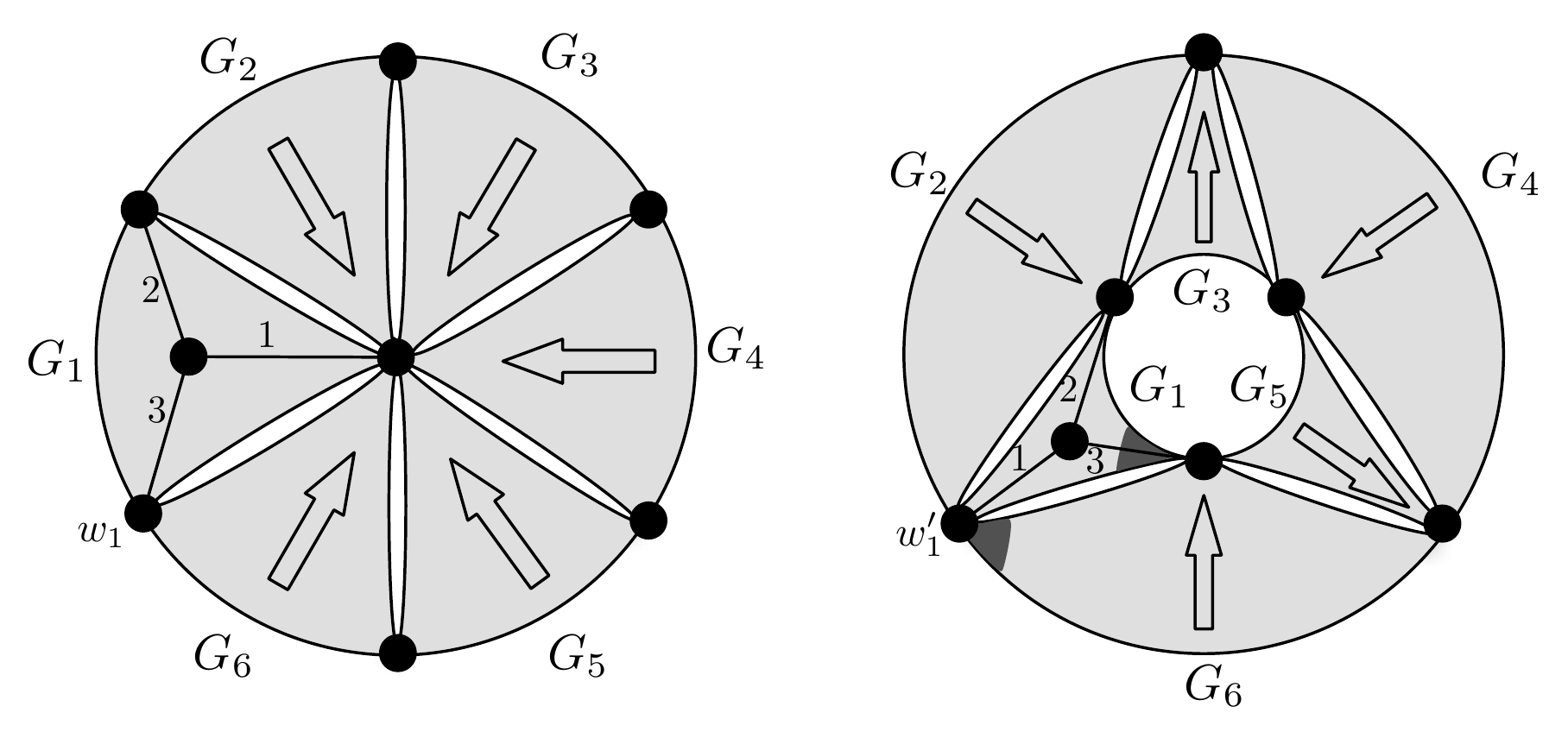}
\end{center}
\vspace{-0.3in}\caption{Shih operation 3.}
\label{fig:Shih3}
\end{figure}

Note that Theorem~\ref{iso:shih} also answers the Isomorphism Problem for even cut matroids in the case that
the even cut matroid represented by a graft $(G,T)$ is co-graphic.
In fact, by Theorem~\ref{isopair-co}, 
we have $\cycle(G)=\ecycle(H,\Gamma)$ if and only if $\ecut(G,T)=\cut(H)$, for some set of terminals $T$ of $G$.

As the Isomorphism Problem is solved for graphic matroids, we will mostly consider non-graphic matroids in this paper.
Let $(H,\Gamma)$ be a signed graph.
Suppose that $\Gamma\subseteq\delta_H(s) \cup \loops(H)$ for some vertex $s$ of $H$.
Let $G$ be obtained from $H$ by splitting $s$ according to $\Gamma$.
Then $\cycle(G)=\ecycle(H,\Gamma)$. Thus we have the following.
\begin{RE}\label{intro:projection}
Let $(H,\Gamma)$ be a signed graph. If $(H,\Gamma)$ has a blocking vertex, then $\ecycle(H,\Gamma)$ is a cycle matroid.
\end{RE}
%
\subsection{The class of Shih siblings}\label{sec:Shih}
Let signed graphs $(G_1,\Sigma_1)$ and $(G_2,\Sigma_2)$ be siblings and let  $T_1,T_2$ be the matching terminal pair.
If $|T_1|=2$ or $|T_2|=2$, we say that $(G_1,\Sigma_1)$ and $(G_2,\Sigma_2)$ are {\em Shih siblings}.

Suppose that $|T_2|=2$ and let $H_2$ be the graph obtained from $G_2$ by identifying the two vertices in $T_2$. 
Then $\ecut(G_2,T_2)=\cut(H_2)$.
It follows that $\ecut(G_1,T_1)=\cut(H_2)$. 
Therefore Theorem~\ref{iso:shih} gives a characterization of Shih siblings.
For example, the graphs $G_1$ and $H_2$ may be as in Figure~\ref{fig:Shih3} and we may obtain $G_2$ from the graph on the
right by splitting a vertex (for example, $w_1'$) into vertices $v^+$ and $v^-$. 
Then, up to resigning, $\Sigma_1=\delta_{G_2}(v^+)$ and $\Sigma_2$ is still $\delta_{G_1}(w_1)$.

Note that Theorem~\ref{iso:shih} completely characterizes the structure of $G_1$ and $H_2$ in cases (2) and (3)
and $G_2$ is obtained from $H_2$ by simply splitting any vertex.
Moreover, the matching signature pair is uniquely determined, by Theorem~\ref{isopair-co}.
However, if $|T_1|=|T_2|=2$, case (1) of the theorem occurs.
What Theorem~\ref{iso:shih} states in this case is that there exist equivalent graphs $H_1,H_2$ such that, for $i=1,2$, 
$H_i$ is obtained from $G_i$ by identifying two vertices.
Hence Theorem~\ref{iso:shih} does not characterize the structure of the graphs in this case.
Therefore we treat this type of siblings separately from the other Shih siblings and we provide an explicit characterization of them.
Let signed graphs $(G_1,\Sigma_1)$ and $(G_2,\Sigma_2)$ be siblings and let  $T_1,T_2$ be the matching terminal pair,
where $|T_1|=|T_2|=2$.
For $i=1,2$, let $H_i$ be obtained from $G_i$ by identifying the two vertices in $T_i$.
Then $\cut(H_1)=\ecut(G_1,T_1)=\ecut(G_2,T_2)=\cut(H_2)$. 
By Theorem~\ref{intro:whitneyCut}, $H_1$ and $H_2$ are equivalent. This justifies the following definition.

Consider a pair of equivalent graphs $H_1$ and $H_2$. 
Suppose that, for $i=1,2$, we have 
$\alpha_i\subseteq\delta_{H_i}(v_i)\cup\loops(H_i)$ for some $v_i\in V(H_i)$. 
Then, for $i=1,2$, let $G_i$ be obtained from $H_i$ by splitting $v_i$ into $v^-_i,v^+_i$ according to $\alpha_i$ and let $T_i:=\{v^-_i,v^+_i\}$.
Since $H_1$ and $H_2$ are equivalent, $\cut(H_1)=\cut(H_2)$. Thus 
\[
\ecut(G_1,T_1)=\cut(H_1)=\cut(H_2)=\ecut(G_2,T_2).
\]
In particular, if $G_1$ and $G_2$ are not equivalent, $(G_1,T_1)$ and $(G_2,T_2)$ are siblings. 
Let $\Sigma_1,\Sigma_2$ be the matching signature pair for $G_1,G_2$.
If $(G_1,\Sigma_1)$ and $(G_2,\Sigma_2)$ are inequivalent we say that the tuple
$
\dst=(H_1,v_1,\alpha_1,H_2,v_2,\alpha_2)
$
is a {\em split-template}
and that $(G_1,\Sigma_1)$ and $(G_2,\Sigma_2)$ (respectively $(G_1,T_1)$ and $(G_2,T_2)$) are {\em split siblings}
which {\em arise} from $\dst$. Split siblings are a special type of Shih siblings,
namely the type arising from outcome (1) in Theorem~\ref{iso:shih}.
An explicit characterization of split siblings representing a $3$-connected matroid is given in Section~\ref{sec:split}. 
%
\subsection{The class of quad siblings}\label{sec:quad}
Let $(H_1,\Gamma_1)$ and $(H_2,\Gamma_2)$ be a pair of equivalent signed graphs.
Suppose that, for $i=1,2$, $\Gamma_i\subseteq\delta_{H_i}(v_i)\cup\delta_{H_i}(w_i)\cup \loops(H_i)$,
for some $v_i,w_i \in V(H_i)$.
Then, for $i=1,2$, let $(G_i,T_i)$ be the graft obtained by unfolding $(H_i,\Gamma_i)$ according to $v_i,w_i$ and $\alpha_i$,$\beta_i$
(where $\Gamma_i=\alpha_i \Delta \beta_i$).
It follows from Remark~\ref{iso:unfold}(2) that
\[
\ecut(G_1,T_1)=
\ecycle(H_1,\Gamma_1)^*=
\ecycle(H_2,\Gamma_2)^*=
\ecut(G_2,T_2).
\]
In particular, if $G_1$ and $G_2$ are not equivalent, $(G_1,T_1)$ and $(G_2,T_2)$ are siblings.
Let $\Sigma_1,\Sigma_2$ be the matching signature pair for $G_1,G_2$.
If $G_1$ and $G_2$ are not equivalent, we say that the tuple
$\dst=(H_1,v_1,w_1,\alpha_1,\beta_1,H_2,v_2,w_2,\alpha_2,\beta_2)$
is a {\em quad-template} and that $(G_1,\Sigma_1)$ and $(G_2,\Sigma_2)$ 
(respectively $(G_1,T_1)$ and $(G_2,T_2)$) are {\em quad siblings}
which {\em arise} from $\dst$. An explicit characterization of quad siblings representing
a $3$-connected non-graphic matroid is given in Section~\ref{sec:doublesplit}. 
%
\subsection{Isomorphism for Shih siblings}\label{sec:split}
Let signed graphs $(G_1,\Sigma_1)$ and $(G_2,\Sigma_2)$ be Shih siblings and let  $T_1,T_2$ be the matching terminal pair.
Suppose $|T_2|=2$, and let $H_2$ be the graph obtained from $G_2$ by identifying the two vertices in $T_2$. 
Then $\ecut(G_1,T_1)=\ecut(G_2,T_2)=\cut(H_2)$ and some graphs $G_1'$ and $H_2'$, equivalent to $G_1$ and $H_2$ respectively, 
satisfy one of the outcomes of Theorem~\ref{iso:shih}. 
Outcomes (2) and (3) completely characterize the structure of $G_1$ and $H_2$.
The aim of this section is to provide a structural characterization of outcome (1).
Recall that if outcome (1) occurs, then $(G_1,\Sigma_1)$ and $(G_2,\Sigma_2)$ are split siblings.
The proof of the following result is given in Section~\ref{sec:proofs1}.
\begin{theorem}\label{split:char}
Let $M$ be a $3$-connected even cycle matroid.
If $(G_1,\Sigma_1)$ and $(G_2,\Sigma_2)$ are representations of $M$ 
which are split siblings, then they are either:
\begin{enumerate}[\;\;\;(1)]
\item simple siblings, or
\item nova siblings, or 
\item can be reduced.
\end{enumerate}
\end{theorem}
\noindent
It remains to define the terms ``simple siblings", ``nova siblings", and ``reduced".
We need some preliminary definitions.

By a {\em sequence} $(X_1,\ldots,X_k)$ we mean a family of sets $\{X_1,\ldots,X_k\}$ where $X_i$ precedes $X_j$ when $i<j$.
We say that $\bS=(X_1,\ldots,X_k)$ is a {\em w-sequence} of $G$ if, for all $i\in[k]$, 
$X_i$ is a $2$-separation of the graph obtained from $G$ by performing Whitney-flips on $X_1,\ldots,X_{i-1}$ (in this order).
We denote by $\Wflip[G,\bS]$ the graph obtained from $G$ by performing Whitney-flips on $X_1,\ldots,X_k$ (in this order).
For our purpose the position of loops is irrelevant. 
Hence we will assume that loops form distinct components of the graph. 
Therefore, if $G$ and $G'$ are equivalent graphs that are $2$-connected, except for possible loops,
then $G'=\Wflip[G,\bS]$ for some w-sequence $\bS$ of $G$.

A family $\bS=\{X_1,\ldots,X_k\}$ of sets of edges of a graph $G$ is a {\em w-star} if
\begin{enumerate}[\;\;\;(a)]
	\item $X_i \cap X_j = \emptyset$, for all distinct $i, j\in[k]$;
	\item there exist distinct $z,v_1,\ldots,v_k\in V(G)$ such that $\zB_G(X_i)=\{z,v_i\}$, for all $i\in[k]$;
	\item no edge with ends $z,v_i$ is in $X_i$, for all $i\in[k]$.
\end{enumerate}
\noindent
Vertex $z$ is the {\em center} of the w-star $\bS$.

Consider a split-template $(H_1,v_1,\alpha_1,H_2,v_2,\alpha_2)$.
If $H_1$ and $H_2$ are $2$-connected, except for possible loops, we have that
$H_2=\Wflip[H_1,\bS]$ for some w-sequence $\bS$.
In this case we slightly abuse terminology and say that $(H_1,v_1$, $\alpha_1$, $H_2$, $v_2$, $\alpha_2,\bS)$ is a split-template.
This is only defined for the case where $H_1$ and $H_2$ are $2$-connected, up to loops.
Thus, when specifying a w-sequence $\bS$ for a slip-template, we will implicitly assume
that $H_1$ and $H_2$ are $2$-connected, except for possible loops.
\begin{RE}\label{iso:alphasign}
Let $\dst=(H_1,v_1,\alpha_1,H_2,v_2,\alpha_2)$ be a split-template and let $(G_1,\Sigma_1)$ and
$(G_2,\Sigma_2)$ be split siblings that arise from $\dst$. 
Then, up to signature exchange, we have $\Sigma_1=\Sigma_2=\alpha_1\tr\alpha_2$. 
\end{RE}
\begin{proof}
For $i=1,2$, vertex $v_i$ of $H_i$ gets split into vertices $v^-_i,v^+_i$ of $G_i$.
By construction, $\alpha_i=\delta_{G_i}(v^-_i)$, for $i=1,2$.
As $v^-_1\in T_1$, Theorem~\ref{isopair-co} implies that $\alpha_1$ is a signature of $(G_2,\Sigma_2)$. 
As $\alpha_2$ is a cut of $G_2$, $\alpha_1\tr\alpha_2$ is a signature of $(G_2,\Sigma_2)$.
By symmetry, $\alpha_1\tr\alpha_2$ is also a signature of $(G_1,\Sigma_1)$.
\end{proof}
\subsubsection{Simple twins}
Consider a split-template $\dst=(H_1,v_1,\alpha_1,H_2,v_2,\alpha_2,\bS)$.
If $\bS=\emptyset$, i.e. $H_1=H_2$, 
then  $\dst$ is {\em simple} and $(G_1,\Sigma_1)$ and $(G_2,\Sigma_2)$ 
arising from $\dst$ are {\em simple twins}.
By Remark~\ref{iso:alphasign}, we may assume that $\Sigma_1=\Sigma_2=\alpha_1\tr\alpha_2$. 
Suppose that vertex $v_1$ of $H_1$ gets split into vertices $v^-_1$ and $v^+_1$ of $G_1$.
Then $\alpha_1\subseteq\delta_{G_1}(v^-_1)$ and $\alpha_2\subseteq\delta_{G_1}(v_2)$.
Hence, $v^-_1$ and $v_2$ form a blocking pair of $(G_1,\Sigma_1)$. 
Thus we have the following.
\begin{RE}\label{split:simplebp}
Simple twins have blocking pairs.
\end{RE}
\noindent
It can easily be verified that two simple twins are related by Lov\'asz-flips.
\subsubsection{Nova twins}
Let $(G,\Sigma)$ be a signed graph with distinct vertices $s_1$ and $s_2$.
For $i=1,2$, let $C_i$ denote a circuit of $H_i$ using $s_i$ and avoiding $s_{3-i}$.
Suppose that $C_1$ and $C_2$ are either vertex disjoint or that $C_1$ and $C_2$ intersect exactly in a path.
In the former case let $P$ denote a path with ends $u_i\in V_G(C_i)-\{s_i\}$, for $i=1,2$, such that 
$V_G(P)\cap\bigl(V_G(C_1)\cup V_G(C_2)\bigr)=\{u_1,u_2\}$.
In the latter case define $P$ to be the empty set.
Then we say that the triple $(C_1,C_2,P)$ form {\em $\{s_1,s_2\}$-handcuffs}.
We say that $X\subseteq G$ is a {\em handcuff-separation} if
$X$ is a $2$-separation of $G$ and
there exist $\{s_1,s_2\}$-handcuffs of $(G[X],\Sigma\cap X)$, where $s_1,s_2$ are the vertices in $\zB_G(X)$.

A split-template $\dst=(H_1,v_1,\alpha_1,H_2,$ $v_2,\alpha_2,\bS)$ is {\em nova} if, for $i=1,2$:
\begin{enumerate}[\;\;\;(N1)]
	\item $\bS$ is a w-star of $H_i$ with center $v_i$, and
	\item all $X'\subseteq X\in \bS$ with $\zB_{H_i}(X')=\zB_{H_i}(X)$ are handcuff-separations of $(H_i,\alpha_1\tr\alpha_2)$.
\end{enumerate}
We say that $(G_1,\Sigma_1)$ and $(G_2,\Sigma_2)$ arising from $\dst$ are {\em nova twins}.
An example of nova twins is given in Figure~\ref{fig:nova}, where dotted lines denote vertices which are identified.
We could have defined nova twins omitting condition~(N2).
This would yield a weaker version of Theorem~\ref{split:char}.
However, the stronger version is needed for an upcoming paper on stabilizer theorems for even cycle matroids.
\begin{figure}[h]
\begin{center}
\includegraphics[width=13.5cm]{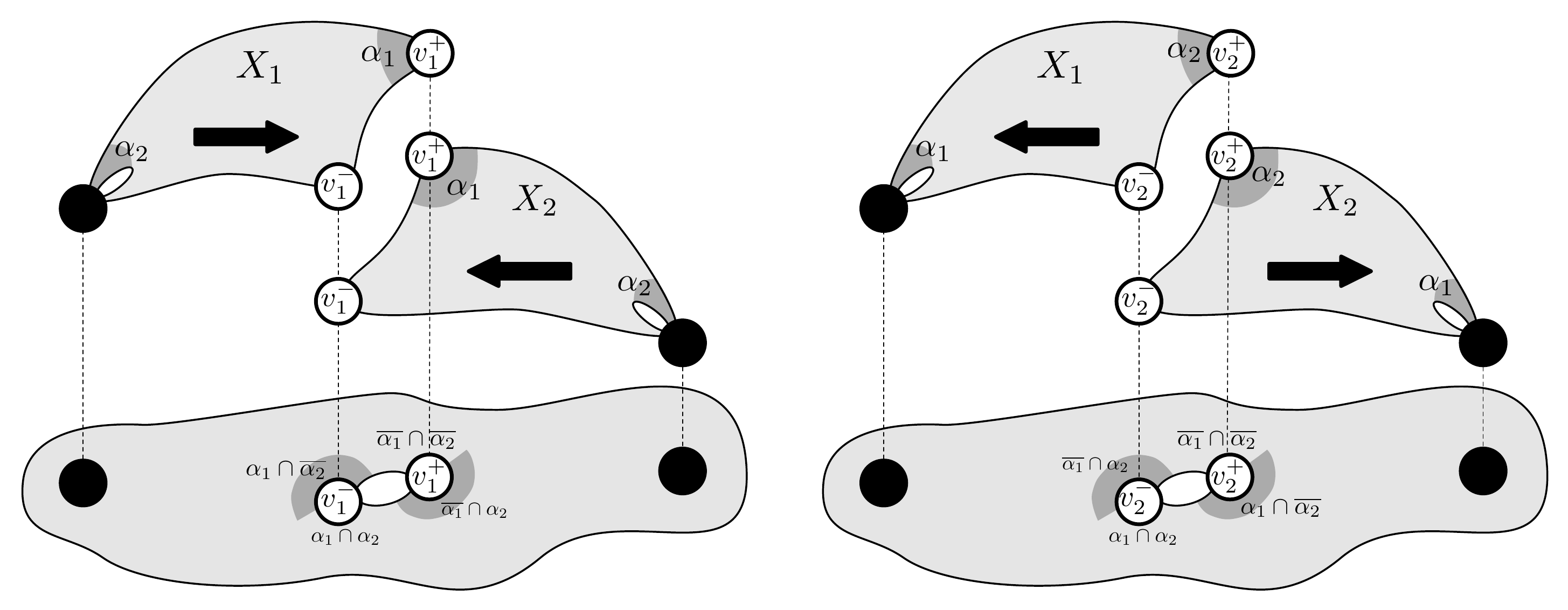}
\end{center}
\vspace{-0.3in}\caption{Example of nova twins with $|\bS|=2$.}
\label{fig:nova}
\end{figure}
%
\subsubsection{Shih siblings}
We say that $(G_1,\Sigma_1)$ and $(G_2,\Sigma_2)$ are simple (respectively nova) {\em siblings} if,
for $i=1,2$, there exists $(G_i',\Sigma_i')$ equivalent to $(G_i,\Sigma_i)$ such that
$(G_1',\Sigma_1')$ and $(G_2',\Sigma_2')$ are simple (respectively nova) twins.
\subsubsection{Reduction}\label{sec:split:reduction}
Consider grafts $(G_1,T_1)$ and $(G_2,T_2)$ where, for $i=1,2$, $T_i$ consists of vertices $v_i^-,v_i^+$. 
We write $(G_1,T_1)\oplus(G_2,T_2)$ to indicate the graft $(G,T)$ where $G$ is obtained from $G_1$ and $G_2$
by identifying vertex $v_1^-$ with $v_2^-$ and by identifying vertex $v_1^+$ with vertex $v_2^+$. 
Denote by $v^-$ (respectively $v^+$) the vertex in $G$ corresponding to $v_1^-,v_2^-$ (respectively $v_2^-,v_2^+$) and let $T=\{v^-,v^+\}$. 
Note that $(G,T)$ is defined uniquely from $(G_1,T_1)$ and $(G_2,T_2)$ up to a possible 
Whitney-flip on $E(G_1)$.

Consider split siblings $(G_1,\Sigma_1)$ and $(G_2,\Sigma_2)$ and let $T_1,T_2$ be the matching terminal pair. 
Suppose that there exists $X\subseteq E(G_1)$ such that $\zB_{G_1}(X)=T_1$. 
For $i=1,2$, let $H_i$ be obtained from $G_i$ by identifying the vertices in $T_i$ to a single vertex $v_i$. 
Then $H_1[X]$ is a block of $H_1$ attached to vertex $v_1$. 
As $(G_1,T_1)$ and $(G_2,T_2)$ are split siblings, $H_2[X]$ is also a block of $H_2$ attached to $v_2$.
It follows that $\zB_{G_2}(X)=T_2$. For $i=1,2$, define $G'_i:=G_i[X]$ and $G''_i:=G_i[\bar{X}]$.
Let $T'_i$ and $T''_i$ denote the vertices corresponding to $T_i$ in $G'_i$ and $G''_i$ respectively. 
Then, for $i=1,2$, $(G_i,T_i)=(G'_i,T'_i)\oplus(G''_i,T''_i)$.
Observe that $(G'_1,T'_1)$ and $(G'_2,T'_2)$ are split siblings and so are $(G''_1,T''_1)$ and $(G''_2,T''_2)$.
We say in that case that $(G_1,\Sigma_1)$ and $(G_2,\Sigma_2)$ can be {\em reduced}.
\subsection{Isomorphism for quad siblings} \label{sec:doublesplit}
%
The main result of this section is the following.
\begin{theorem}\label{doublesplit:char}
Let $M$ be a $3$-connected non-graphic even cycle matroid. If $(G_1,\Sigma_1)$ and $(G_2,\Sigma_2)$ are representations of $M$
which are quad siblings, then they are either:
\begin{enumerate}[\;\;\;(1)]
	\item shuffle siblings,
	\item tilt siblings, 
	\item twist siblings,
	\item widget siblings,
	\item gadget siblings, or
	\item $\Delta$-reducible.
\end{enumerate}
\end{theorem}

\noindent
The proof of Theorem~\ref{doublesplit:char} is in Section~\ref{sec:proofs2}.
The terms ``shuffle siblings", ``tilt siblings", ``twist siblings", ``widget siblings", ``gadget siblings"
and ``$\Delta$-reducible" are defined in the next sections.
\subsubsection{Shuffle twins}
Consider a graph $G$ and let $\{a,b,c,d\}\subseteq V(G)$.
Suppose that $E(G)$ can be partitioned into sets $X_1,\ldots,X_4$ (not necessarily all non-empty) 
such that, for all $i\in[4]$, $\zB_G(X_i)\subseteq\{a,b,c,d\}$. 
For all $i\in[4]$, denote by $a_i$ (respectively $b_i,c_i,d_i$) the copy of vertex 
$a$ (respectively $b,c,d$) of $G[X_i]$. Then construct $G'$ by:
\begin{itemize}
\item identifying vertices $a_1,b_2,c_3,d_4$ to a vertex $a'$;
\item identifying vertices $b_1,a_2,d_3,c_4$ to a vertex $b'$;
\item identifying vertices $c_1,d_2,a_3,b_4$ to a vertex $c'$;
\item identifying vertices $d_1,c_2,b_3,a_4$ to a vertex $d'$.
\end{itemize}
We say that $G$ and $G'$ are {\em shuffle twins}.
We will show that they are siblings with matching terminal pair $\{a,b,c,d\}$ and $\{a',b',c',d'\}$.
Shuffle twins were introduced by Norine and Thomas~\cite{NT}.

Let $H$ (respectively $H'$) be obtained by folding $(G,\{a,b,c,d\})$ (respectively $(G',\{a',b'$, $c',d'\})$) with the pairing $a,b$ and $c,d$
(respectively $a',b'$ and $c',d'$).
Let $\alpha:=\delta_{G}(a)$, $\beta:=\delta_{G}(c)$,
$\alpha':=\delta_{G'}(a')$ and $\beta':=\delta_{G'}(c')$.
Then $(H_1,\alpha \tr \beta)$ and $(H_2,\alpha' \tr \beta')$ are equivalent, hence $G$ and $G'$
are quad siblings with matching terminal pair 
$\{a,b,c,d\}$ and $\{a',b',c',d'\}$.
\begin{figure}[h]
\begin{center}
\includegraphics[width=11cm]{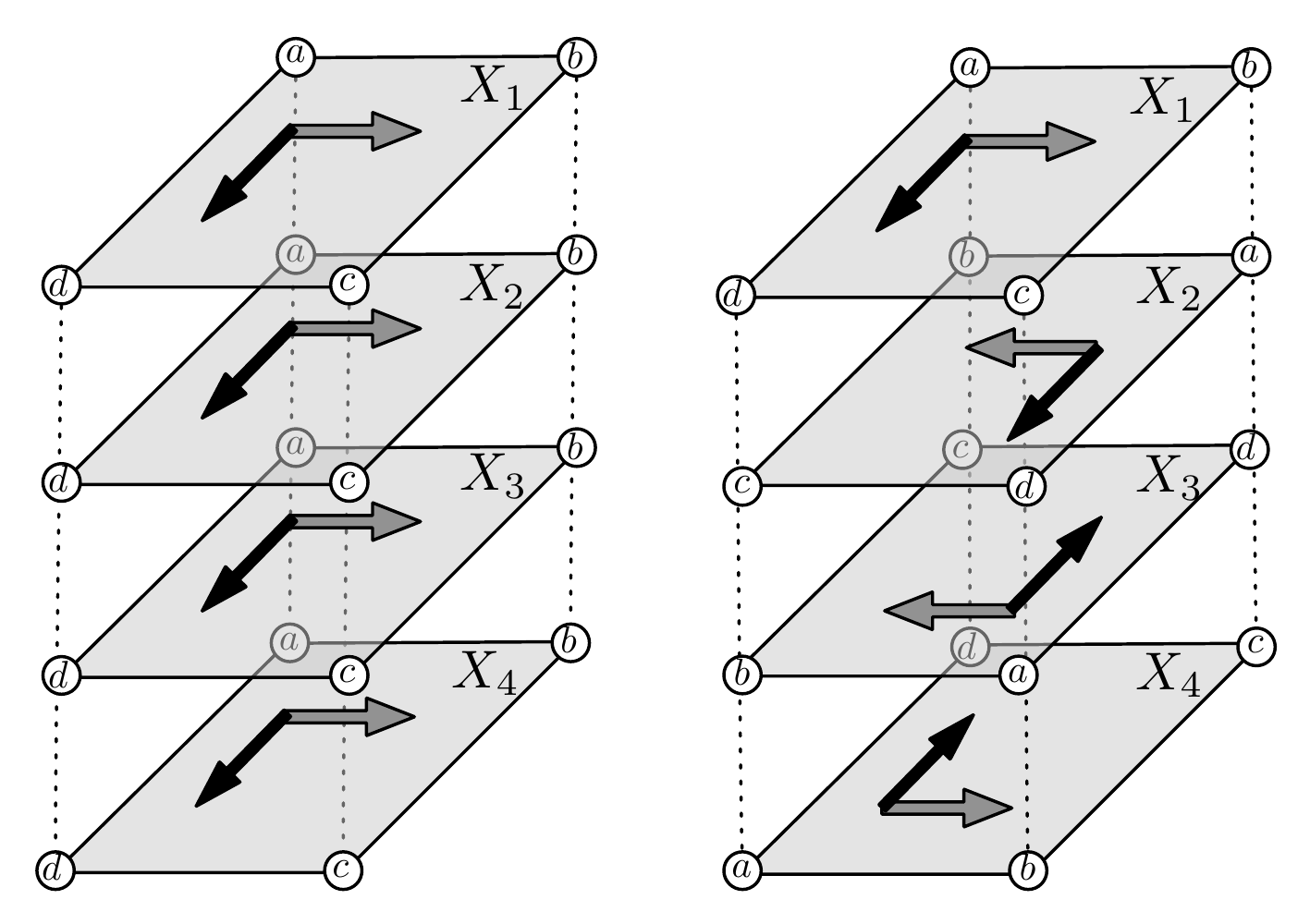}
\end{center}
\vspace{-0.3in}\caption{Shuffle twins. Dotted lines denote vertices that are identified.}
\end{figure}
\subsubsection{Tilt twins}\label{sec:tiltDef}
Consider a graph $G$ with 
distinct edges $e,f,g,h\in E(G)$ and distinct vertices $a_1$, $a_2$, $b_1$, $b_2$, $c$, $d$.
Suppose that $e$ and $f$ have ends $a_1,a_2$ and $g$ and $h$ have ends $b_1,b_2$. 
Suppose we can partition $E(G)$ into $X_1,X_2,\{e,f,g,h\}$, such that 
$V_G(X_1)\cap V_G(X_2)=\{c,d\}$ and 
$a_1,b_1\in V_G(X_1)$, 
$a_2,b_2\in V_G(X_2)$. 
For $i=1,2$, denote by $c_i$ (respectively $d_i$) the copy of vertex $c$ (respectively $d$) in $G[X_i]$. 
Construct $G'$ from $G[X_1],G[X_2]$ by:
\begin{itemize}
\item identifying vertices $a_1$ and $a_2$;
\item identifying vertices $b_1$ and $b_2$;
\item joining $c_1$ and $c_2$ with edges $e$ and $g$;
\item joining $d_1$ and $d_2$ with edges $f$ and $h$.
\end{itemize}
We say that $G$ and $G'$ are {\em tilt twins}.
In general, we say that $G$ and $G'$ are tilt twins even if not all edges $e,f,g,h$ in the above construction are present.
Tilt twins were introduced by Gerards~\cite{Gerards00}.

Let $H$ (respectively $H'$) be obtained 
by folding $(G,\{a_1,a_2,b_1,b_2\})$ (respectively $(G',\{c_1$, $c_2,d_1,d_2\})$) with the pairing $a_1,a_2$ and $b_1,b_2$
(respectively $c_1,c_2$ and $d_1,d_2$).
Let $\alpha:=\delta_{G}(a_1)$, $\beta:=\delta_{G}(b_1)$,
$\alpha':=\delta_{G'}(c_1)$ and $\beta':=\delta_{G'}(d_1)$.
Then $(H_1,\alpha \tr \beta)$ and $(H_2,\alpha' \tr \beta')$ are equivalent, hence $G$ and $G'$
are quad siblings with matching terminal pair 
$\{a_1,a_2,b_1,b_2\}$ and $\{c_1,c_2,d_1,d_2\}$.
\begin{figure}[h]
\begin{center}\includegraphics[width=10.5cm]{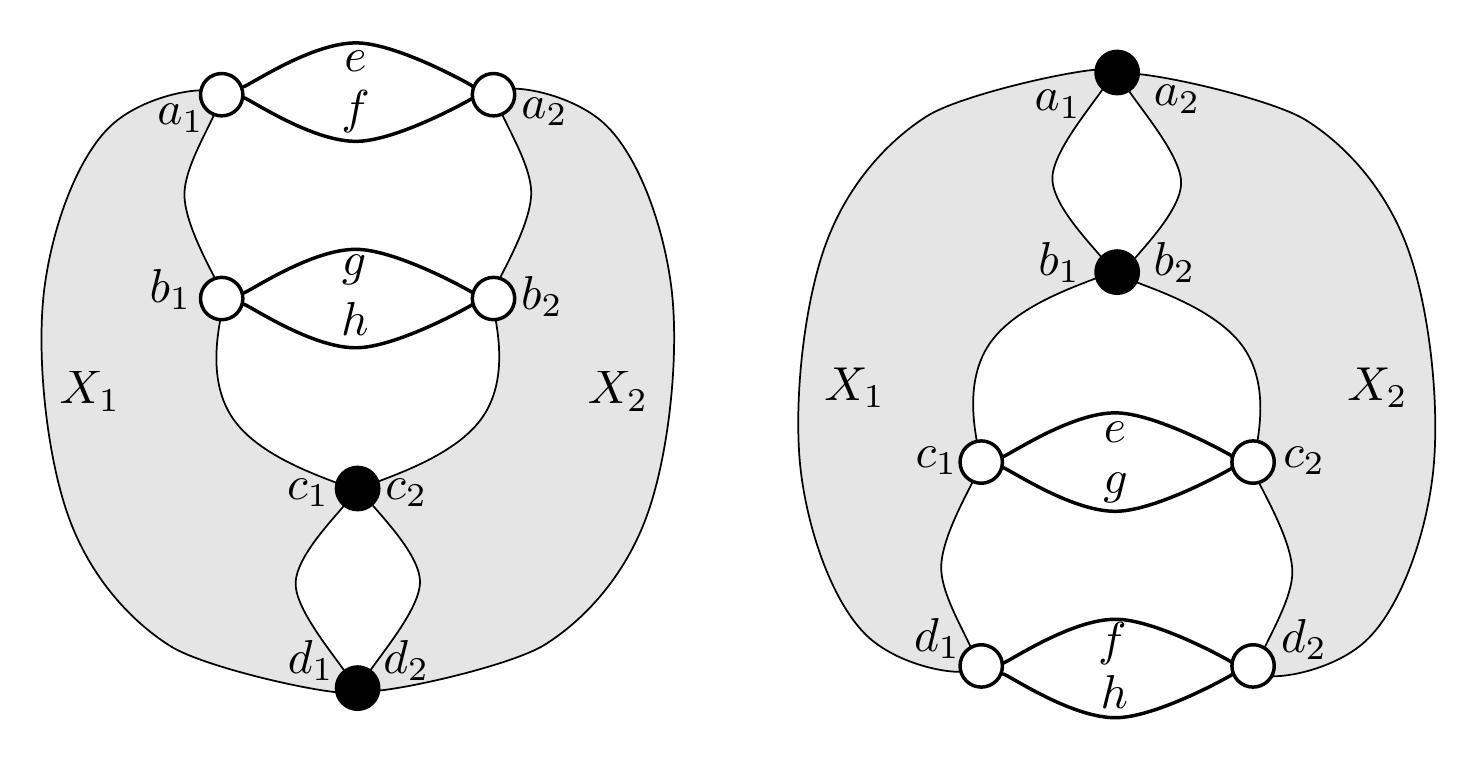}\end{center}
\vspace{-0.3in}\caption{Tilt twins.}
\label{fig:tilt}
\end{figure}
\subsubsection{Twist twins}
Consider a graph $G$ with 
distinct edges $e,f,g,h$ and distinct vertices $a_1,a_2,b,c,d$.
Suppose that $e$ and $f$ have ends $a_1,a_2$ and $g$ and $h$ have ends $b,c$.
Suppose we can partition $E(G)$ into $X_1,X_2,\{e,f,g,h\}$ such that 
$V_G(X_1)\cap V_G(X_2)=\{b,c,d\}$ and
$a_1\in V(X_1)$, $a_2\in V(X_2)$. 
For $i=1,2$, let $b_i$ (respectively $c_i,d_i$) denote the copy of vertex $b$ (respectively $c,d$) in $G[X_i]$. 
Construct $G'$ from $G[X_1],G[X_2]$ by:
\begin{itemize}
\item identifying vertices $a_1$ and $a_2$;
\item identifying vertices $b_1$ and $c_2$, calling the resulting vertex $\tilde{b}$;
\item identifying vertices $c_1$ and $b_2$, calling the resulting vertex $\tilde{c}$;
\item joining $\tilde{b}$ and $\tilde{c}$ with edges $e$ and $g$;
\item joining $d_1$ and $d_2$ with edges $f$ and $h$.
\end{itemize}
We say that $G$ and $G'$ are {\em twist twins}.
In general, we say that $G$ and $G'$ are twist twins even if not all edges $e,f,g,h$ in the above construction are present.

Let $H$ (respectively $H'$) be obtained 
by folding $(G,\{a_1,a_2,b,c\})$ (respectively $(G',\{\tilde{b},\tilde{c}$, $d_1,d_2\})$) with the pairing $a_1,a_2$ and $b,c$
(respectively $\tilde{b},\tilde{c}$ and $d_1,d_2$).
Let $\alpha:=\delta_{G}(a_1)$, $\beta:=\delta_{G}(b)$,
$\alpha':=\delta_{G'}(\tilde{b})$ and $\beta':=\delta_{G'}(d_1)$.
Then $(H_1,\alpha \tr \beta)$ and $(H_2,\alpha' \tr \beta')$ are equivalent, hence $G$ and $G'$
are quad siblings with matching terminal pair 
$\{a_1,a_2,b,c\}$ and $\{\tilde{b},\tilde{c},d_1,d_2\}$.
\begin{figure}[h]
\begin{center}\includegraphics[width=10.5cm]{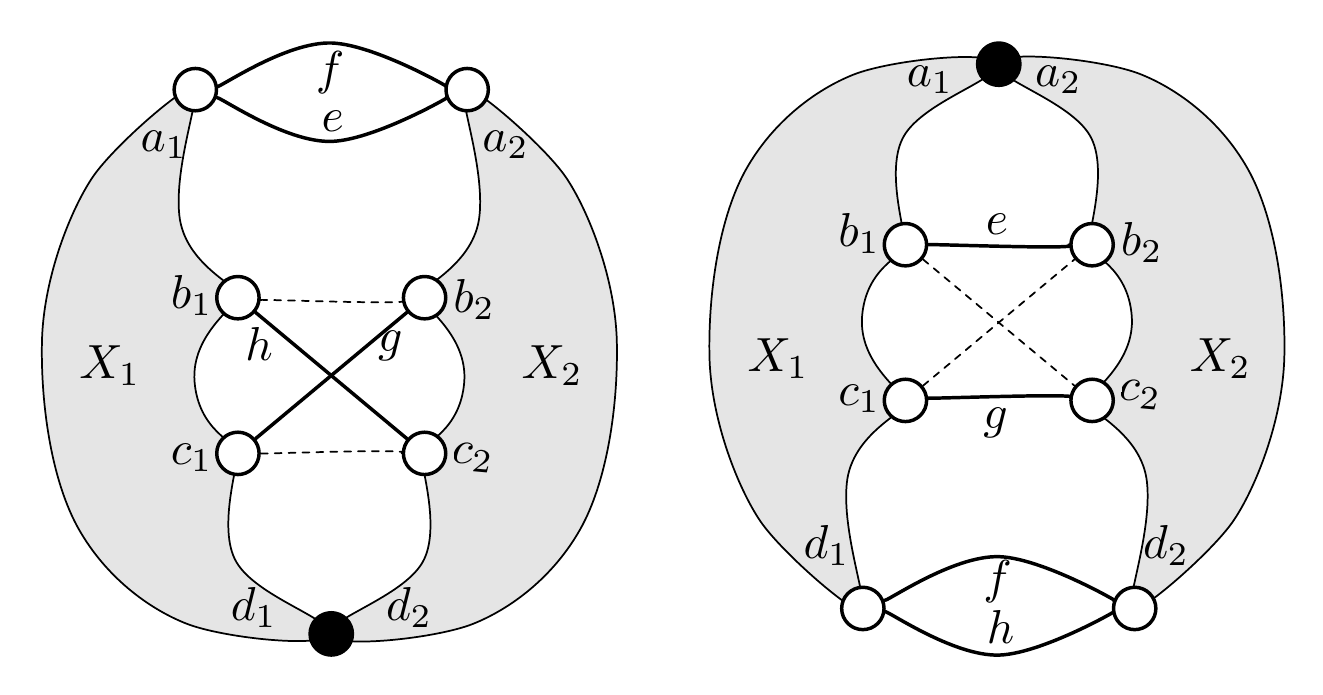}\end{center}
\vspace{-0.3in}\caption{Twist twins.  Dotted lines denote vertices that are identified.}
\end{figure}
\subsubsection{Widget twins}
Consider a graph $H_1$ with distinct edges $a,b,c,d,e,f,\ell_1,\ell_2,\ell_3,\ell_4$ 
and distinct vertices $v_1,z_1,w_1$ and $w_2$.
Suppose that $a$ and $b$ have ends $v_1,w_2$; $c$ and $d$ have ends $z_1,w_2$; $e$ and $f$ have ends $v_1,w_1$ 
and $\loops(H_1)=\{\ell_1,\ell_2,\ell_3,\ell_4\}$.
Suppose we can partition $E(H_1)$ into $X,\{a,b,c,d,e,f\},\loops(H_1)$ such that 
$\delta_{H_1}(w_2)=\{a,b,c,d\}$ and $\zB_{H_1}(X)=\{v_1,z_1,w_1\}$.
Let $H_2=\Wflip[H_1,\{a,b,c,d\}]$. Let the vertices in $H_2$ which are not in $\zB_{H_2}(\{a,b,c,d\})$
be labeled as in  $H_1$. Let $v_2 \in V(H_2)$ be the endpoint of $c$ distinct from $w_2$.
Let $\gamma \subseteq \delta_{H_1}(v_1) \cap X$.
Define $\alpha_1:=\gamma \cup \{a,e,\ell_1,\ell_2\}$; $\beta_1:=\{e,f,\ell_3,\ell_4\}$;
$\alpha_2:= \gamma \cup \{f,c,\ell_1,\ell_3\}$ and $\beta_2:=\{a,c,\ell_2,\ell_4\}$.
Let $\dst=(H_1,v_1,w_1,\alpha_1,\beta_1,H_2,v_2,w_2,$ $\alpha_2,\beta_2)$.
Note that $\dst$ is a quad-template.
Let $(G_1,\Sigma_1)$ and $(G_2,\Sigma_2)$ be the quad siblings arising from $\dst$.
We say that $G_1$ and $G_2$ are {\em widget twins}.
\subsubsection{Gadget twins}
Consider a graph $H_1$ with distinct edges $a_1,b_1,c_1,d_1,a_2,b_2,c_2,d_2,\ell_1,\ell_2,\ell_3,\ell_4$ 
and distinct vertices $v_1,z_1,u_1,w_1,w_2$.
Suppose that $a_i$ and $b_i$ have ends $v_1,w_i$, for $i=1,2$; $c_1$ and $d_1$ have ends $z_1w_1$; $c_2$ and $d_2$ have ends $u_1w_2$ 
and $\loops(H_1)=\{\ell_1,\ell_2,\ell_3,\ell_4\}$.
Suppose that we can partition $E(H_1)$ into sets $X,\{a_1,b_1,c_1,d_1,a_2,b_2,c_2,d_2\},\loops(H_1)$ such that 
$\delta_{H_1}(w_i)=\{a_i,b_i,c_i,d_i\}$, for $i=1,2$, and $\zB_{H_1}(X)=\{v_1,z_1,u_1\}$.
Let $H_2=\Wflip[H_1,(\{a_1,b_1,c_1,d_1\},\{a_2,b_2,c_2,d_2\})]$. 
Let the vertices in $H_2$ which are not in $\zB_{H_2}(\{a_1,a_2,b_1,b_2,c_1,c_2,d_1,d_2\})$
be labeled as in $H_1$. Let $v_2 \in V(H_2)$ be the endpoint of $c_1$ distinct from $w_1$.
Let $\gamma \subseteq \delta_{H_1}(v_1) \cap X$. 
Define $\alpha_1:=\gamma \cup \{a_1,a_2,\ell_1,\ell_2\}$, $\beta_1:=\{a_1,c_1,\ell_3,\ell_4\}$,
$\alpha_2= \gamma \cup \{c_1,c_2,\ell_1,\ell_3\}$ and $\beta_2:=\{a_2,c_2,\ell_2,\ell_4\}$.
Let $\dst=(H_1,v_1,w_1,\alpha_1,\beta_1$, $H_2$, $v_2$, $w_2,\alpha_2,\beta_2)$.
Note that $\dst$ is a quad-template.
Let $(G_1,\Sigma_1)$ and $(G_2,\Sigma_2)$ be the quad siblings arising from $\dst$.
We say that $G_1$ and $G_2$ are {\em gadget twins}.

\subsubsection{Quad siblings}
We say that signed graphs $(G_1,\Sigma_1)$ and $(G_2,\Sigma_2)$ are {\em shuffle} 
(respectively {\em tilt}, {\em twist}, {\em widget}, {\em gadget}) {\em siblings} if,
for $i=1,2$, there exists $(G_i',\Sigma_i')$ equivalent to $(G_i,\Sigma_i)$ such that
$(G_1',\Sigma_1')$ and $(G_2',\Sigma_2')$ are shuffle (respectively tilt, twist, widget, gadget) twins.

\subsubsection{$\Delta$-reduction}
Consider siblings $(G_1,\Sigma_1)$ and $(G_2,\Sigma_2)$ and 
suppose that edges $\{e_1,e_2,e_3\}$ form a
triangle of both $G_1$ and $G_2$ and (after possibly resigning) $\{e_1,e_2,e_3\}\cap\Sigma_i=\emptyset$, for $i=1,2$.
Let $H$ be a graph with distinct vertices $v_{12},v_{13},v_{23}$.
For $i=1,2$, let $G'_i$ be the graph obtained from $G_i$ by
(for all distinct $j,k\in[3]$) 
identifying the vertex of $G_i$ incident to both $e_j$ and $e_k$ 
with the vertex $v_{jk}$ of $H$, and by then deleting the edges $e_1,e_2,e_3$.
We say that $(G'_1,\Sigma_1)$ and $(G'_2,\Sigma_2)$ are obtained by a {\em $\Delta$-substitution} 
from $(G_1,\Sigma_1)$ and $(G_2,\Sigma_2)$ 
and that $(G_1,\Sigma_1)$ and $(G_2,\Sigma_2)$ are obtained by a {\em $\Delta$-reduction} from $(G_1',\Sigma_1)$ and $(G_2',\Sigma_2)$. 
By possibly omitting some of the edges of the triangle, we will make sure to not create parallel edges of the same parity
when applying a $\Delta$-reduction.
Note that in this case $(G'_1,\Sigma_1)$ and $(G'_2,\Sigma_2)$ are also siblings.

We say that siblings $(G_1,\Sigma_1)$ and $(G_2,\Sigma_2)$ are {\em $\Delta$-irreducible}
if no $\Delta$-reduction is possible in $(G_1,\Sigma_1)$, $(G_2,\Sigma_2)$,
otherwise we say that the siblings are {\em $\Delta$-reducible}.
We mainly consider $\Delta$-reductions to simplify the definitions of the various types of quad siblings.
For example, suppose that $(G_1,\Sigma_1)$ and $(G_2,\Sigma_2)$ are tilt twins, with the same notation as
in the definition of tilt twins in Section~\ref{sec:tiltDef}. Suppose that
$G_1$ contains edges $e_1$ and $e_2$ with ends $a_1,c$ and $a_2,c$ respectively. 
Then $\{e,e_1,e_2\}$ is an even triangle of both $G_1$ and $G_2$ and such a triangle may be substituted by any graph $H$.
\section{Whitney-flips}\label{sec:whitney}
In this section we provide results about equivalent graphs which will be used 
to prove Theorems~\ref{split:char} and~\ref{doublesplit:char}.
A difficulty when dealing with Whitney-flips comes from crossing $2$-separations.
We show how, in the cases we are interested in, we can reduce to considering only 
Whitney-flips on non-crossing separations.
Throughout this section graphs are $2$-connected.
However, the notions of w-sequences, the operation $\Wflip$ and the results in this section extend naturally to the class of graphs 
that are $2$-connected except for possible loops.
\subsection{Whitney-flips avoiding vertices}\label{sec:avoidVertex}
%
Recall the definitions of w-sequence and w-star given in Section~\ref{sec:split}.
We say that two sets $X$ and $Y$ are \emph{crossing} if all of $X \cap Y, X-Y, Y-X$ and $\overline{X} \cap \overline{Y}$ are non-empty.
A family of sets (or sequence) $\bS$ is {\em non-crossing} if $X$ and $Y$ are non-crossing for every $X,Y\in\bS$.
\begin{RE}\label{conn:laminar}
Let $G$ be a graph and let $\bS=(X_1,\ldots,X_k)$ be a non-crossing w-sequence for $G$.
Then for any permutation $i_1,\ldots,i_k$ of $1,\ldots,k$,
$\bS'=(X_{i_1},\ldots,X_{i_k})$ is a w-sequence and $\Wflip[G,\bS]=\Wflip[G,\bS']$.
\end{RE}
\noindent
In light of the previous remark, given a non crossing w-sequence $(X_1,\ldots,X_k)$,
we call the family $\bS:=\{X_1,\ldots,X_k\}$ a w-sequence and the notation $\Wflip[G,\bS]$ is well defined.

We can now state the first of the two main technical results of this section.
\begin{PR}\label{conn:star2}
Let $G$ and $G'$ be $2$-connected equivalent graphs and let $Z \subseteq V(G)$, where $ |Z|\leq 2$.
Then there exist a w-sequence $\bS_1$ of $G$ and a graph $H$ with a non-crossing w-sequence $\bS_2$ such~that: 
\begin{enumerate}[\;\;\;(1)]
	\item $H=\Wflip[G,\bS_1]$, where $Z\cap\zB_G(X)=\emptyset$, for all $X\in \bS_1$; and
	\item $G'=\Wflip[H,\bS_2]$, where $Z\cap\zB_G(X)\neq\emptyset$, for all $X\in\bS_2$.
\end{enumerate}
\end{PR}
\noindent
Note that we cannot replace $|Z|\leq 2$ by $|Z|\leq k$ for any $k>2$ in the previous proposition, as the following example illustrates.
Suppose that $G$ consists of edges $e_1,e_2,e_3,e_4,e_5$ that form a circuit with edges appearing in that order.
Let $G'$ be the graph obtained from $G$ by rearranging the edges to form a circuit with edges appearing in order $e_1,e_3,e_5,e_2,e_4$.
Suppose that $Z$ consists of $3$ consecutive vertices of the circuit in $G$.
Then every $2$-separation of $G$ contains a vertex of $Z$ but there is no non-crossing w-sequence $\bS$ for which $G'=\Wflip[G,\bS]$.

The other main result of this section is the following.
\begin{PR}\label{conn:startbg}
Consider $2$-connected equivalent graphs $G$ and $G'$ and let $z\in V(G)$ and $z'\in V(G')$.
There exist w-sequences $\bL$ of $G$, $\bL'$ of $G'$ and graphs $H$ and $H'$ such that: 
\begin{enumerate}[\;\;\;(1)]
	\item $H=\Wflip[G,\bL]$, where $z\not\in\zB_G(X)$, for all $X\in\bL$;
	\item $H'=\Wflip[G',\bL']$, where $z'\not\in\zB_{G'}(X)$, for all $X\in\bL'$; and
	\item $H'=\Wflip[H,\bS]$, 
\end{enumerate}
where $\bS$ is a w-star of $H$ with center $z$ and a w-star of $H'$ with center $z'$.
\end{PR}
Recall that w-stars were defined in Section~\ref{sec:split}.
The proofs of Propositions~\ref{conn:star2} and~\ref{conn:startbg} are postponed until Section~\ref{conn:proof1}.
The next section introduces flowers and contains results needed to prove Propositions~\ref{conn:star2} and~\ref{conn:startbg}.
\subsection{Flowers}
For a graph $H$, we say that a partition $\bF=\{B_1,\ldots,B_t\}$ of $E(H)$, with $t\geq 2$, is a {\em flower} if there exist
distinct $u_1,\ldots,u_t \in V(H)$ such that (after possibly relabeling $B_1,\ldots,B_t$):
\begin{enumerate}[\;\;\;(a)]
	\item $H[B_i]$ is connected, for every $i\in[t]$, and
	\item $\zB_H(B_i)=\{u_i,u_{i+1}\}$, for every $i\in[t]$ (where $t+1$ denotes $1$).
\end{enumerate}
\noindent
For $i\in[t]$, $B_i$ (or $H[B_i]$) is a {\em petal} with {\em attachments} $u_i$ and $u_{i+1}$.
We say that the flower is {\em maximal} if no petal has a cut-vertex separating its attachments.
Maximal flowers correspond to generalized circuits as introduced by Tutte in~\cite{Tutte66}.
The term flower was introduced to describe crossing $3$-separations in matroids (see~\cite{OSW}).

Given two partitions $\bF_1$ and $\bF_2$ of the same set, 
we say that $\bF_1$ is a {\em refinement} of $\bF_2$ if every set in $\bF_2$ is the union of sets in $\bF_1$.
Note that, for every flower $\bF$, there is a maximal flower that is a refinement of $\bF$.
Let $\bS_1$ and $\bS_2$ be families of sets over the same ground set. 
We say that $\bS_1$ and $\bS_2$ are {\em independent} if, for every $X\in\bS_1$ and $Y\in\bS_2$, $X$ and $Y$ do not cross.
This definition extends to sequences of sets. Thus we may refer to pairs of independent w-sequences and pairs of independent flowers.
For a graph $H$, we say that a partition $\bF=\{B_1,\ldots,B_t\}$ of $E(H)$, with $t\geq 2$, is a {\em leaflet} if
there exist distinct $u_1,u_2 \in V(H)$ such that:
\begin{enumerate}[\;\;\;(a)]
	\item $H[B_i]$ is connected, for every $i\in[t]$, and
	\item $\zB_H(B_i)=\{u_1,u_2\}$, for every $i\in[t]$.
\end{enumerate}
\begin{RE}\label{rem:easycross}
Let $G$ be a $2$-connected graph and let $X$ and $Y$ be $2$-separations of $G$ that cross. Then $\bF:=\{X\cap Y,X-Y,Y-X,\bar{X}\cap\bar{Y}\}$ is either a flower or a leaflet.
\end{RE}
\noindent
Let $\bF$ be a flower of $G$. 
We say that a $2$-separation $X$ of $G$, where $X$ is the union of petals of $\bF$, is a $2$-separation of $\bF$. The following theorem characterizes pairs of independent flowers.
\begin{LE} \label{conn:idpd-flower}
Let $\bF_1$ and $\bF_2$ be distinct maximal flowers of $G$. The following are equivalent.
\begin{enumerate}[\;\;\;(1)]
	\item $\bF_1$ and $\bF_2$ are independent.
	\item The set of all $2$-separations of $\bF_1$ is independent from the set of all $2$-separations of~$\bF_2$.
	\item There exist petals $B_1$ of $\bF_1$ and $B_2$ of $\bF_2$ such that $\bar{B}_1\subset B_2$ and $\bar{B}_2\subset B_1$.
	\item There is no leaflet $\{B_1,B_2,B_3,B_4\}$ with $\bF_1=\{B_1\cup B_2,B_3\cup B_4\}$ and $\bF_2=\{B_1\cup B_3,B_2\cup B_4\}$. 
\end{enumerate}
\end{LE}
\begin{proof}
It is easy to see that (3) $\Rightarrow$ (2) and that (2) $\Rightarrow$ (1).
Let us show that (1) $\Rightarrow$ (3).
\begin{claim}
For $i=1,2$, no petal of $\bF_i$ can be partitioned into a set $\bS$ of petals of $\bF_{3-i}$ with $|\bS|>1$.
\end{claim}
\begin{cproof}
Suppose for a contradiction that $B_i\in\bF_i$ can be partitioned into a set $\bS$ of petals of $\bF_{3-i}$, where $|\bS|>1$. Let $\bF':=\bS\cup\{\bar{B_i}\}$. 
Then $\bF_{3-i}$ is a refinement of $\bF'$, hence $\bF'$ is a flower. 
It follows that the sets in $\bS$ are petals of $\bF_i$, a contradiction as $\bF_i$ is maximal.
\end{cproof}
\noindent
It follows from the claim that there exists a petal $B_1\in\bF_1$ that is not included in any petal of $\bF_2$ and that there exists a petal $B_2\in\bF_2$ such that $B_2\cap B_1$ and $B_2-B_1$ are non-empty.
As $B_1-B_2$ is non-empty and $B_1$ and $B_2$ do not cross, by (1) we must have that $B_1\cup B_2=E(G)$, i.e. (3) holds.
Let us show that (1) $\Leftrightarrow$ (4).
Clearly, if (4) does not hold then neither does (1).
Suppose (1) does not hold, i.e. some petals $X\in\bF_1$ and $Y\in\bF_2$ cross.
Let $\bF:=\{X\cap Y,X-Y,Y-X,\bar{X}\cap\bar{Y}\}$.
Remark~\ref{rem:easycross} implies that $\bF$ is either a flower or a leaflet.
The former case contradicts the fact that $\bF_1$ is maximal, 
and the latter case shows that (4) does not hold.
\end{proof}
Given sequences $\bS=(S_1,\ldots,S_k)$ and $\bS'=(S'_1,\ldots,S'_r)$ we 
denote by $\bS\odot\bS'$ the concatenated sequence $(S_1,\ldots,S_k,S'_1,\ldots,S'_r)$.
Consider a flower $\bF$ of $G$ and let $\bS$ be a w-sequence $(X_1,\ldots,X_k)$ such that, for every $i\in[k]$, 
$X_i$ is a $2$-separation of the flower $\bF$ in the graph $\Wflip[G,(X_1,\ldots$, $X_{i-1})]$.
We then say that $\bS$ is a w-sequence for the flower $\bF$ of $G$. 
\begin{RE}\label{conn:partialreorder}
Let $\bS$ be a w-sequence of $G$ and suppose that $\bS=\bS_1\odot\bS_2$ for some independent sequences $\bS_1$ and $\bS_2$. 
Let $\bS'$ be obtained from $\bS$ by rearranging the order of sets in $\bS$ such that, for $i=1,2$ and every $X,Y\in\bS_i$, 
if $X$ precedes $Y$ in $\bS_i$ it does so in $\bS'$ as well. 
Then $\bS'$ is a w-sequence of $G$ and $\Wflip[G,\bS]=\Wflip[G,\bS']$.
In particular, if $\bF_1$ and $\bF_2$ are independent flowers and, for $i=1,2$, $\bS_i$ is a w-sequence for flower $\bF_i$, 
then $\Wflip[G,\bS_1\odot\bS_2]=\Wflip[G,\bS_2\odot\bS_1]$.
\end{RE}
\begin{LE}\label{conn:intoflowers}
Let $G$ and $H$ be equivalent $2$-connected graphs.
Then there exists a set of maximal independent flowers $\bF_1,\ldots,\bF_k$ and there exists, 
for each $i\in[k]$, a w-sequence $\bS_i$ of $\bF_i$ such that
\[H=\Wflip[G,\bS_1\odot\ldots\odot\bS_k].\]
\end{LE}
\begin{proof}
Since $G$ and $H$ are equivalent and $2$-connected, there exists a w-sequence $\bS$ of $G$ for which $H=\Wflip[G,\bS]$.
Let us proceed by induction on the cardinality $\ell$ of $\bS$. 
Let $X$ be the last set in $\bS$ and let $\bS'$ be the sequence for which $\bS=\bS'\odot(X)$.
Let $\bF'$ be the maximal flower that refines $\{X,\bar{X}\}$.
If $\ell=1$, then $\bF'$ and $(X)$ are the required flower and corresponding sequence.
Otherwise, by induction, there exists a set of maximal independent flowers $\bF_1,\ldots,\bF_r$ and there exists, 
for each $i\in[r]$, a w-sequence $\bS_i$ of $\bF_i$ such that
$H=\Wflip[G,\bS_1\odot\ldots\odot\bS_r\odot(X)]$.
Suppose that $\bF'=\bF_i$ for some $i\in[r]$. 
Because of Remark~\ref{conn:partialreorder}, we may assume that $\bF'=\bF_r$.
Then $\bF_1,\ldots,\bF_r$ and $\bS_1,\ldots,\bS_r\odot(X)$ are the required flowers and corresponding w-sequences.
Thus we may assume that $\bF'$ is distinct from $\bF_i$, for all $i\in[r]$.
Suppose that $\bF'$ is independent from $\bF_1,\ldots,\bF_r$. 
Then $\bF_1,\ldots,\bF_r,\bF'$ and $\bS_1,\ldots,\bS_r,(X)$ are the required flowers and corresponding w-sequences.
Hence, we may assume that for some $i\in[r]$, $\bF'$ and $\bF_i$ are not independent.
Because of Remark~\ref{conn:partialreorder}, we may assume that $\bF'$ and $\bF_r$ are not independent.
It follows from Lemma~\ref{conn:idpd-flower} that there exists a leaflet $\{B_1,B_2,B_3,B_4\}$ of $H':=\Wflip[G,\bS_1\odot\ldots\odot\bS_{r-1}]$,
where $\bF_r=\{B_1\cup B_2,B_3\cup B_4\}$ and $\bF'=\{B_1\cup B_3,B_2\cup B_4\}$. 
Hence $\bS_r=(B_1\cup B_2)$ and $X=B_1\cup B_3$. It follows that $\Wflip[H',\bS_r\odot(X)]=\Wflip[H',(B_2\cup B_3)]$.
Then $\bF_1,\ldots,\bF_r$ and $\bS_1,\ldots,\bS_{r-1},(B_2\cup B_3)$ are the required flowers and corresponding w-sequences.
\end{proof}
\noindent
Let $G$ be a graph and let $\bF=\{B_1,\ldots,B_t\}$ be a flower of $G$. 
If $H=\Wflip[G,(B_i)]$ for some $i\in[t]$, then we say that $H$ is obtained from $G$ by {\em reversing} petal $B_i$. 
We say that distinct petals $B_i$ and $B_j$ are {\em consecutive} in $\bF$ if $V_G(B_i)\cap V_G(B_j)\neq\emptyset$.
\begin{LE}\label{conn:4petals}
Let $\bF$ be a flower of a graph $G$ and let $B_1,B_2,B_3,B_4$ be petals of $\bF$.
We can find a non-crossing w-sequence $\bS$ of $\bF$ such that, for $H:=\Wflip[G,\bS]$,
both $B_1,B_2$ and $B_3,B_4$ are consecutive petals of $\bF$ in $H$.
\end{LE}
\begin{proof}
There exists a flower $\bF'=\{B'_1,B'_2,B'_3,B'_4\}$ such that:
\begin{itemize}
	\item $\bF$ is a refinement of $\bF'$;
	\item $B_i\subseteq B'_i$ for $i=1,2,3,4$;
	\item $\zB_G(B_i)\cap \zB_G(B'_i)\neq\emptyset$.
\end{itemize}
Since $\bF'$ has only $4$ petals, there is a non-crossing w-sequence $\bS'$ of $\bF'$ such that, for $H':=\Wflip[G,\bS']$, 
$B'_1,B'_2,B'_3,B'_4$ appear consecutively in $H'$. 
As $H$ can be obtained from $H'$ by possibly reversing some of the petals of $\bF'$, the result follows.
\end{proof}
%
\subsection{Proof of Propositions~\ref{conn:star2} and~\ref{conn:startbg}}\label{conn:proof1}
%
\begin{LE}\label{conn:rearrangeflower}
Let $\bF$ be a flower of $G$, let $\bL$ be a w-sequence for flower $\bF$, and let $H=\Wflip[G,\bL]$.
Consider $Z\subseteq V(G)$, where $|Z|\leq 2$. Then there exists a w-sequence $\bL'\odot\bL''$ of $G$ such that:
\begin{enumerate}[\;\;\;(1)]
	\item $H=\Wflip[G,\bL'\odot\bL'']$;
	\item $Z\cap\zB_G(X)=\emptyset$, for all $X\in\bL'$;
	\item $\bL''$ is non-crossing.
\end{enumerate}
\end{LE}
\begin{proof}
We only consider the case where $Z=\{z_1,z_2\}$ and where both $z_1$ and $z_2$ are attachments of $\bF$ in $G$, as the other cases are similar.
For $i=1,2$, there exist consecutive petals $B_i$ and $B'_i$ in $G$ such that $z_i\in\zB_G(B_i)\cap \zB_G(B'_i)$.
Note that $H$ is obtained from $G$ by first permuting the petals of $\bF$ and then by reversing a subset of the petals.
Since the petals are $2$-separations that do not cross any $2$-separation of $\bF$, we may assume that $H$ is obtained from $G$ by only permuting the petals of $\bF$.
It follows from Lemma~\ref{conn:4petals} that there is a non-crossing w-sequence $\bL''$ of $H$ such that, in $H':=\Wflip[H,\bL'']$, $B_1,B'_1$ and $B_2,B'_2$ are consecutive.
Moreover, we can assume (by possibly reversing petals) that $z_i\in\zB_{H'}(B_i)\cap\zB_{H'}(B'_i)$, for $i=1,2$.
Let $\bF'$ be the flower obtained from $\bF$ by replacing, for $i=1,2$, petals $B_i$ and $B'_i$ by a unique petal $B_i\cup B'_i$.
Then let $\bL'$ be a w-sequence for flower $\bF'$ such that $\Wflip[G,\bL']=H'$.
\end{proof}
\noindent
We are now ready for the proof of the first main result.
\begin{proof}[\textbf{Proof of Proposition~\ref{conn:star2}}]
We say that a set of sequences $\bS_1,\bS_2,\bL$ satisfies property (P) if there exist graphs $H$ and $H'$, where $\bS_1$, $\bS_2$ and $\bL$ are w-sequences of $G$, $H'$ and $H$ respectively, and
\begin{enumerate}[\;\;\;(1')]
	\item $H=\Wflip[G,\bS_1]$, where $Z\cap\zB_G(X)=\emptyset$, for all $X\in \bS_1$;
	\item $H'=\Wflip[H,\bL]$;
	\item $G'=\Wflip[H',\bS_2]$ and $\bS_2$ is non-crossing.
\end{enumerate}
As we can choose $\bS_1=\bS_2=\emptyset$ and since $G$ and $G'$ are equivalent, a set of sequences $\bS_1,\bS_2,\bL$ satisfying property (P) exists. 
Lemma~\ref{conn:intoflowers} implies that there exist maximal independent flowers $\bF_1,\ldots,\bF_k$ and there exists, for all $i\in[k]$, a w-sequence $\bL_i$ for $\bF_i$ such that $H'=\Wflip[H,\bL_1\odot\cdots\odot\bL_k]$. 

Among all choices of $\bS_1,\bS_2,\bL_1,\ldots,\bL_k$ where $\bS_1,\bS_2,\bL_1\odot\cdots\odot\bL_k$ satisfy property (P), choose one that minimizes~$k$. 
Suppose $k>0$.
Apply Lemma~\ref{conn:rearrangeflower} to the sequence $\bL_1$ and let $\bL'_1$ and $\bL''_1$ 
correspond to $\bL'$ and $\bL''$ in the statement of the lemma.
Define,
\[
\hat{\bS}_1:=\bS_1\odot\bL_1'\quad\quad
\hat{\bS}_2:=\bL''_1\odot\bS_2.
\]
Since flowers $\bF_1,\ldots,\bF_k$ of $H$ are independent, $\bL''_1$ is independent from $\bL_2\odot\cdots\odot\bL_k$ (see Proposition~\ref{conn:idpd-flower}).
Therefore, by Remark~\ref{conn:partialreorder},
\[
G'=\Wflip[G,\hat{\bS}_1\odot\bL_2\odot\cdots\odot\bL_k\odot\hat{\bS}_2].
\]
Then $\hat{\bS}_1,\hat{\bS}_2,\bL_2\odot\cdots\odot\bL_k$ contradict our choice of $\bS_1,\bS_2,\bL_1,\bL_2$.
Thus $k=0$. 
Note that, if $Z\cap\zB_G(X)=\emptyset$ for some $X\in\bS_2$, then we can redefine $\bS_1$ to be 
$\bS_1\odot(X)$ and $\bS_2$ to be $\bS_2-\{X\}$.
Hence we may assume that $Z\cap\zB_G(X)\neq\emptyset$ for all $X\in\bS_2$ and the result follows.
\end{proof}
\begin{RE}\label{conn:nested}
Let $G$ be a graph, let $\bS$ be a non-crossing w-sequence of $G$, and let $G'=\Wflip[G,\bS]$.
Suppose that there exist $X_1,X_2,X_3\in\bS$, where $X_1\subset X_2\subset X_3$ and $X_1,X_2,X_3$ have distinct boundaries.
Suppose that for some vertex $z$ we have $z\in\zB_G(X_i)$, for $i=1,2,3$. 
Then $\zB_{G'}(X_1)\cap\zB_{G'}(X_2)\cap\zB_{G'}(X_3)=\emptyset$.
\end{RE}
\begin{proof}
Because of Remark~\ref{conn:laminar}, we may assume that $X_1,X_2,X_3$ appear first in $\bS$.
Let $H:=\Wflip[G,$ $(X_1,X_2,X_3)]$. Then $\zB_H(X_1)\cap\zB_H(X_2)\cap\zB_H(X_3)=\emptyset$.
The result now follows as $\bS$ is non-crossing.
\end{proof}
\begin{LE}\label{conn:makestar}
Let $H$ and $H'$ be equivalent graphs with $H'=\Wflip[H,\bS]$ for some non-crossing w-sequence $\bS$.
Suppose that there exist vertices $z$ in $V(H)$ and $z'$ in $V(H')$ such that $z \in \zB_H(X)$ and $z' \in \zB_{H'}(X)$
for every $X \in \bS$.
Then $H'=\Wflip[H,\bS']$ for some $\bS'$ which is a w-star of $H$ with center $z$ and a w-star of $H'$ with center $z'$.
\end{LE}
\begin{proof}
Note that we may swap any $X$ in $\bS$ with its complement and maintain the properties of $\bS$.
Since $\bS$ is non-crossing we may assume (after possibly replacing some sets $\bS$ by their complement) that $\bS$ is laminar,
i.e. every two sets in $\bS$ are either disjoint or one contains the other.
First suppose there exist $X_1,X_2 \in \bS$ with $\zB_H(X_1)=\zB_H(X_2)$.
Then we may remove $X_1,X_2$ from $\bS$ and add $X_1 \tr X_2$.
This keeps the w-sequence non-crossing and gives rise to the same graph $H'$.
Hence we may assume that, for every $X_1,X_2 \in \bS$, $\zB_H(X_1) \cap \zB_H(X_2) =\{z\}$
and $\zB_{H'}(X_1) \cap \zB_{H'}(X_2) =\{z'\}$, and condition (b) in the definition of w-star holds.
Suppose that for some $X_1,X_2\in\bS$ we have $X_1\subset X_2$. 
By Remark~\ref{conn:nested}, there is no set $X_3\in\bS$ where $X_3\supseteq X_2$ or $\bar{X}_3\supseteq X_2$. 
After replacing $X_2$ by $\bar{X}_2$ the sets in $\bS$ satisfy condition (a) of the definition of w-stars. 
Finally, if any $X\in\bS$ contains an edge $e$ where the ends of $e$ are $\zB_H(X)$,
we may replace $X$ by $X-\{e\}$. Then property (c) of w-stars holds.
\end{proof}
We are now ready for the proof of the second main result.
\begin{proof}[\textbf{Proof of Proposition~\ref{conn:startbg}}]
Proposition~\ref{conn:star2} implies that there exist a w-sequence $\bL$ of $G$ and a graph $H$ with a non-crossing sequence $\bS_0$ such that (1) holds in the statement of the proposition, $G'=\Wflip[H,\bS_0]$ and $z\in\zB_H(X)$ for all $X\in\bS_0$. Because of Remark~\ref{conn:laminar}, we can view $\bS_0$ as a set. Hence, $H=\Wflip[G',\bS_0]$. Let $\bL'=\{X\in\bS_0:z'\not\in\zB_{G'}(X)\}$ and let $\bS_1:=\bS_0-\bL'$. Let $H':=\Wflip[H,\bS_1]$. Then condition (2) in the statement of the proposition holds.
Finally, by Lemma~\ref{conn:makestar}, there exists a w-sequence $\bS$ for $H$ that is a w-star of $H$ with center $z$ and a w-star of $H'$ with center $z'$ and such that $H'=\Wflip[H,\bS]$.
\end{proof}
%


\section{Proof of Theorem~\ref{split:char} - split siblings}\label{sec:proofs1}
Recall the definition of split-templates given at the end of Section~\ref{sec:Shih}.
We say that split-templates:
\begin{equation}\label{split:eq1}
\dst=(H_1,v_1,\alpha_1,H_2,v_2,\alpha_2,\bS)
\quad\mbox{and}\quad
\dst'=(H'_1,v_1',\alpha_1',H'_2,v'_2,\alpha'_2,\bS')
\end{equation}
are {\em compatible} if: 
\begin{enumerate}[\;\;\;(a)]
	\item $H_i$ and $H'_i$ are equivalent, for $i=1,2$, and
	\item $\alpha_i\tr\alpha'_i$ forms a cut of $H_1$, for $i=1,2$.
\end{enumerate}
Note that in this case, by Theorem~\ref{intro:whitney}, $\cut(H_1)=\cut(H_2)=\cut(H'_1)=\cut(H'_2)$.
\begin{LE}\label{split:eqvtemplates}
Let $\dst$ and $\dst'$ be compatible split-templates.
Let $(G_1,\Sigma_1)$ and $(G_2,\Sigma_2)$ be the siblings arising from $\dst$ and
$(G'_1,\Sigma'_1)$, $(G'_2,\Sigma'_2)$ be the siblings arising from $\dst'$.
Then, for $i=1,2$, $(G_i,\Sigma_i)$ and $(G'_i,\Sigma'_i)$ are equivalent.
\end{LE}
\begin{proof}
Let us assume that $\dst$ and $\dst'$ are as described in~\eqref{split:eq1}.
Then, by construction, 
\[
\cut(G_1)=\spa\bigl(\cut(H_1)\cup\{\alpha_1\}\bigr)
\quad\mbox{and}\quad
\cut(G'_1)=\spa\bigl(\cut(H_1)\cup\{\alpha'_1\}\bigr).
\]
By hypothesis, $\alpha_1\tr\alpha'_1\in\cut(H_1)$.
Hence, $\cut(G_1)=\cut(G'_1)$.
It follows from Theorem~\ref{intro:whitney} that $G_1$ and $G'_1$ are equivalent.
Similarly, $G_2$ and $G'_2$ are equivalent.
It follows that $(G'_1,\Sigma_1)$ and $(G'_2,\Sigma_2)$ are siblings.
As the matching signature pair for $G'_1$ and $G'_2$ is unique up to signature exchange,
$(G_i,\Sigma_i)$ and $(G'_i,\Sigma'_i)$ are equivalent, for $i=1,2$.
\end{proof}
\begin{LE}\label{split:templatekey}
Every split-template has a compatible split-template which is simple or nova.
\end{LE}
\begin{proof}
Suppose that $\dst:=(H_1,v_1,\alpha_1,H_2,v_2,\alpha_2,\bS)$ is a split-template.
\begin{claim}
There is a template $(H'_1,v_1,\alpha_1,H'_2,v_2,\alpha_2,\bS')$ which is compatible 
with $\dst$ and has the property that $\bS'$ is a w-star of $H'_1$ and $H'_2$.
\end{claim}
\begin{cproof}
The proof follows easily from Proposition~\ref{conn:startbg},
since $H_1$ and $H_2$ are equivalent.
\end{cproof}
\noindent
Choose a split-template $\dst'=(H'_1,v'_1,\alpha'_1,H'_2,v'_2,\alpha'_2,\bS')$ with the following properties:
\begin{enumerate}[\;\;\;(M1)]
	\item $\dst'$ is compatible with $\dst$;
	\item for $i=1,2$, $\bS'$ is a w-star of $H'_i$ with center $v'_i$;
	\item $|\bigcup\{X: X\in \bS'\}|$ is minimized among all split-templates satisfying (M1) and (M2).
\end{enumerate}
Such a split-template exists because of Claim~1. 
We may assume that $\bS'\neq\emptyset$ for otherwise $\dst'$ is simple and we are done. 
We will show that $\dst'$ is nova. As (N1) (from the definition of nova) holds, it suffices to prove (N2).
Let $X'\subseteq X\in \bS'$, where 
$\zB_{H'_1}(X')=\zB_{H'_1}(X)=\{v'_1,w\}$, for some vertex $w$.
Let us assume that we chose $X'$ to be an inclusion-wise minimal subset with that property.
It suffices to show for (N2) (as we can interchange the role of $H'_1$ and $H'_2$) 
that there exist $\{v'_1,w\}$-handcuffs included in $X'$ in $(H'_1,\alpha'_1\tr\alpha'_2)$.
\begin{claim}
None of the following holds:
\begin{enumerate}[\;\;\;(1)]
	\item $\delta_{H'_1}(v'_1) \cap X' \cap \alpha_1'$ is empty; 
	\item $(\delta_{H'_1}(v'_1) \cap X') - \alpha_1'$ is empty;
	\item we can partition $X'$ into $Z$ and $Z'$ such that
		$\zB_{H'_1}(X')=\zB_{H'_1}(Z)=\zB_{H'_1}(Z')$ and $\alpha'_1\cap X'=\delta_{H_1'}(v_1')\cap Z$.
\end{enumerate}
\end{claim}
\begin{cproof}
Define
\[
D:=
\begin{cases} 
\emptyset & \text{if (1) holds}\\
\delta_{H'_1}(v'_1) & \text{if (2) holds}\\
\delta_{H'_1}(\zI_{H_1'}(Z)) & \text{if (3) holds.}
\end{cases} 
\]

\vspace{0.1in}\noindent
Let $\tilde{\alpha}=\alpha'_1\tr D$, let $\tilde{H}:=\Wflip[H'_1,(X')]$ and let 
$\tilde{\bS}=\bS'-\{X\} \cup \{X-X'\}$.
There is a vertex $\tilde{v}$ of $\tilde{H}$ where 
$\delta_{\tilde{H}}(\tilde{v})\supseteq\tilde{\alpha}$.
Since $\bS$ is non-crossing, $H'_2=\Wflip[\tilde{H},\tilde{\bS}]$.
Hence, (M2) holds for $\tilde{\dst}:=(\tilde{H},\tilde{v},\tilde{\alpha},$ $H'_2,v'_2,\alpha'_2,\tilde{\bS})$.
Since $D$ is a cut of $H'_1$, (M3) holds for $\tilde{\dst}$.
As $|\bigcup\{X:X\in \tilde{S}\}|<|\bigcup\{X:X\in \bS'\}|$, this contradicts our choice (M3).
\end{cproof}
\begin{claim}
There exists a circuit $C\subseteq X'$ of $H'_1$ avoiding $w$ with $|C\cap\alpha'_1|$ odd.
\end{claim}
\begin{cproof}
We claim that otherwise (1), (2), or (3) of Claim~2 must hold, giving a contradiction.
Let $G$ be the graph obtained from $H'_1[X']$ by splitting $v'_1$ into $v'^+_1,v'^-_1$ according to $\alpha'_1$. 
Every $(v'^-_1,v'^+_1)$-path $P$ of $H[X']$ avoiding $w$ is a required circuit. 
Hence, we may assume that no such path exists. 
It follows that $w$ is a cut-vertex separating $v'^-_1$ and $v'^+_1$ in $G[X']$. 
Let $Z,Z'$ be the partition of $X'$ such that $V_{G[X']}(Z) \cap V_{G[X']}(Z')=\{w\}$ and
$v'^-_1\in G[Z]$, $v'^+_1\in G[Z']$. Then (3) holds.
\end{cproof}
\noindent
By Claim~3 and by reversing the role of $H'_1$ and $H'_2$, we deduce that 
there exists an odd circuit $C_1$ (respectively $C_2$) included in $X'$ using $v'_1$ (respectively $w$) 
and avoiding $w$ (respectively $v'_1$). 
Consider first the case where $C_1$ and $C_2$ have at least one common vertex in $H'_1$.
As $\alpha'_1\subseteq\delta_{H'_1}(v'_1)$ and $\alpha'_2\subseteq\delta_{H'_1}(w)$,
we may assume, after possibly redefining $C_1$, that $C_1$ and $C_2$ intersect in exactly one
vertex or intersect in a path. Hence, in that case $(C_1,C_2,\emptyset)$ form $\{v'_1,w\}$-handcuffs
included in $X'$ in $(H'_1,\alpha'_1\tr\alpha'_2)$, as required.
Consider now the case where $C_1$ and $C_2$ have no common vertex in $H'_1$.
As $X'$ was selected to be inclusion-wise minimal, there exists a path $P\subset X'$ joining $C_1$ and $C_2$ which avoids $v'_1$ and $w$. 
For an inclusion-wise minimal such $P$, $(C_1,C_2,P)$ form $\{v'_1,w\}$-handcuffs in $(H'_1,\alpha'_1\tr\alpha'_2)$ as required.
\end{proof}
\begin{proof}[\textbf{Proof of Theorem~\ref{split:char}}]
By definition, $(G_1,\Sigma_1)$ and $(G_2,\Sigma_2)$ arise from a split-template 
$(H_1,v_1,\alpha_1,$ $H_2,v_2,\alpha_2)$. 
Let $T_1,T_2$ be the matching terminal pair for $G_1$ and $G_2$.
Proposition~\ref{connectivity:n3c} implies that $G_1$ and $G_2$ are $2$-connected, except for possible loops.
Consider first the case where $H_1 \setminus \loops(H_1)$ is not $2$-connected.
Then for some $X\subseteq E(G_1)$, $\zB_{G_1}(X)=T_1$.
It follows, from the argument in Section~\ref{sec:split:reduction}, that $(G_1,\Sigma_1)$ and $(G_2,\Sigma_2)$ can be reduced. 
Hence we may assume that $H_1$ is $2$-connected, except for possible loops, and so is $H_2$.
It follows that $\dst=(H_1,v_1,\alpha_1,H_2,v_2,\alpha_2,\bS)$ is a split-template for some w-sequence $\bS$ of $H_1$,
where $H_2=\Wflip[H_1,\bS]$. 
Lemma~\ref{split:templatekey} implies that there exists a split-template $\dst'$ which is simple or nova and compatible with $\dst$. 
Let $(G'_1,\Sigma'_1)$ and $(G'_2,\Sigma'_2)$ arise from $\dst'$.
By definition $(G'_1,\Sigma'_1)$ and $(G'_2,\Sigma'_2)$ are simple twins or nova twins. 
By Lemma~\ref{split:eqvtemplates}, for $i=1,2$, $(G'_i,\Sigma'_i)$ is equivalent to $(G_i,\Sigma_i)$.
It follows that $(G_1,\Sigma_1)$ and $(G_2,\Sigma_2)$ are simple or nova siblings.
\end{proof}
\section{Proof of Theorem~\ref{doublesplit:char} - quad siblings}\label{sec:proofs2}
Similarly to the proof for split siblings, we define compatible quad-templates.
The different types of quad siblings arise from different types of templates.
%
\subsection{Templates}
%
\begin{RE}\label{iso:alphabetasign}
Suppose that $\dst=(H_1,v_1,w_1,\alpha_1,\beta_1,H_2,v_2,w_2,\alpha_2,\beta_2)$ is a quad-template and $(G_1,\Sigma_1)$ and
$(G_2,\Sigma_2)$ are the quad siblings arising from $\dst$. 
Then $\alpha_{3-i}$ and $\beta_{3-i}$ are signatures of $(G_i,\Sigma_i)$, for $i=1,2$.
\end{RE}
\begin{proof}
For $i=1,2$, vertex $v_i$ of $H_i$ gets split into vertices $v^-_i,v^+_i$ of $G_i$.
By construction, $\alpha_i=\delta_{G_i}(v^-_i)$, for $i=1,2$.
As $v^-_1\in T_1$, Theorem~\ref{isopair-co} implies that $\alpha_1$ is a signature of $(G_2,\Sigma_2)$. 
Similarly $\beta_1$ is a signature of $(G_2,\Sigma_2)$.
By symmetry, $\alpha_2$ and $\beta_2$ are signatures of $(G_1,\Sigma_1)$.
\end{proof}
We say that two quad-templates:
\begin{eqnarray}
& \dst=(H_1,v_1,w_1,\alpha_1,\beta_1,H_2,v_2,w_2,\alpha_2,\beta_2) & \nonumber\\
& \mbox{and} & \label{doublesplit:eq1}\\
& \dst'=(H_1',v_1',w_1',\alpha_1',\beta_1',H_2',v_2',w_2',\alpha_2',\beta_2') & \nonumber
\end{eqnarray}
are {\em compatible} if, for $i =1,2$:
\begin{enumerate}[\;\;\;(a)]
	\item $H_i$ is equivalent to $H_i'$;
	\item $\alpha_i \Delta \alpha_i'$ is a cut of $H_1$;
	\item $\beta_i \Delta \beta_i'$ is a cut of $H_1$.
\end{enumerate}
Note that, by Theorem~\ref{intro:whitney}, $\cut(H_1)=\cut(H_2)=\cut(H'_1)=\cut(H'_2)$.
\begin{LE}\label{doublesplit:le1}
Let $\dst$ and $\dst'$ be compatible quad-templates.
Let $(G_1,\Sigma_1)$, $(G_2,\Sigma_2)$ and $(G_1',\Sigma_1')$, $(G_2',\Sigma_2')$
be quad siblings arising from $\dst$ and $\dst'$ respectively.
Then $(G_i,\Sigma_i)$ and $(G_i',\Sigma_i')$ are equivalent, for $i=1,2$.
\end{LE}
\begin{proof}
Let $\dst$ and $\dst'$ be compatible quad-templates defined as in~\eqref{doublesplit:eq1}.
Fix $i \in [2]$.
Let $v_i^-,v_i^+ \in V(G_i)$ be obtained by splitting $v_i$ according to $\alpha_i$ in $H_i$.
We first show that $G_i$ is equivalent to $G_i'$ by showing that $\cut(G_i)=\cut(G_i')$.
Let $C:=\alpha_i \Delta \alpha_i'$.
By definition of compatible templates, $C$ is a cut of $H_i'$, hence it is a cut of $G_i'$.
Moreover, by construction, $\delta_{G_i}(v_i^-)=\alpha_i$. Thus $\delta_{G_i}(v_i^-)=C \Delta \alpha_i'$
is a cut of $G_i'$. Similarly, we can show that $\delta_{G_i}(v_i^+)$ is a cut of $G_i'$.
By symmetry between $v_i$ and $w_i$, we have that $\delta_{G_i}(w_i^-)$ and $\delta_{G_i}(w_i^+)$ are cuts of $G_i'$.
Moreover, for every $u \in V(G_i)$, if $u \notin \{v_i^-,v_i^+,w_i^-,w_i^+\}$, then $\delta_{G_i}(u)$
is a cut of $H_i$, hence a cut of $H_i'$ and a cut of $G_i'$.
Thus $\delta_{G_i}(u)$ is a cut of $G_i'$ for every $u \in V(G_i)$. 
As the cuts of $G_i$ are generated by its fundamental cuts (i.e. the cuts of the form $\delta_{G_i}(u)$, for $u \in V(G_i)$), 
this shows that $\cut(G_i) \subseteq \cut(G_i')$. 
By symmetry between $\dst$ and $\dst'$, we have that $\cut(G_i')=\cut(G_i)$, thus $G_i$ and $G_i'$ are equivalent.
As $G_i$ is equivalent to $G_i'$, $\Sigma_1,\Sigma_2$ is a matching signature pair for $G_1',G_2'$.
By Theorem~\ref{isopair-co}, the matching signature pair is unique up to resigning,
thus $(G_i,\Sigma_i)$ and $(G_i',\Sigma_i')$ are equivalent.
\end{proof}
Let $(H_1,v_1,w_1,\alpha_1,\beta_1,H_2,v_2,w_2,\alpha_2,\beta_2)$ be a quad-template.
If $H_1$ and $H_2$ are $2$-connected, except for possible loops, we have that $H_2=\Wflip[H_1,\bS]$ for some w-sequence $\bS$.
We abuse terminology slightly and say that $(H_1,v_1,w_1,\alpha_1,\beta_1,H_2,v_2,w_2,\alpha_2,\beta_2,\bS)$ is a {\em quad-template}.
This is only defined for the case where $H_1$ and $H_2$ are $2$-connected up to loops.
Thus, when specifying a w-sequence $\bS$ for a quad-template, we will implicitly assume
that graphs are connected, except for possible loops.

Consider a template $\dst=(H_1,v_1,w_1,\alpha_1,\beta_1,H_2,v_2,w_2,\alpha_2,\beta_2,\bS)$, 
where $\bS=(X_1,\ldots$, $X_k)$ for some $k \geq 0$ and $X_i \neq \emptyset$, for every $i \in [k]$.
We say that $\dst$ is of {\em type I} if:
\begin{enumerate}[\;\;\;({TI}a)]
	\item $X_i \cap X_j = \emptyset$, for all distinct $i,j \in [k]$;
	\item $H_i[X_j] \setminus \zB_{H_i}(X_j)$ is non-empty and connected, for every $i=1,2$ and $j\in [k]$;
	\item $\zB_{H_i}(X_j)=\{v_i,w_i\}$, for $i=1,2$ and $j \in [k]$.
\end{enumerate}

\noindent
We say that $\dst$ is of {\em type II} if:
\begin{enumerate}[\;\;\;({TII}a)]
	\item $k=1$ or $k=2$;
	\item if $k=2$, $X_1$ is disjoint from $X_2$;
	\item $v_i \in  \zB_{H_i}(X_j)$, for $i=1,2$ and $j \in [k]$;
	\item $w_1 \in \zI_{H_1}(X_1)$;
	\item if $k=1$, $w_2 \in \zI_{H_2}(\bar{X}_1 - \loops(H_2))$;
	\item if $k=2$, $w_2 \in \zI_{H_2}(X_2)$.
\end{enumerate}
%
\subsection{The proof}
A signed graph $(G,\Sigma)$ is {\em ec-standard} if $\ecycle(G,\Sigma)$ is $3$-connected
and, for every $(G',\Sigma')$ equivalent to $(G,\Sigma)$, $(G',\Sigma')$ does not contain a blocking vertex.
To prove Theorem~\ref{doublesplit:char} we require the following four results, 
which will be proved at the end of the chapter.
\begin{LE}\label{doublesplit:le2}
Suppose that $(G_1,\Sigma_1)$ and $(G_2,\Sigma_2)$ are quad siblings arising from a quad-template $\dst$ of type I. 
Suppose that $\ecycle(G_1,\Sigma_1)$ is $3$-connected.
Then $(G_1,\Sigma_1)$ and $(G_2,\Sigma_2)$ are either shuffle, tilt or twist twins.
\end{LE}
\begin{LE}\label{doublesplit:le3}
Suppose that $(G_1,\Sigma_1)$ and $(G_2,\Sigma_2)$ are $\Delta$-irreducible ec-standard quad siblings arising from a 
quad-template $\dst$ of type II. Then $(G_1,\Sigma_1)$ and $(G_2,\Sigma_2)$ are either widget or gadget twins.
\end{LE}
\begin{LE}\label{doublesplit:le4}
Suppose that $(G_1,\Sigma_1)$ and $(G_2,\Sigma_2)$ are $\Delta$-irreducible ec-standard quad siblings arising from a 
quad-template $\dst=(H_1,v_1,w_1,\alpha_1,\beta_1,H_2,v_2,w_2,\alpha_2,\beta_2,\bS)$. 
Then there exists a template $\dst'$ which is compatible with $\dst$ and is of type I or type II.
\end{LE}
\begin{LE}\label{doublesplit:le0}
Let $\dst=(H_1,v_1,w_1,\alpha_1,\beta_1,H_2,v_2,w_2,\alpha_2,\beta_2)$ be a quad-template. 
Let $(G_1,\Sigma_1)$ and $(G_2,\Sigma_2)$ be the quad siblings arising from $\dst$.
If $(G_1,\Sigma_1)$ and $(G_2,\Sigma_2)$ are ec-standard and $\Delta$-irreducible,
then either $(G_1,\Sigma_1)$ and $(G_2,\Sigma_2)$ are shuffle, tilt, twist, gadget or widget siblings or
$H_1$ and $H_2$ are $2$-connected, except for the possible presence of loops.
\end{LE}
\begin{proof}[\textbf{Proof of Theorem~\ref{doublesplit:char}}]
Let $M$ be a $3$-connected non-graphic matroid and $(G_1,\Sigma_1)$ and $(G_2,\Sigma_2)$ be quad siblings representing $M$.
By Remark~\ref{intro:projection}, $(G_1,\Sigma_1)$ and $(G_2,\Sigma_2)$ are ec-standard.
If they are $\Delta$-reducible we are done. 
Thus in the remainder of the proof we will assume that $(G_1,\Sigma_1)$ and $(G_2,\Sigma_2)$ are $\Delta$-irreducible quad siblings.
Suppose that they arise from a quad-template $(H_1,v_1,w_1,\alpha_1,\beta_1,H_2,v_2,w_2,\alpha_2,\beta_2)$.
By Lemma~\ref{doublesplit:le0},  either $(G_1,\Sigma_1)$ and $(G_2,\Sigma_2)$ fall into one
of the cases $(1) - (5)$ in the statement of the theorem, or $H_1$ and $H_2$ are $2$-connected, except for
the presence of loops.
Therefore we may assume that $H_2=\Wflip[H_1,\bS]$ for some w-sequence $\bS$ of $H_1$ and
$(G_1,\Sigma_1)$ and $(G_2,\Sigma_2)$ arise from a quad-template $\dst$ with w-sequence $\bS$.
By Lemma~\ref{doublesplit:le4}, there exists a quad-template $\dst'$
compatible with $\dst$ which is of type I or of type II.
Let $(G_1',\Sigma_1')$ and $(G_2',\Sigma_2')$ be the quad siblings arising from $\dst'$.
If $\dst'$ is of type I then, by Lemma~\ref{doublesplit:le2},  $(G_1',\Sigma_1')$ and $(G_2',\Sigma_2')$
are shuffle, tilt or twist siblings.
If $\dst'$ is of type II then, by Lemma~\ref{doublesplit:le3},  $(G_1',\Sigma_1')$ and $(G_2',\Sigma_2')$
are widget or gadget twins.
Finally, by Lemma~\ref{doublesplit:le1}, $(G_i,\Sigma_i)$ and $(G_i',\Sigma_i')$ are equivalent for $i=1,2$.
Therefore the result follows.
\end{proof}

The proofs of Lemma~\ref{doublesplit:le2}, Lemma~\ref{doublesplit:le3},
Lemma~\ref{doublesplit:le4} and Lemma~\ref{doublesplit:le0} are given in Section~\ref{sec:lemmaProof}.
First we require some technical results.
\subsection{Technical lemmas}
Recall that signed graph $(G,\Sigma)$ is bipartite if $G$ doesn't contain any $\Sigma$-odd cycle;
moreover, a set $X$ is a $3$-$(0,1)$-separation if $X$ is a $3$-separation of $G$ such that 
$(G[X],\Sigma \cap X)$ is bipartite and $(G[\bar{X}],\Sigma-X)$ is non-bipartite.
\begin{LE}\label{doublesplit:clean}
Let $(G_1,\Sigma_1)$ and $(G_2,\Sigma_2)$ be ec-standard siblings.
Let $X$ be a $3$-$(0,1)$-separation in both $(G_1,\Sigma_1)$ and $(G_2,\Sigma_2)$.
Then $(G_1,\Sigma_1)$ and $(G_2,\Sigma_2)$ are $\Delta$-reducible.
\end{LE}
\begin{proof}
Let $\zB_{G_1}(X)=\{u_1,u_2,u_3\}$ and $\zB_{G_2}(X)=\{u'_1,u'_2,u'_3\}$. 
We claim that we can relabel the vertices in $\zB_{G_1}(X)$
so that every $(u_i,u_j)$-path in $G_1[X]$ is a $(u_i,'u_j')$-path in $G_2[X]$,
for every choice of $i,j \in [3]$ with $i\neq j$.
Consider distinct $i,j \in [3]$. Let $P$ be a $(u_i,u_j)$-path in $G_1[X]$.
Let $Q$ be a $(u_i,u_j)$-path in $G_1[\bar{X}]$ of the same parity as $P$.
Note that some such $Q$ exists as $\ecycle(G_1,\Sigma_1)$ is $3$-connected and $(G_1[\bar{X}],\Sigma_1 \cap \bar{X})$ is non-bipartite.
By the choice of $Q$, $C:=P \cup Q$ is an even circuit of $\ecycle(G_1,\Sigma_1)$, hence a circuit of $\ecycle(G_2,\Sigma_2)$.
As $(G_2[X],\Sigma_2 \cap X)$ is bipartite, every cycle in $G_2[X]$ is even, 
hence $G_2[P]$ does not contain any cycle. 
Therefore $P$ is a $(u_s',u_t')$-path in $G_2[X]$ for some $s,t \in [3]$ with $s \neq t$. 
The same argument holds for every choice of distinct $i,j \in [3]$.
Now, if $P_1$ is a $(u_i,u_j)$-path and $P_2$ is a $(u_j,u_h)$-path in $G_1[X]$, for distinct $i,j,h \in [3]$,
then $P_1$ is a $(u_s',u_t')$-path in $G_2[X]$ and $P_2$ is a $(u_q',u_r')$-path in $G_2[X]$. 
We cannot have $\{s,t\}=\{q,r\}$, as otherwise $P_1 \cup P_2$ would be an even cycle in $(G_2,\Sigma_2)$ and a path in $G_1$. 
Therefore $P_1$ and $P_2$ share exactly one end in $G_2[X]$, say $u_t'$.
Thus, we can reindex $u_j$ as $u_t$. Similarly we can reindex all the vertices in $\zB_{G_1}[X]$
as desired. Note that, in particular, $G_1[X]=G_2[X]$.
For $i=1,2$, let $\Sigma_i'$ be a resigning of $(G_i,\Sigma_i)$ such that $\Sigma_i' \cap X =\emptyset$.
Define $Y:=X \cup \{e \in E(G_1): e \notin \Sigma_i', e=(u_i,u_j) \ \textrm{for some} \ i,j \in [3] \}$.
Now we can apply a $\Delta$-reduction to $Y$.
\end{proof}
\begin{LE}\label{doublesplit:rem4}
Let $H$ be a graph and let $s_1$ and $s_2$ be distinct vertices of $H$. Let $\varphi_i \subseteq \delta_H(s_i)$, for $i=1,2$.
Suppose that $\varphi_1 \Delta \varphi_2$ is a non-empty cut of $H$
such that $\varphi_1 \Delta \varphi_2 \neq \delta_H(s_2)$.
Then there exists $Y \subseteq E(H)$ such that the following hold:
\begin{enumerate}[\;\;\;(1)]
	\item $\zB_H(Y) \subseteq \{s_1,s_2\}$;
	\item $\zI_H(Y) \neq \emptyset$;
	\item $\delta_H(s_1)\cap Y=\varphi_1 - \varphi_2$;
	\item for $\hat{\varphi}_2:= \varphi_2$ or $\hat{\varphi}_2:=\varphi_2 \Delta \delta_H(s_2)$,
		$\delta_H(s_2)\cap Y=\hat{\varphi}_2 - \varphi_1$.
\end{enumerate}
\end{LE}
\begin{proof}
As $\varphi_1 \Delta \varphi_2$ is a non-empty cut of $H$, $\varphi_1 \Delta \varphi_2=\delta_H(U)$
for some $U \subset V(H)$, where $U\neq \emptyset, V(H)$. 
If $s_1 \in U$, we can pick $V(H)-U$ instead of $U$. Thus we may assume that $s_1 \notin U$.
If $s_2 \notin U$, let $\hat{\varphi_2} := \varphi_2$ and $W:=U$,
otherwise let $\hat{\varphi_2} := \varphi_2 \Delta \delta_H(s_2)$ and $W:=U - \{s_2\}$.
Thus $s_1,s_2 \notin W$ and $\delta_H(W)=\varphi_1 \Delta \hat{\varphi_2}$.
Define $Y:=\{(u,v) \in E(H): \{u,v\} \cap W \neq \emptyset\}$.
Conditions (3) and (4) in the statement are satisfied by construction.
Note that $U \neq \{s_2\}$, as $\varphi_1 \Delta \varphi_2 \neq \delta_H(s_2)$.
Hence $W$ is non-empty and $\zI_H(Y)$ is non-empty.
For every $v \in W$, $\delta_H(v) \subseteq Y$, hence $v \notin \zB_H(Y)$. 
Moreover, for every $v \notin W \cup \{s_1,s_2\}$, $\delta_H(v) \cap Y =\emptyset$, hence $v \notin \zB_H(Y)$.
Hence $\zB_H(Y) \subseteq \{s_1,s_2\}$.
\end{proof}
\begin{LE}\label{doublesplit:rem5}
Let $H$ be a graph and $s_1,s_2,s_3$ be distinct vertices of $H$. Let $\varphi_i \subseteq \delta_H(s_i)$, for $i=1,2,3$.
Suppose that $\varphi_1 \Delta \varphi_2 \Delta \varphi_3$ is a non-empty cut of $H$.
Suppose moreover that $\varphi_1 \Delta \varphi_2 \Delta \varphi_3$
is not equal to any of the sets $\delta_H(s_2), \delta_H(s_3), \delta_H(\{s_2,s_3\})$.
Then there exists $Y \subseteq E(H)$ such that the following hold:
\begin{enumerate}[\;\;\;(1)]
	\item $\zB_H(Y)\subseteq \{s_1,s_2,s_3\}$;
	\item $\zI_H(Y) \neq \emptyset$;
	\item $\delta_H(s_1)\cap Y=\varphi_1 - (\varphi_2 \cup \varphi_3)$;
	\item for $\hat{\varphi}_2:= \varphi_2$ or $\hat{\varphi}_2:=\varphi_2 \Delta \delta_H(s_2)$,
		$\delta_H(s_2)\cap Y=\hat{\varphi}_2 - (\varphi_1 \cup \varphi_3)$;
	\item for $\hat{\varphi}_3:= \varphi_3$ or $\hat{\varphi}_3:=\varphi_3 \Delta \delta_H(s_3)$,
		$\delta_H(s_3)\cap Y=\hat{\varphi}_3 - (\varphi_1 \cup \varphi_2)$.
\end{enumerate}
\end{LE}
\begin{proof}
As $\varphi_1 \Delta \varphi_2 \Delta \varphi_3$ is a non-empty cut of $H$, 
$\varphi_1 \Delta \varphi_2 \Delta \varphi_3=\delta_H(U)$
for some $U \subset V(H)$, where $U\neq \emptyset, V(H)$. 
If $s_1 \in U$, we can pick $V(H)-U$ instead of $U$. Thus we may assume that $s_1 \notin U$.
For $i=2,3$, define $\hat{\varphi_i}:= \varphi_i$ if $s_i \notin U$ 
and $\hat{\varphi_i}=\delta_H(s_i) \Delta \varphi_i$ otherwise. Let $W:= U - \{s_2,s_3\}$.
Thus $s_1,s_2,s_3 \notin W$ and $\delta_H(W)=\varphi_1 \Delta \hat{\varphi_2} \Delta \hat{\varphi_3}$.
Define $Y:=\{(u,v) \in E(H): \{u,v\} \cap W \neq \emptyset\}$.
By construction, $\delta_H(s_1)\cap Y=\varphi_1 - (\varphi_2 \Delta \varphi_3)$.
If $e \in \varphi_2 \cap \varphi_3$, then $e=(s_2,s_3)$ and $e \notin \varphi_1$.
Thus $\varphi_1 - (\varphi_2 \Delta \varphi_3)=\varphi_1 - (\varphi_2 \cup \varphi_3)$ and condition (3) holds.
Conditions (4) and (5) follow similarly.
It follows from the hypothesis of the lemma that $U$ is not contained in $\{s_2,s_3\}$.
Hence $W$ is non-empty and $\zI_H(Y)$ is non-empty.
For every $v \in W$, $\delta_H(v) \subseteq Y$, hence $v \notin \zB_H(Y)$. 
Moreover, for every $v \notin W \cup \{s_1,s_2,s_3\}$, $\delta_H(v) \cap Y =\emptyset$, hence $v \notin \zB_H(Y)$.
It follows that $\zB_H(Y) \subseteq \{s_1,s_2,s_3\}$.
\end{proof}
\begin{RE}\label{doublesplit:compeasy}
Let $\dst=(H_1,v_1,w_1,\alpha_1,\beta_1,H_2,v_2,w_2,\alpha_2,\beta_2,\bS)$ be a quad-template.
Suppose that $\bS=(X_1,\ldots,X_k)$ and $\zB_{H_1}(X_1) \cap \{v_1,w_1\} = \emptyset$.
Let $\dst'=(\Wflip[H_1,\bS],v_1,w_1,\alpha_1$, $\beta_1$, $H_2$, $v_2$, $w_2$, $\alpha_2$, $\beta_2,\bS')$,
where  $\bS'=(X_2,\ldots,X_k)$.
Then $\dst'$ is a quad-template and $\dst$ and $\dst'$ are compatible.
\end{RE}
Suppose that $\dst=(H_1,v_1,w_1,\alpha_1,\beta_1,H_2,v_2,w_2,\alpha_2,\beta_2,\bS)$ is a quad-template.
If we substitute $\alpha_i$ (respectively $\beta_i$) with $\delta_{H_i}(v_i) \Delta \alpha_i$ (respectively $\delta_{H_i}(w_i) \Delta \beta_i$)
for $i=1$ or $i=2$ we obtain a quad-template $\dst'$ giving rise to the same quad siblings.
We say that $\dst'$ is obtained from $\dst$ by a {\em swap} on $v_i$ (respectively $w_i$).
We will make repeated use of swaps in the next section.
\begin{LE}\label{doublesplit:cl1}
Let $(G_1,\Sigma_1)$ and $(G_2,\Sigma_2)$ be ec-standard and $\Delta$-irreducible quad siblings arising from
a quad-template $\dst=(H_1,v_1,w_1,\alpha_1,\beta_1,H_2,v_2,w_2,\alpha_2,\beta_2)$.
Let $X$ be a $k$-separation of $H_1$ and $H_2$, for $k \leq 2$. Let $Y:=E(H_1)-(X\cup \loops(H_1))$.
Suppose that $\zI_{H_i}(X), \zI_{H_i}(Y) \neq \emptyset$ and $v_i,w_i \in V(H_i[X])$, for $i=1,2$. 
Suppose moreover that, for $h=1$ or $h=2$, $\zI_{H_h}(X) \cap \{v_h,w_h\} \neq \emptyset$.
Let $j=3-h$.
Then $\zB_{H_j}(X)=\{v_j,w_j\}$ and all the sets $\alpha_j \cap Y$, $(\delta_{H_j}(v_j) - \alpha_j) \cap Y$,
$\beta_j \cap Y$ and $(\delta_{H_j}(w_j) - \beta_j) \cap Y$ are non-empty.
In particular, $X$ is a $2$-separation in $H_1$ and $H_2$.
\end{LE}
\begin{proof}
To simplify the notation we prove the result for the case $h=1$.
Thus we may assume that $v_1 \in V(H_1[X])$, $w_1 \in \zI_{H_1}(X)$ and $v_2, w_2 \in V(H_2[X])$.
Suppose for contradiction that $w_2 \in \zI_{H_2}(X)$ or that $w_2 \in \zB_{H_2}(X)$ but one of the sets
$\beta_2 \cap Y$, $(\delta_{H_2}(w_2) - \beta_2) \cap Y$ is empty.
If $w_2 \in \zI_{H_2}(X)$, then $\beta_2 \cap Y = \emptyset$.
Thus either $\beta_2 \cap Y = \emptyset$ or $w_2 \in \zB_{H_2}(X)$ and $\delta_{H_2}(w_2) \cap Y \subseteq \beta_2$.
In the second case, we may substitute $\beta_2$ with $\delta_{H_2}(w_2) \Delta \beta_2$ (this is just a swap), 
reducing to the case $\beta_2 \cap Y = \emptyset$.
As $w_1 \in \zI_{H_1}(X)$, we have $\beta_1 \cap Y = \emptyset$.
For $i=1,2$, let $v_i$ be split into vertices $v_i^-$ and $v_i^+$ of $G_i$. Define $w_i^-,w_i^+$ similarly. 
Recall that $\beta_i$ is a signature of $(G_{3-i},\Sigma_{3-i})$ for $i=1,2$. 
Every edge in $\beta_i \cap \loops(H_i)$ is also in $\alpha_{3-i} \Delta \beta_{3-1}$
(by definition of unfolding). 
Thus every edge in $\beta_i \cap \loops(H_i)$ is either a $(v_{3-i}^-,v_{3-i}^+)$ edge or a $(w_{3-i}^-,w_{3-i}^+)$ edge in $G_{3-i}$.
This implies that, for $i=1,2$, $(G_i[Y],\Sigma_i \cap Y)$ is bipartite and
$Y$ is a $k_i$-separation of $G_i$ for $k_i\leq 3$.
As $\ecycle(G_1,\Sigma_1)$ is $3$-connected, $Y$ is not a $1$- or a $2$-separation in $G_1$ or $G_2$, by Lemma~\ref{connectivity:n3c}.
Thus $k_1=k_2=3$.
Moreover, $Y$ is not a $3$-$(0,0)$-separation in $(G_i, \Sigma_i)$, for $i=1,2$, for otherwise $(G_i,\Sigma_i)$
would contain a blocking vertex.
Thus $Y$ is a $3$-$(0,1)$-separation in $(G_1,\Sigma_1)$ and $(G_2,\Sigma_2)$. 
By Lemma~\ref{doublesplit:clean}, $(G_1,\Sigma_1)$ and $(G_2,\Sigma_2)$ are $\Delta$-reducible, a contradiction.
This implies that $w_2 \in \zB_{H_2}(X)$ and the sets $\beta_2 \cap Y$ and $(\delta_{H_2}(w_2) - \beta_2) \cap Y$ are non-empty.
By symmetry between $v_2$ and $w_2$, $v_2 \in \zB_{H_2}(X)$ and the sets $\alpha_2 \cap Y$ and $(\delta_{H_2}(v_2) - \alpha_2) \cap Y$ are non-empty.
\end{proof}
%
\subsection{Proofs of Lemmas~\ref{doublesplit:le2}, \ref{doublesplit:le3}, \ref{doublesplit:le4} and~\ref{doublesplit:le0}}\label{sec:lemmaProof}
\begin{proof}[\textbf{Proof of Lemma~\ref{doublesplit:le2}}]
Let $\dst=(H_1,v_1,w_1,\alpha_1,\beta_1,H_2,v_2,w_2,\alpha_2,\beta_2,\bS)$ be a quad-template of type I, 
where $\bS=(X_1,\ldots,X_k)$ for some $k \geq 0$. 
For $i=1,2$, let $\Gamma_i := \alpha_i \Delta \beta_i$.
By definition of quad siblings, $(H_1,\Gamma_1)$ and $(H_2,\Gamma_2)$ are equivalent.
Thus $\Gamma_1 \Delta \Gamma_2 = \alpha_1 \Delta \beta_1 \Delta \alpha_2 \Delta \beta_2$ is a cut of $H_1$.
Let $\Gamma_1 \Delta \Gamma_2 = \delta_{H_1}(U)$ for some $U \subseteq V(H_1)$.
By possibly swapping on $v_1$ or $w_1$, we may assume that $v_1,w_1 \notin U$.

\textbf{Case 1:} Suppose $k \geq 1$.
Let $X_{k+1},\ldots, X_t$ be a partition of $E(H_1) - (X_1 \cup \ldots \cup X_k \cup \loops(H_1))$
into minimal $2$-separations having as boundary $\{v_1,w_1\}$ plus possibly edges with ends $v_1$ and $w_1$.
Let $U_j=U \cap V_{H_1}(X_j)$, for every $j \in [t]$.
As $v_1,w_1 \notin U$ and $X_j,X_h$ are disjoint for all distinct $j,h \in [t]$,
the sets $U_1,\ldots,U_t$ are all disjoint.
Suppose that $U_j \neq \emptyset$ for some $j \in [t]$.
Thus $(\Gamma_1 \Delta \Gamma_2) \cap X_j$ is a non-empty cut of $H_1[X_j]$.
By Lemma~\ref{doublesplit:rem4}, there exists a set $Y \subseteq X_j$
such that $\zB_{H_1}(Y) \subseteq \{v_1,w_1\}$; $\zI_{H_1}(Y) \neq \emptyset$;
$\delta_{H_1}(v_1)\cap Y=(\Gamma_1 \Delta \Gamma_2) \cap \delta_{H_1}(v_1) \cap X_j$;
$\delta_{H_1}(w_1)\cap Y=(\Gamma_1 \Delta \Gamma_2) \cap \delta_{H_1}(w_1) \cap X_j$.
As $H_1[X_j] \setminus \{v_1,w_1\}$ is connected, $Y=X_j$ and $U_j=\zI_{H_1}(X_j)$.
Thus for every $j \in [t]$, either $U_j = \emptyset$ or $U_j = \zI_{H_1}(X_j)$.
Therefore $U = \bigcup_{i \in I} \zI_{H_1}(X_i)$, for some $I \subseteq [t]$.
Define the following index sets:
$I_1:=([t] - [k]) \cap I$; $I_2:=[k] - I$; $I_3:=[k] \cap I$; $I_4:=[t] - ([k] \cup I)$.
Note that $I_1,I_2,I_3,I_4$ partition $[t]$. 
The idea is that for each $2$-separation $X_j$ with $\zB_{H_i}(X_j)=\{v_i,w_i\}$, there are four possible choices,
depending whether, when going from $H_1$ to $H_2$, we resign, flip, resign and flip or do not perform any operation in $H_1[X_j]$.
Now partition the edges in $\loops(H_1) \cap \Gamma_1$ as $L_1 \cup L_2$, where $e \in L_1$ if 
$e \in \alpha_1 \cap \alpha_2$ or $e \in \beta_1 \cap \beta_2$ and $e \in L_2$ otherwise.
Finally define $Y_1 := \bigcup_{j \in I_1}(X_j) \cup L_1$;
$Y_2 := \bigcup_{j \in I_2}(X_j)$;
$Y_3 := \bigcup_{j \in I_3}(X_j) \cup L_2$;
$Y_4 := \bigcup_{j \in I_4}(X_j)$.
Then $(G_1,\Sigma_1)$ and $(G_2,\Sigma_2)$ form a shuffle with partition $Y_1,Y_2,Y_3,Y_4$.
 
\textbf{Case 2:} Suppose $k=0$. This implies that $H_1 = H_2$. In this case we may also assume that $v_2,w_2 \notin U$
(by possibly swapping on $v_2,w_2$). 
We now have different cases depending on the cardinality of $\{v_1,w_1\} \cap \{v_2,w_2\}$.

\textbf{Case 2.1:} Suppose $\{v_1,w_1\}=\{v_2,w_2\}$. 
Then, similarly to case 1, we obtain a shuffle (where the sets $Y_2$ and $Y_3$ are empty).

\textbf{Case 2.2:} Suppose $\{v_1,w_1\} \cap \{v_2,w_2\} = \{v_1\}=\{v_2\}$. 
This implies that $\delta(U) \subseteq \delta(v_1) \cup \delta(w_1) \cup \delta(w_2)$. 
Moreover, $\delta(w_1) \cap \delta(U) = \delta(w_1) \cap \Gamma_1$
and $\delta(w_2) \cap \delta(U) = \delta(w_2) \cap \Gamma_2$.
Define $Y_1 := E(H_1[U]) \cup \delta(U)$ and $Y_2 := E(H_1) - (Y_1 \cup \loops(H_1))$.
If $e \in \loops(H_i) - (\alpha_i \cup \beta_i)$, then $e$ is an even loop of $(G_i, \Sigma_i)$,
contradicting the fact that $\ecycle(G_i,\Sigma_i)$ is $3$-connected.
Thus every loop of $H_i$ is either in $\alpha_i$ or in $\beta_i$ (but not both, by definition of unfolding).
Moreover, for $i=1,2$, $(G_i,\Sigma_i)$ does not have parallel edges of the same parity.
It follows that $|\loops(H_i)|\leq 4$ and every edge in $\loops(H_1)$ is in 
exactly one of $\alpha_1,\beta_1$ and in exactly one of $\alpha_2,\beta_2$.
If $\loops(H_1) \cap \beta_1 \cap \alpha_2$ is non-empty, let $e \in \loops(H_1) \cap \beta_1 \cap \alpha_2$.
Similarly, if they exist, define edges $f,g,h \in \loops(H_1)$ as follows:
$f \in \beta_1 \cap \beta_2$; $g \in \alpha_1 \cap \alpha_2$; $h \in \alpha_1 \cap \beta_2$.
Then $(G_1,\Sigma_1)$ and $(G_2,\Sigma_2)$ are related by a twist with partition $Y_1,Y_2,\{e,f,g,h\}$.

\textbf{Case 2.3:} Suppose $\{v_1,w_1\} \cap \{v_2,w_2\} = \emptyset$. 
This implies that $\delta(U) \subseteq \delta(v_1) \cup \delta(w_1) \cup \delta(v_2) \cup \delta(w_2)$. 
Moreover, $\delta(v_i) \cap \delta(U) = \delta(v_i) \cap \Gamma_i$ and $\delta(w_i) \cap \delta(U) = \delta(w_i) \cap \Gamma_i$ for $i=1,2$.
Define $Y_1 := E(H_1[U]) \cup \delta(U), Y_2 := E(H_1) - (Y_1 \cup \loops(H_1))$ and, if they exist, edges $e,f,g,h \in \loops(H_1)$ as follows:
$e \in \alpha_1 \cap \alpha_2$; $f \in \alpha_1 \cap \beta_2$; $g \in \beta_1 \cap \alpha_2$; $h \in \beta_1 \cap \beta_2$.
Then $(G_1,\Sigma_1)$ and $(G_2,\Sigma_2)$ are related by a tilt with partition $Y_1,Y_2,\{e,f,g,h\}$.
\end{proof}

\begin{proof}[\textbf{Proof of Lemma~\ref{doublesplit:le3}}]
Let $\dst=(H_1,v_1,w_1,\alpha_1,\beta_1,H_2,v_2,w_2,\alpha_2,\beta_2,\bS)$ be a quad
template of type II.
Fix $i=1$ or $i=2$. If $e \in \loops(H_i) - (\alpha_i \cup \beta_i)$, then $e$ is an even loop of $(G_i, \Sigma_i)$,
contradicting the fact that $\ecycle(G_i,\Sigma_i)$ is $3$-connected.
Thus every loop of $H_i$ is either in $\alpha_i$ or in $\beta_i$ (but not both, by definition of unfolding).
Moreover, for $i=1,2$, $(G_i,\Sigma_i)$ does not have parallel edges of the same parity.
It follows that $|\loops(H_i)|\leq 4$ and every edge in $\loops(H_1)$ is in 
exactly one of $\alpha_1$ and $\beta_1$ and in exactly one of $\alpha_2$ and $\beta_2$.
Thus we will not consider the behavior of the loops of $H_1$ any further in this proof.
Now we consider two cases, depending on whether $|\bS|=1$ or $|\bS|=2$.

\textbf{Case 1:}
Suppose $|\bS|=1$. We will show that $(G_1,\Sigma_1)$ and $(G_2,\Sigma_2)$ are widget twins.
In this case, $H_2=\Wflip[H_1,X]$ for some $2$-separation $X$ of $H_1$, and
$v_i \in \zB_{H_i}(X)$, for $i=1,2$. Moreover, $w_1 \in \zI_{H_1}(X)$
and, for $Y:= \bar{X} - \loops(H_1)$, $w_2 \in \zI_{H_2}(Y)$.
For $i=1,2$, let $z_i$ be the vertex in $\zB_{H_i}(X)$ distinct from $v_i$.
By swapping the role of $X$ and $Y$ and of $H_1$ and $H_2$, we may assume that
$\delta_{H_1}(v_1)\cap X = \delta_{H_2}(v_2)\cap X$.
Define $\varphi_1:=(\alpha_1 \Delta \alpha_2) \cap X$ and $\varphi_2:=\beta_1 \cap X$.
Let $H:=H_1[X]$.
We have $\varphi_1 \subseteq \delta_H(v_1)$ and $\varphi_2 \subseteq \delta_H(w_1)$.
Moreover, $\varphi_1 \Delta \varphi_2 = (\alpha_1 \Delta \alpha_2 \Delta \beta_1) \cap X$.
By definition of quad siblings, $\alpha_1 \Delta \alpha_2 \Delta \beta_1 \Delta \beta_2$ is a cut of $H_1$. 
As $\beta_2 \cap X$ is empty, $C_1:=(\alpha_1 \Delta \alpha_2 \Delta \beta_1) \cap X$ is a cut of $H$.
First suppose that $C_1$ is empty. Then all the edges in $\beta_1 - \loops(H_1)$ are either in $\alpha_1$
or in $\alpha_2$ (but not both). As $(G_1,\Sigma_1)$ does not contain parallel edges of the same parity,
there cannot be two edges in $\beta_1 \cap \alpha_1$ or in $\beta_1 \cap \alpha_2$.
If $H_1$ contains a $(v_1,w_1)$ edge in $\beta_1 \cap \alpha_1$ (respectively in $\beta_1 \cap \alpha_2$) call such an edge $e$ (respectively $f$).
Let $\gamma = (X \cap \alpha_1) - \{e\}$.
As $C_1$ is empty, $\alpha_2 \cap X = \gamma \cup \{f\}$.

Now suppose that $C_1$ is non-empty. 
If $\delta_H(w_1)=C_1$, we may swap on $w_1$ and reduce to the case where $C_1 = \emptyset$ (as $\delta_H(w_1)=\delta_{H_1}(w_1)$).
Thus we may assume that $C_1 \neq \delta_{H_1}(w_1)$.
By Lemma~\ref{doublesplit:rem4}, there exists $Z \subseteq X$
such that $\zB_H(Z)\subseteq \{v_1,w_1\}$, $\zI_H(Z) \neq \emptyset$, $\delta_H(v_1) \cap Z = \varphi_1 - \varphi_2$
and for $\hat{\varphi}_2 = \varphi_2$ or $\hat{\varphi}_2=\varphi_2 \Delta \delta_H(w_1)$,
we have $\delta_H(w_1) \cap Z = \hat{\varphi}_2 - \varphi_1$.
Note that $Z$ is a $2$-separation in $H_1$, because $\zB_H(Z)\subseteq \{v_1,w_1\}$ and $H_1$ is $2$-connected except for loops.
Let $\hat{Z}:=E(H_1) - (\loops(H_1) \cup \{(v_1,w_1) \in E(H_1)\})$.
The condition $\delta_H(w_1) \cap Z = \hat{\varphi}_2 - \varphi_1$ implies that either
$\delta_H(w_1) \cap Z \subseteq \beta_2$ or $\delta_H(w_1) \cap Z \subseteq \delta_H(w_1)- \beta_2$.
Hence $\hat{Z}$ violates Lemma~\ref{doublesplit:cl1}.

We conclude that, by possibly swapping on $w_1$, $\beta_1 - \loops(H_1) = \{e,f\}$,
$\alpha_1 \cap X = \gamma \cup \{e\}$ and $\alpha_2 \cap X = \gamma \cup \{f\}$.
Now we proceed to consider the structure of $H_1[Y]$.
We assume that every edge with endpoints $v_1$ and $z_1$ in $H_1$ is in $X$.
Define sets $\varphi_1 = \alpha_1 \cap Y$, $\varphi_2 = \alpha_2 \cap Y$ and $\varphi_3 = \beta_2 \cap Y$.
As $\beta_1$ does not intersect $Y$, 
$C_2:= (\alpha_1 \Delta \alpha_2 \Delta \beta_1 \Delta \beta_2) \cap Y=\varphi_1 \Delta \varphi_2 \Delta \varphi_3$.
As $\alpha_1 \Delta \alpha_2 \Delta \beta_1 \Delta \beta_2$ is a cut of $H_1$,
we have that $C_2$ is a cut of $H_1[Y]$.
If $C_2= \emptyset$, then every edge in $\beta_2$ is either contained in $\alpha_1$ or in $\alpha_2$.
Similarly for the edges in $\alpha_1 \cap Y$ and in $\alpha_2 \cap Y$.
As there are no $(v_1,z_1)$ edges in $Y$, we have $\beta_2 -\loops(H_1) = \{a,c\}$
for two edges $a=(v_1,w_2)$ and $c=(z_1,w_2)$ in $H_1$ (if they exist). 
Moreover, $\alpha_1 \cap Y =  \{a\}$ and $\alpha_2 \cap Y = \{c\}$. 
Let $Z=Y - \{a,c\} - \{(v_1,w_2), (z_1,w_2) \in E(H_1)\}$.
Then all the sets $\alpha_1 \cap Z$, $\alpha_2 \cap Z$, $\beta_1 \cap Z$, $\beta_2 \cap Z$, are empty.
Therefore, if $\zI_{H_1}(Z)$ is non-empty, $Z$ is a $3$-$(0,1)$-separation of $(G_1,\Sigma_1)$ and $(G_2,\Sigma_2)$
and $(G_1,\Sigma_1)$ and $(G_2,\Sigma_2)$ are $\Delta$-reducible by Lemma~\ref{doublesplit:clean}, a contradiction.
Hence $Z$ is empty and $Y = \{a,b,c,d\}$, where $b=(v_1,w_2)$, $d=(z_1,w_2)$ (if they exist)
and $b,d \notin \alpha_1 \cup \alpha_2 \cup \beta_2$. 
We conclude that, in the case $C_2= \emptyset$, $(G_1,\Sigma_1)$ and $(G_2,\Sigma_2)$ are widget twins.

Now suppose that $C_2 \neq \emptyset$. Let $H:=H_1[Y]$. 
If $C_2$ is equal to one of the sets $\delta_H(w_2), \delta_H(z_1)$, $\delta_H(\{z_1,w_2\})$,
we may swap on $w_2$ or $v_2$ and reduce to the case where $C_2=\emptyset$ (as $\delta_H(z_1) \subset \delta_{H_2}(v_2)$).
Therefore we may assume that $\varphi_1, \varphi_2, \varphi_3$ satisfy the hypotheses of Lemma~\ref{doublesplit:rem5}. 
Hence there exists a set $W \subseteq Y$ such that $\zB_H(W) \subseteq \{v_1,z_1,w_2\}$, $\zI_H(W)\neq \emptyset$,
and
\begin{enumerate}[\;\;\;(a)]
	\item $\delta_H(v_1) \cap W = \alpha_1 - (\alpha_2 \cup \beta_2)$;
	\item either $\delta_H(z_1) \cap W = \alpha_2 - (\alpha_1 \cup \beta_2)$, or
		$\delta_H(z_1) \cap W = \delta_H(z_1) - (\alpha_1 \cup \alpha_2 \cup \beta_2)$;
	\item either $\delta_H(w_2) \cap W = \beta_2 - (\alpha_1 \cup \alpha_2)$, or
		$\delta_H(w_2) \cap W = \delta_H(w_2) - (\alpha_1 \cup \alpha_2 \cup \beta_2)$.
\end{enumerate}
Therefore $W$ is a $3$-$(0,1)$-separation of $(G_1,\Sigma_1)$ and $(G_2,\Sigma_2)$. 
By Lemma~\ref{doublesplit:clean},  $(G_1,\Sigma_1)$ and $(G_2,\Sigma_2)$
are $\Delta$-reducible, a contradiction. 

\textbf{Case 2:} $|\bS|=2$. We will show that $(G_1,\Sigma_1)$ and $(G_2,\Sigma_2)$ are gadget twins.
In this case $H_2=\Wflip[H_1,(Y,Z)]$ for some disjoint $2$-separations $Y$ and $Z$ of $H_1$, where
$v_i \in \zB_{H_i}(Y) \cap \zB_{H_i}(Z)$, for $i=1,2$, $w_1 \in \zI_{H_1}(Y)$ and $w_2 \in \zI_{H_2}(Z)$.
For $i=1,2$, let $z_i$ be the vertex in $\zB_{H_i}(Y)$ distinct from $v_i$
and $u_i$ the vertex in $\zB_{H_i}(Z)$ distinct from $v_i$.
For $X:=E(H_1) - (Y \cup Z \cup \loops(H_1))$, $\zB_{H_i}(X)=\{v_i,u_i,z_i\}$, for $i=1,2$.
Moreover, we can choose $Y$ and $Z$ so that all the edges in $H_1$ with both ends in $\{v_1,z_1,u_1\}$ are contained in $X$.
By construction, $\delta_{H_1}(v_1) \cap X = \delta_{H_2}(v_2) \cap X$.
Moreover $\delta_{H_1}(z_1) \cap Y = \delta_{H_2}(v_2) \cap Y$
and $\delta_{H_1}(u_1) \cap Z = \delta_{H_2}(v_2) \cap Z$.
Define $\varphi_1=\alpha_2 \cap Y$, $\varphi_2=\alpha_1 \cap Y$ and $\varphi_3=\beta_1 \cap Y$.
Let $H:=H_1[Y]$. So $\varphi_1 \subseteq \delta_H(z_1)$, $\varphi_2 \subseteq \delta_H(v_1)$
and $\varphi_3 \subseteq \delta_H(w_1)$.
Note that $C:=\varphi_1 \Delta \varphi_2 \Delta \varphi_3 = (\alpha_1 \Delta \alpha_2 \Delta \beta_1 \Delta \beta_2) \cap Y$.
As $\alpha_1 \Delta \alpha_2 \Delta \beta_1 \Delta \beta_2$ is a cut of $H_1$,
we have that $C$ is a cut of $H$.

If $C= \emptyset$, then every edge in $\beta_1$ is either contained in $\alpha_1$ or in $\alpha_2$.
Similarly for the edges in $\alpha_1 \cap Y$ and in $\alpha_2 \cap Y$.
As there are no $(v_1,z_1)$ edges in $Y$, we have $\beta_1 -\loops(H_1) = \{a_1,c_1\}$
for two edges $a_1=(v_1,w_1)$ and $c_1=(z_1,w_1)$ in $H_1$ (if they exist). 
Moreover, $\alpha_1 \cap Y =  \{a_1\}$ and $\alpha_2 \cap Y = \{c_1\}$. 
Let $W=Y - \{a_1,c_1\} - \{(v_1,w_1), (z_1,w_1) \in E(H_1)\}$.
Then all the sets $\alpha_1 \cap W$, $\alpha_2 \cap W$, $\beta_1 \cap W$, $\beta_2 \cap W$, are empty.
Therefore, if $\zI_{H_1}(W)$ is non-empty, $W$ is a $3$-$(0,1)$-separation of $(G_1,\Sigma_1)$ and $(G_2,\Sigma_2)$
and $(G_1,\Sigma_1)$ and $(G_2,\Sigma_2)$ are $\Delta$-reducible by Lemma~\ref{doublesplit:clean}, a contradiction.
Hence $W$ is empty and $Y = \{a_1,b_1,c_1,d_1\}$, where $b_1=(v_1,w_1)$, $d_1=(z_1,w_1)$ (if they exist)
and $b_1,d_1 \notin \alpha_1 \cup \alpha_2 \cup \beta_1$.
 
Now suppose that $C \neq \emptyset$. 
If $C$ is equal to one of the sets $\delta_H(v_1), \delta_H(w_1), \delta_H(\{v_1,w_1\})$,
we may swap on $v_1$ or $w_1$ and reduce to the case $C=\emptyset$.
Therefore we may assume that $\varphi_1, \varphi_2, \varphi_3$ satisfy the hypothesis of Lemma~\ref{doublesplit:rem5}. 
Hence there exists a set $W' \subseteq Y$ such that $\zB_H(W') \subseteq \{v_1,z_1,w_1\}$, $\zI_H(W')\neq \emptyset$,
and
\begin{enumerate}[\;\;\;(a)]
	\item $\delta_H(z_1) \cap W' = (\alpha_2 \cap Y) - (\alpha_1 \cup \beta_1)$;
	\item either $\delta_H(v_1) \cap W' = (\alpha_1 \cap Y) - (\alpha_2 \cup \beta_1)$, or
		$\delta_H(v_1) \cap W' = \delta_H(v_1) - (\alpha_1 \cup \alpha_2 \cup \beta_1)$;
	\item either $\delta_H(w_1) \cap W' = (\beta_1 \cap Y) - (\alpha_1 \cup \alpha_2)$, or
		$\delta_H(w_1) \cap W' = \delta_H(w_1) - (\alpha_1 \cup \alpha_2 \cup \beta_1)$.
\end{enumerate}
Therefore $W'$ is a $3$-$(0,1)$-separation of $(G_1,\Sigma_1)$ and $(G_2,\Sigma_2)$. 
By Lemma~\ref{doublesplit:clean},  $(G_1,\Sigma_1)$ and $(G_2,\Sigma_2)$
are $\Delta$-reducible, a contradiction.
We deduce that, up to swaps on $v_1$ and $w_1$, $Y=\{a_1,b_1,c_1,d_1\}$, 
with the conditions on $\alpha_1,\beta_1,\alpha_2,\beta_2$ established before.
Now consider the structure of $H_1[Z]$. 
Define $\varphi_1=\alpha_1 \cap Z$, $\varphi_2=\alpha_2 \cap Z$ and $\varphi_3=\beta_1 \cap Z$.
Then, with an argument similar to the one above, we conclude that, up to possible swaps on $v_2$ and $w_2$,
$Z=\{a_2,b_2,c_2,d_2\}$, where the ends of $a_2$ and $b_2$ are $v_1$ and $w_2$ and
the ends of $c_2$ and $d_2$ are $u_1$ and $w_2$. 
Moreover, $\beta_2 -\loops(H_1)=\{a_2,c_2\}$, $\alpha_1 \cap Z =\{a_2\}$ and $\alpha_2 \cap Z =\{c_2\}$. 

Let $\gamma:=\alpha_1 \cap X$. As $(\alpha_1 \Delta \alpha_2) \cap X$ is a cut of $H_1[X]$,
either $\alpha_2 \cap X = \gamma$ or $\alpha_2 \cap X = (\delta_{H_2}(v_2) \cap X) - \gamma$.
In the second case, $\alpha_1 \Delta \beta_1 \Delta \alpha_2 \Delta \beta_2 = \delta_{H_2}(v_2) \cap X$,
which is not a cut of $H_2$, contradiction. 
It follows that $\alpha_2 \cap X = \gamma$ and $(G_1,\Sigma_1)$ and $(G_2,\Sigma_2)$ are gadget twins. 
\end{proof}

\begin{proof}[\textbf{Proof of Lemma~\ref{doublesplit:le4}}]
Let $\dst=(H_1,v_1,w_1,\alpha_1,\beta_1,H_2,v_2,w_2,\alpha_2,\beta_2,\bS)$ and $\bS=(X_1,\ldots,X_k)$.
By Proposition~\ref{conn:star2} applied to $H_1$ and $Z=\{v_1,w_1\}$, there exists a graph $H$ such that:
\begin{itemize}
	\item $H=\Wflip[H_1,\bS_1]$ for some w-sequence $\bS_1$ of $H_1$, where 
		$\{v_1,w_1\} \cap \zB_{H_1}(X)=\emptyset$ for all $X\in \bS_1$, and
	\item $H_2=\Wflip[H,\bS_2]$ for some non-crossing w-sequence $\bS_2$ such that, 
		for all $X \in \bS_2$, $\{v_1,w_1\} \cap \zB_{H_1}(X) \neq \emptyset$.
\end{itemize}

Let $\dst':=(H,v_1,w_1,\alpha_1,\beta_1,H_2,v_2,w_2,\alpha_2,\beta_2,\bS_2)$.
By Remark~\ref{doublesplit:compeasy}, $\dst'$ is a quad-template and $\dst$ and $\dst'$ are compatible.
Thus we may assume that $(G_1,\Sigma_1)$ and $(G_2,\Sigma_2)$ arise from a template
$\dst=(H_1,v_1,w_1,\alpha_1,\beta_1,H_2,v_2,w_2,\alpha_2,\beta_2,\bS)$, where $\bS=(X_1,\ldots,X_k)$ is non-crossing,
and for all $X \in \bS$, $\{v_1,w_1\} \cap \zB_{H_1}(X) \neq \emptyset$.
Similarly we may assume that, for all $X \in \bS$, $\{v_2,w_2\} \cap \zB_{H_2}(X) \neq \emptyset$.
We will also assume that every Whitney-flip in $\bS$ is non-trivial, that is, $\zI_{H_1}(X) \neq \emptyset$ for every $X \in \bS$.

First suppose that, for every $X \in \bS$, $\zB_{H_i}(X)=\{v_i,w_i\}$, for $i=1,2$.
We show that in this case we can find a w-sequence $\bS'$ for $H_1$
such that $\dst':=(H_1,v_1,w_1,\alpha_1,\beta_1,H_2,v_2,w_2,\alpha_2$, $\beta_2,\bS')$ is a quad-template
of type I.
As $\dst'$ is trivially compatible with $\dst$, this would prove the statement for this case.
Suppose that there exists $X \in \bS$ such that $H_i[X] \setminus \zB_{H_i}(X)$ is not connected. 
$\bS$ is non-crossing, thus we may rearrange the sets in $\bS$ in any order. Hence we may assume that $X=X_1$.
As $H_1$ is $2$-connected except for loops, there exists a partition $Y_1,\ldots,Y_s$ of $X$
such that $\zB_{H_i}(Y_j)=\{v_i,w_i\}$ and $H_i[Y_j] \setminus \zB_{H_i}(Y_j)$ is connected for every $i=1,2$ and $j \in [s]$.
Therefore, we can replace $\bS$ with $(Y_1,\ldots,Y_s,X_2,\ldots,X_k)$.
Hence we may assume that $H_i[X_j] \setminus \zB_{H_i}(X_j)$ is connected for every $i=1,2$ and $j\in [k]$.
If there exist distinct $i,j \in [k]$ such that $X_i \cap X_j \neq \emptyset$, then $X_i=X_j$.
Thus we may just remove $X_i$ and $X_j$ from $\bS$.
This will lead to a w-sequence $\bS'$ with the required properties.

Now suppose that there exists $X \in \bS$ with $\zB_{H_i}(X) \neq \{v_i,w_i\}$, for $i=1$ or $i=2$.
We will show that in this case we can find a compatible quad-template of type II.
\begin{claim}\label{doublesplit:cl2}
Let $X \in \bS$ such that $|\zB_{H_i}(X) \cap \{v_i,w_i\}|=1$ and $|\zI_{H_i}(X) \cap \{v_i,w_i\}|=1$
for $i=1$ or $i=2$. Then for $j=3-i$ and $Y:=\bar{X} - \loops(H_i)$,
$|\zB_{H_j}(X) \cap \{v_j,w_j\}|=1$ and $|\zI_{H_j}(Y) \cap \{v_j,w_j\}|=1$.
\end{claim}
\begin{cproof}
To simplify the notation we prove the claim for the case $i=1$.
Thus we may assume that $v_1 \in \zB_{H_1}(X)$ and $w_1 \in \zI_{H_1}(X)$.
As $\zB_{H_2}(Z) \cap \{v_2,w_2\}\neq \emptyset$ for every $Z \in \bS$,
we have $\zB_{H_2}(X) \cap \{v_2,w_2\} \neq \emptyset$. Thus we may assume that $v_2 \in \zB_{H_2}(X)$.
Suppose for contradiction that $X$ violates the statement, that is, $w_2 \in V(H_2[X])$.
Note that we may choose $X$ such that for no other $X' \in \bS$ do we have $X \subseteq X'$
or $X \cap \bar{X}' = \emptyset$. By this choice, $H_1[Y]=H_2[Y]$.
If there exists an edge $e$ with ends $\zB_{H_1}(X)$, we will assume that such an edge is in $X$.
By Lemma~\ref{doublesplit:cl1}, $w_2 \in \zB_{H_2}(X)$ and the sets 
$\beta_2 \cap Y$ and $(\delta_{H_2}(w_2) - \beta_2) \cap Y$ are non-empty.
Thus $\zB_{H_2}(X)=\{v_2,w_2\}$.
By symmetry between $v_2$ and $w_2$, we may assume that $\delta_{H_1}(v_1) \cap Y = \delta_{H_2}(v_2) \cap Y$.
Define $\varphi_1 := (\alpha_1 \Delta \alpha_2) \cap Y$ and $\varphi_2 := \beta_2 \cap Y$.
Then $\varphi_1 \subseteq \delta_{H_2}(v_2)$ and  $\varphi_2 \subseteq \delta_{H_2}(w_2)$.
Moreover, $C:=\varphi_1 \Delta \varphi_2$ is a cut of $H_2[Y]$.
As there is no $(v_2,w_2)$ edge in $Y$, the sets $\varphi_1, \varphi_2$ are disjoint.
Moreover, the sets $\beta_2 \cap Y$ and $(\delta_{H_2}(w_2) - \beta_2) \cap Y$ are non-empty,
thus $C$ is non-empty and $C \neq \delta_{H_2}(w_2)$. 
Let $H:=H_2[Y]$.
By Lemma~\ref{doublesplit:rem4}, there exists a set $Z \subset Y$ such that 
$\zB_H(Z) \subseteq \{v_2,w_2\}$; $\zI_H(Z) \neq \emptyset$; $\delta_H(v_2)\cap Y=\varphi_1$; and
for $\hat{\varphi}_2 = \varphi_2$ or $\hat{\varphi}_2=\varphi_2 \Delta \delta_H(w_2)$, $\delta_H(w_2)\cap Y=\hat{\varphi}_2$.
Define $W:= E(H_1) - (Z \cup \loops(H_1))$. Then $W$ contradicts Lemma~\ref{doublesplit:cl1}.
\end{cproof}

Now we can conclude the proof.
We have already considered the case in which, for every $X \in \bS$, $\zB_{H_i}(X)=\{v_i,w_i\}$, for $i=1,2$.
Thus we have that for some $X \in \bS$ and $i=1$ or $i=2$,
$|\zB_{H_i}(X) \cap \{v_i,w_i\}|=1$ and $|\zI_{H_i}(X) \cap \{v_i,w_i\}|=1$. 
Let $Y:=\bar{X}-\loops(H_j)$, for $j=3-i$.
By Claim~\ref{doublesplit:cl2}, $|\zB_{H_j}(X) \cap \{v_j,w_j\}|=1$ 
and $|\zI_{H_j}(Y) \cap \{v_j,w_j\}|=1$.
Thus we may assume that $v_1 \in \zB_{H_1}(X)$, $w_1 \in \zI_{H_1}(X)$, $v_2 \in \zB_{H_2}(X)$ and $w_2 \in \zI_{H_2}(Y)$.
Now suppose that there exists $X' \in \bS$ such that $w_1 \in \zB_{H_1}(X')$.
Let $Y':=\bar{X}' - \loops(H_1)$.
As $w_1\in \zI_{H_1}(X)$, $X$ is not contained in $X'$ and $X'$ is not disjoint from $X$.
As $\bS$ is non-crossing, by possibly swapping $X'$ with $Y'$,
we may assume that $X' \subset X$. Thus $v_1 \notin \zI_{H_1}(X')$.
Moreover, as $w_2 \in \zI_{H_2}(Y)$ and $Y \subset Y'$, we have $w_2 \in \zI_{H_2}(Y')$.
Therefore, by the choice of $\bS$, $v_2 \in \zB_{H_2}(X')$.
Hence $X'$ violates Claim~\ref{doublesplit:cl2}.
This shows that for every $X \in \bS$, $w_1 \notin \zB_{H_1}(X)$.
By symmetry between $H_1$ and $H_2$, for every $X \in \bS$, $w_2 \notin \zB_{H_2}(X)$.
Moreover, as $\zB_{H_i}(X) \cap \{v_i,w_i\} \neq \emptyset$, for $i=1,2$,
we have $v_i \in \zB_{H_i}(X)$ for every $X \in \bS$ and $i=1,2$.
Lemma~\ref{conn:makestar} implies that there exists a w-sequence $\bS'$ of $H_1$
with $H_2=\Wflip[H_1,\bS']$ and that $\bS'$ is a star of $H_i$ with center $v_i$, for $i=1,2$.
Let $\bS'=(Y_1,\ldots,Y_h)$. For distinct $Y,Y' \in \bS'$, $Y$ and $Y'$ are disjoint.
It follows that if $h \geq 3$, then for some $Y \in \bS$, $w_i \notin \zI_{H_i}(Y)$, for $i=1,2$. 
Hence $\bar{Y}-\loops(H_1)$ contradicts Lemma~\ref{doublesplit:cl1}.
Therefore $h=1$ or $h=2$ and $(H_1,v_1,w_1,\alpha_1,\beta_1,H_2,v_1,w_2,\alpha_2,\beta_2,\bS')$ 
is a quad-template of type II, as required.
\end{proof}

\begin{proof} [\textbf{Proof of Lemma~\ref{doublesplit:le0}}]
Suppose that $H_1 \setminus \loops(H_1)$ is not $2$-connected. This is equivalent
to $H_2 \setminus \loops(H_2)$ not being $2$-connectred, as $H_1$ and $H_2$ are equivalent.
For $i=1,2$, let $\tau_i$ be the tree of blocks of $H_i \setminus \loops(H_i)$.
So the vertices of $\tau_i$ are partitioned into sets $A_i$ and $\bfB_i$, 
where $A_i$ is the set of the cut-vertices and $\bfB_i$ is the set of blocks of $H_i \setminus \loops(H_i)$.
Note that, as $H_1$ and $H_2$ are equivalent, there is a bijection between the vertices in $\bfB_1$ and the vertices in $\bfB_2$.
By Lemma~\ref{connectivity:n3c}(2), for $i=1,2$, $G_i \setminus \loops(G_i)$ does not contain $1$-separations.
Thus $A_i \subseteq \{v_i,w_i\}$, for $i=1,2$.
In particular this implies that at most one vertex in $\bfB_i$ is not a leaf of $\tau_i$, for $i=1,2$.
Hence there exists $X \in \bfB_1$ which is a leaf of both $\tau_1$ and $\tau_2$.
By symmetry between $v_1$ and $w_1$, we may assume that $\zB_{H_1}(X)=\{v_1\}$. Similarly we may assume that $\zB_{H_2}(X)=\{v_2\}$.
Note that $|X| \geq 2$, as otherwise $X$ would be a bridge of $G_1$.

If, for $i=1$ or $i=2$, $w_i \in V_{H_i}(Y)$, for $Y=X$ or $Y=E(H_1)-(X \cup \loops(H_1))$,
we derive a contradiction by Lemma~\ref{doublesplit:cl1}.
Therefore, by symmetry between $H_1$ and $H_2$, we may assume that $w_1 \in \zI_{H_1}(X)$ and
$w_2 \notin V_{H_2}(X)$.
\begin{claim}\label{doublesplit:clBoh}
$H_1[X]=H_2[X]$.
\end{claim}
\begin{cproof}
As $H_1$ and $H_2$ are equivalent and $H_1[X]$ and $H_2[X]$ are $2$-connected,
by Lemma~\ref{conn:star2} there exists a graph $H$ such that:
\begin{itemize}
	\item $H=\Wflip[H_1[X],\bS_1]$ for some w-sequence $\bS_1$, where $v_1,w_1 \notin \zB_{H_1}(Y)$ for all $Y \in \bS_1$, and
	\item $H_2[X]=\Wflip[H,\bS_2]$ for some non-crossing w-sequence $\bS_2$ such that $\zB_{H_1}(Y) \cap \{v_1,w_1\} \neq \emptyset$, 
	for all $Y \in \bS_2$.
\end{itemize}
Suppose that $\bS_1=(Y_1,\ldots,Y_k)$. Then either $Y_1$ or $X-Y_1$ is a $2$-separation in $H_1$ and
$(\Wflip[H_1,Y_1]$, $v_1, w_1, \alpha_1, \beta_1,H_2,v_2,w_2,\alpha_2,\beta_2)$ is a quad-template which
is compatible with $\dst$. By Lemma~\ref{doublesplit:le1}, proving the statement for a compatible
quad-template leads to a proof for the original template.
Thus, by repeating this reasoning on $Y_2,\ldots,Y_k$, we may assume that $\bS_1 = \emptyset$.
Therefore $H_2[X]=\Wflip[H_1[X],\bS]$ for a non-crossing w-sequence $\bS$, where for every $Y \in \bS$, 
$\zB_{H_1}(Y) \cap \{v_1,w_1\} \neq \emptyset$.
Consider $Y \in \bS$. If $v_2 \notin \zB_{H_2}(Y)$, then either $Y$ or $X-Y$ is a $2$-separation of $H_2$
and $(H_1,v_1,w_1,\alpha_1,\beta_1,\Wflip[H_2,Y]$, $v_2$, $w_2,\alpha_2,\beta_2)$ is a quad-template which
is compatible with $\dst$. Thus we may assume that $v_2 \in \zB_{H_2}(Y)$, for every $Y \in \bS$.
In particular, this implies that, for every $Y \in \bS$, both $Y$ and $X-Y$ are $2$-separations in $H_2$.
As $w_2 \notin H_2[X]$, we have $w_2 \notin V_{H_2}(Y),V_{H_2}(X-Y)$, for every $Y \in \bS$.
Note that we may assume that, for every $Y \in \bS$, $\zI_{H_i}(Y), \zI_{H_i}(X-Y) \neq \emptyset$, for $i=1,2$,
otherwise the Whitney-flip on $Y$ is trivial and may be omitted.
As $\zB_{H_1}(Y) \cap \{v_1,w_1\} \neq \emptyset$ and $v_1,w_1 \in V_{H_1}(X)$,
either $v_1,w_1 \in V_{H_1}(Y)$ or $v_1,w_1 \in V_{H_1}(X-Y)$.
By Lemma~\ref{doublesplit:cl1}, $v_1,w_1 \in \zB_{H_1}(Y)$ and all the sets 
$\alpha_1 \cap Y$, $(\delta_{H_1}(v_1)-\alpha_1)\cap Y$, $\beta_1 \cap Y$, and $(\delta_{H_1}(w_1)-\beta_1) \cap Y$
are non-empty.
Fix a minimal $Y \in \bS$. We may assume that no edge $(v_1,w_1)$ is in $Y$. 
Either $\alpha_2 \cap Y \subseteq \delta_{H_1}(v_1)$ or
$\alpha_2 \cap Y \subseteq \delta_{H_1}(w_1)$.
In the first case, define $\varphi_1 = (\alpha_1 \Delta \alpha_2) \cap Y$ and $\varphi_2 = \beta_1 \cap Y$.
In the second case, define $\varphi_1=\alpha_1 \cap Y$ and $\varphi_2=(\beta_1 \Delta \alpha_2) \cap Y$.
In both cases, $\varphi_1 \subseteq \delta_{H_1}(v_1)$ and $\varphi_2 \subseteq \delta_{H_1}(w_1)$.
By definition of quad siblings, $\alpha_1 \Delta \beta_1 \Delta \alpha_2 \Delta \beta_2$
is a cut of $H_1$. As $\beta_2 \cap Y = \emptyset$, this implies that 
$C:=(\alpha_1 \Delta \beta_1 \Delta \alpha_2) \cap Y$ is a cut of $H_1[Y]$.
As all the sets 
$\alpha_1 \cap Y$, $(\delta_{H_1}(v_1)-\alpha_1)\cap Y$, $\beta_1 \cap Y$, and $(\delta_{H_1}(w_1)-\beta_1) \cap Y$
are non-empty and there is no edge with ends $v_1$ and $w_1$ in $Y$, $C$ is a non-empty cut.
Moreover, $C \neq \delta_{H_1}(w_1)$ and $\varphi_1 \cap \varphi_2 = \emptyset$.
By Lemma~\ref{doublesplit:rem4},
there exists $Z \subseteq Y$ such that the following hold:
\begin{itemize}
	\item $\zB_{H_1}(Z) \subseteq \{v_1,w_1\}$;
	\item $\zI_{H_1}(Z) \neq \emptyset$;
	\item $\delta_{H_1}(v_1)\cap Z=\varphi_1$;
	\item for $\hat{\varphi}_2 = \varphi_2$ or $\hat{\varphi}_2=\varphi_2 \Delta \delta_{H_1}(w_1)$,
		$\delta_{H_1}(w_1)\cap Z=\hat{\varphi}_2$.
\end{itemize}
Therefore, for $i=1,2$, $(G_i[Z],\Sigma_i \cap Z)$ is bipartite and $Z$ is a $k_i$-separation of $G_i$,
where $k_i \leq 3$. By Lemma~\ref{connectivity:n3c}, $k_1=k_2=3$ and
$Z$ is a $3$-$(0,1)$-separation in both $(G_1,\Sigma_1)$ and $(G_2,\Sigma_2)$.
By Lemma~\ref{doublesplit:clean}, $(G_1,\Sigma_1)$ and $(G_2,\Sigma_2)$ are $\Delta$-reducible, a contradiction.
We conclude that $\bS = \emptyset$ and $H_1[X]=H_2[X]$.
\end{cproof}
As $X$ is a leaf of $\tau_1$ and $w_1 \in \zI_{H_1}(X)$,
no block of $H_1 \setminus \loops(H_1)$ has as boundary $\{w_1\}$.
Thus, for every $Y \in \bfB_1$, $\zB_{H_1}(Y)=\{v_1\}$.
Suppose that, for some $Y \in \bfB_2$, $\zB_{H_2}(Y)=\{w_2\}$. Thus $v_2 \not \in V_{H_2}(Y)$,
$\zB_{H_1}(Y)=\{v_1\}$ and $w_1 \notin V_{H_1}(Y)$, contradicting Lemma~\ref{doublesplit:cl1}.
It follows that, for every $Y \in \bfB_i$, $\zB_{H_i}(Y)=\{v_i\}$, for $i=1,2$.
If $|\bfB_1|\geq 3$, then for some $Y \in \bfB_1$, $w_i \notin \zI_{H_i}(Y)$, for $i=1,2$,
contradicting Lemma~\ref{doublesplit:cl1}. 
Thus $\bfB_1=\{X,Y\}$ for some set $Y$, where
$w_1 \in \zI_{H_1}(X)$, $w_2 \in \zI_{H_2}(Y)$ and $\zB_{H_i}(X)=\{v_i\}$, for $i=1,2$.
By Claim~\ref{doublesplit:clBoh}, $H_1[X]=H_2[X]$. By symmetry between $H_1$ and $H_2$, we also have $H_1[Y]=H_2[Y]$.
In particular this implies that $w_2$ is a vertex of $H_1$
and $H_2 \setminus \loops(H_2)$ is obtained by identifying a vertex $x \in V(H_1[X])$ with a vertex $y \in V(H_1[Y])$.
Define paths $P_x$ and $P_y$ as follows.
If $x=v_1$, let $P_x$ be a $(w_1,v_1)$-path in $H_1[X]$,
otherwise let $P_x$ be an $(x,v_1)$-path in $H_1[X]$.
If $y=v_1$, let $P_y$ be a $(w_2,v_1)$-path in $H_1[Y]$,
otherwise let $P_y$ be a $(y,v_1)$-path in $H_1[Y]$.
It follows that $P_x$ and $P_y$ are non-empty and $P:=P_x \cup P_y$ is a path of $H_1$.
As $x$ is an end of $P_x$ and $y$ is an end of $P_y$, $P$ is also a path of $H_2$.
For $i=1,2$, construct a graph $H_i'$ by adding to $H_i$ an edge $\Omega$ with ends
the ends of $P$ in $H_i$. Note that $H_1'$ is now $2$-connected, except for the possible presence of loops.
We show that $H_1'$ and $H_2'$ are equivalent by showing that they have the same cycles.
By construction, $P \cup \Omega$ is a cycle in both $H_1', H_2'$.
Let $C$ be a cycle of $H_1'$. If $\Omega \notin C$, $C$ is a cycle of $H_1$ and $H_2$ and we are done.
If $\Omega \in C$, then $C':=C \Delta (P \cup \Omega)$ is a cycle of $H_1'$ not using $\Omega$,
hence it is a cycle of $H_2'$. It follows that $C=C' \Delta (P \cup \Omega)$ is a cycle of $H_2'$.
We conclude that $H_1',H_2'$ are equivalent.
Define a w-sequence for $H_1'$ as follows:
\[
\bS:=
\begin{cases} 
\emptyset & \text{if $x=v_1$ and $y=v_1$}\\
(X) & \text{if $x\neq v_1$ and $y=v_1$}\\
(Y) & \text{if $x = v_1$ and $y\neq v_1$}\\
(X,Y) & \text{if $x\neq v_1$ and $y\neq v_1$.}
\end{cases} 
\]

Then $H_2'=\Wflip[H_1',\bS]$.
For $i=1,2$, if $P$ is $(\alpha_i \Delta \beta_i)$-even, define $\alpha_i':=\alpha_i$,
otherwise set $\alpha_i':=\alpha_i \Delta \delta_{H_i}(\zI_{H_i}(Y))$.
With this choice, $P \cup \Omega$ is an $(\alpha_i'\Delta \beta_i)$-even cycle in $H_i'$, for $i=1,2$.
Therefore $(H_1',\alpha_1'\Delta \beta_1)$ and $(H_2',\alpha_2'\Delta \beta_2)$ have the same even cycles.
Moreover, $\alpha_i \subseteq \delta_{H_i'}(v_i)$.
It follows that $\dst':=(H_1',v_1,w_1,\alpha_1',\beta_1,H_2',v_2,w_2,\alpha_2',\beta_2,\bS)$
is a quad-template. 
Moreover $\dst'$ is of type I if $\bS = \emptyset$ and of type II in the other three cases.
Let $\dst'':=(H_1,v_1,w_1,\alpha_1',\beta_1,H_2,v_2,w_2,\alpha_2',\beta_2)$.
Then $\dst''$ and $\dst$ are compatible quad-templates.
Let $(G_1',\Sigma_1')$ and $(G_2',\Sigma_2')$ (respectively $(G_1'',\Sigma_1'')$ and $(G_2'',\Sigma_2'')$) be the quad
siblings arising from $\dst'$ (respectively $\dst''$). 
By Lemma~\ref{doublesplit:le2} and Lemma~\ref{doublesplit:le3}, 
$(G_1',\Sigma_1')$ and $(G_2',\Sigma_2')$ are either shuffle, tilt, twist, widget or gadget siblings.
For $i=1,2$, $(G_i'',\Sigma_i'')=(G_i',\Sigma_i') \setminus \Omega$, therefore
$(G_1'',\Sigma_1'')$ and $(G_2'',\Sigma_2'')$ are either shuffle, tilt, twist, widget or gadget siblings.
As $\dst$ and $\dst''$ are compatible, the statement follows by Lemma~\ref{doublesplit:le1}.
\end{proof}


\begin{thebibliography}{99}

\bibitem{Gerards00} 
A.M.H. Gerards, 
{\em A few comments on isomorphism of even cycle spaces}, 
unpublished manuscript.

\bibitem{GPW241}
B. Guenin, I. Pivotto, P. Wollan,
{\em Relation between pairs of representations of signed binary matroids},
accepted by the SIAM J. Discrete Math.

\bibitem{GPWbp}
B. Guenin, I. Pivotto, P. Wollan, 
{\em Displaying blocking pairs in signed graphs},
in preparation.

\bibitem{NT}
S. Norine, R. Thomas,
personal communication.

\bibitem{Oxley92} 
J. Oxley, 
{\em Matroid Theory}, 
Oxford University Press, (1992).

\bibitem{OSW}
J. Oxley, C. Semple, G. Whittle, 
{\em The structure of 3-separations of 3-connected matroids}, 
J. Combin. Theory Ser. B \textbf{92} (2004), 257--293.

\bibitem{Shih82} 
C. H. Shih, 
{\em On graphic subspaces of graphic spaces}, 
Ph.D. dissertation, Ohio State Univ. (1982).

\bibitem{Tutte66} 
W. T. Tutte, 
{\em Connectivity in graphs}, 
University of Toronto Press, Toronto, (1966).

\bibitem{Whitney33} 
H. Whitney, 
{\em 2-isomorphic graphs}, 
Amer. J. Math. \textbf{55}  (1933), 245--254.

\end{thebibliography}
\end{document}